\documentclass[10pt]{amsart}

\usepackage{latexsym,amsfonts,amssymb,exscale,comment}
\usepackage{amsmath,amsthm,amscd}
\usepackage{verbatim}

\usepackage{setspace}
\usepackage{enumerate}

\usepackage{fullpage}

\usepackage{caption}
\usepackage{subcaption}

\usepackage{url}
\usepackage[bookmarks=true,%
    colorlinks=true,%
    linkcolor=blue,%
    citecolor=blue,%
    filecolor=blue,%
    menucolor=blue,%
    urlcolor=blue,%
    breaklinks=true,
		bookmarksdepth=2]{hyperref}

\usepackage{bm}

\input xy
\usepackage[all]{xy}
\SelectTips{cm}{}

\usepackage{tikz}

\tikzset{->-/.style={decoration={
  markings,
  mark=at position #1 with {\arrow{>}}},postaction={decorate}}}

\tikzset{middlearrow/.style={
        decoration={markings,
            mark= at position 0.5 with {\arrow{#1}} ,
        },
        postaction={decorate}
    }
}

\usepackage{graphicx}
\usepackage{color}

\theoremstyle{plain}
\newtheorem{theorem}{Theorem}

\newtheorem{corollary}[theorem]{Corollary}
\newtheorem{proposition}[theorem]{Proposition}
\newtheorem{lemma}[theorem]{Lemma}

\theoremstyle{definition}

\newtheorem{definition}[theorem]{Definition}
\newtheorem{conjecture}[theorem]{Conjecture}

\theoremstyle{definition}
\newtheorem{remark}[theorem]{Remark}

\numberwithin{equation}{section}
\numberwithin{theorem}{section}

\newcommand{\refequal}[1]{\xy {\ar@{=}^{#1}
(-1,0)*{};(1,0)*{}};
\endxy}

\newcommand{\Hom}{{\rm Hom}}

\newcommand{\maps}{\colon}
\renewcommand{\to}{\rightarrow}

\def\1{\mathbf{1}}%
\def\id{\mathrm{id}}

\def\mf{\mathfrak}

\numberwithin{equation}{section}

\let\hat=\widehat
\let\tilde=\widetilde

\let\epsilon=\varepsilon

\usepackage{bbm}
\def\C{{\mathbb{C}}}

\def\cal#1{\mathcal{#1}}%
\def\cal#1{\mathcal{#1}}

\newcommand\nc{\newcommand}
\nc\rnc{\renewcommand}
\nc\Kar{\operatorname{Kar}}
\nc\End{\operatorname{End}}

\nc\Sym{\operatorname{Sym}}

\allowdisplaybreaks

\DeclareMathOperator{\alt}{alt}

\DeclareMathOperator{\del}{del}

\DeclareMathOperator{\full}{full}

\DeclareMathOperator{\Grr}{Gr}

\DeclareMathOperator{\OSz}{OSz}

\DeclareMathOperator{\proj}{proj}

\DeclareMathOperator{\rest}{rest}
\DeclareMathOperator{\sign}{sign}

\DeclareMathOperator{\var}{var}

\newcommand{\At}{\tilde{A}}
\newcommand{\Bt}{\tilde{B}}

\newcommand{\A}{\mathcal{A}}
\newcommand{\B}{\mathcal{B}}
\newcommand{\V}{\mathcal{V}}
\newcommand{\Eb}{\mathbf{E}}

\newcommand{\Fb}{\mathbf{F}}

\newcommand{\gl}{\mathfrak{gl}}

\newcommand{\M}{\mathfrak{M}}
\newcommand{\Q}{\mathbb{Q}}
\newcommand{\R}{\mathbb{R}}

\newcommand{\tf}{\mathfrak{t}}

\newcommand{\Z}{\mathbb{Z}}
\newcommand{\Zc}{\mathcal{Z}}

\newcommand{\x}{\mathbf{x}}
\newcommand{\y}{\mathbf{y}}
\newcommand{\z}{\mathbf{z}}

\newcommand{\Ib}{\mathbf{I}}

\newcommand{\Del}{\mathbf{Del}}
\newcommand{\Rest}{\mathbf{Rest}}

\title{From hypertoric geometry to bordered Floer homology via the $m=1$ amplituhedron}

\begin{document}

\author{Aaron D. Lauda}
\email{lauda@usc.edu}
\address{Department of Mathematics\\ University of Southern California \\ Los Angeles, CA}
\thanks{Research was sponsored by the Army Research Office and was
accomplished under Grant Number W911NF-20-1-0075. The views and conclusions contained in this
document are those of the authors and should not be interpreted as representing the official policies, either
expressed or implied, of the Army Research Office or the U.S. Government. The U.S. Government is
W911NF2010075 authorized to reproduce and distribute reprints for Government purposes notwithstanding any copyright
notation herein.}

\author{Anthony M. Licata}
\email{anthony.licata@anu.edu.au}
\address{Mathematical Sciences Institute\\ Australian National University \\ Canberra, Australia}

\author{Andrew Manion}
\email{amanion@usc.edu}
\address{Department of Mathematics\\ University of Southern California \\ Los Angeles, CA}

\begin{abstract}
We give a conjectural algebraic description of the Fukaya category of a complexified hyperplane complement, using the algebras defined in \cite{BLPPW} from the equivariant cohomology of toric varieties.  We prove this conjecture for cyclic arrangements by showing that these algebras are isomorphic to algebras appearing in work of Ozsv{\'a}th--Szab{\'o} \cite{OSzNew} in bordered Heegaard Floer homology \cite{LOTBorderedOrig}.  The proof of our conjecture in the cyclic case extends work of Karp--Williams \cite{KarpWilliams} on sign variation and the combinatorics of the $m=1$ amplituhedron.  We then use the algebras associated to cyclic arrangements to construct categorical actions of $\gl(1|1)$.
\end{abstract}

\maketitle

\section{Introduction}

\subsection{Generalities}

Let $\cal{V}$ be an arrangement of real affine hyperplanes in real Euclidean space, and let $\cal X_\cal{V}$ denote the complexified hyperplane complement.  This paper investigates algebraic and combinatorial structures which arise in the symplectic geometry of $\cal X_\cal{V}$.  We give a conjectural presentation of a wrapped Fukaya category of $\cal X_\cal{V}$ as the module category of a convolution algebra $\tilde{B}(\cal{V})$  defined and studied in \cite{Gale,BLPPW,HypertoricCatO} in their work on geometric representation theory and symplectic duality.

We prove our conjecture in an interesting special case, namely when the arrangement $\cal{V}$ is cyclic, by identifying $\tilde{B}(\cal{V})$ with a known description of the Fukaya category in this case \cite{OSzNew, LP, MMW2, Auroux}.
Cyclic arrangements are of independent interest for at least three reasons, each of which we discuss later in this section:
\begin{itemize}
\item The complexified cyclic hyperplane complement $\cal X_\cal{V}$ is isomorphic to a symmetric product of the punctured plane $\Sym^k(\C - \{z_1, \ldots, z_n\})$.  This connects Fukaya categories of complexified hyperplane complements in the cyclic case to algebras appearing in Floer theory.
\item  The union of the compact regions in a cyclic arrangement is the $m=1$ amplituhedron introduced by Arkani-Hamed and Trnka.  This suggests that hypertoric varieties for cyclic arrangements should be a source of ``positive geometries" relevant for scattering amplitudes in gauge theories.
\item The hypertoric categories $\cal O$ associated to cyclic arrangements are the weight space categories for a categorified action of (a variant of) the quantum supergroup $U_q(\mf{gl}(1|1))$ on its representation $V^{\otimes n}$, with $V$ the vector representation.
 \end{itemize}

Our conjectural description of the Fukaya category of $\cal X_\cal{V}$ is closely connected to the combinatorial geometry of the hypertoric variety\footnote{These varieties are also commonly referred to elsewhere in the literature as toric hyperk{\"a}hler manifolds or toric hyperk{\"a}hler varieties.} $\cal M_\cal{V}$ associated to the real arrangement $\cal{V}$.  In the last decade, hypertoric varieties have appeared prominently in investigations of symplectic duality, a mathematical incarnation of 3d mirror symmetry from physics \cite{IS}, in part because the mirror dual of a hypertoric variety is also hypertoric.  This makes hypertoric varieties useful as testing grounds for more general 3d mirror symmetry expectations.  One such expectation, which comes from the work of the second author with Braden, Proudfoot and Webster \cite{BLPW-II,Gale,HypertoricCatO}, is a relationship between symplectic duality and Koszul duality.  This expectation has been established in the case of hypertoric varieties, as the hypertoric categories $\cal O$ associated to symplectic dual hypertoric varieties are Koszul dual.  Moreover, the finite-dimensional Koszul algebra $B(\cal{V})$ governing hypertoric category $\cal O$ conjecturally arises as the endomorphism algebra of a canonical Lagrangian in a Fukaya category of the hypertoric variety $\cal M_\cal{V}$.  

The hypertoric variety $\cal M_\cal{V}$ and the complexified complement $\cal X_\cal{V}$ are related via a moment map for a torus action.  Thus it seems reasonable to expect that a Fukaya category of the complexified complement $\cal X_\cal{V}$ is also governed by an algebra related to $B(\cal{V})$.  Indeed, there is a ``universal deformation" $\tilde{B}(\cal{V})$ of the Koszul algebra $B(\cal{V})$, which is related to the algebra $B(\cal{V})$ much the way torus-equivariant cohomology is related to ordinary cohomology.  It is tempting to speculate that the universal deformation $\tilde{B}(\cal{V})$ governs a ``torus-equivariant Fukaya category" of a hypertoric variety, but it seems challenging to make this precise.  Our main conjecture is somewhat simpler:  the deformation $\tilde{B}(\cal{V})$ describes a wrapped Fukaya category of the complexified hyperplane complement $\cal X_\cal{V}$.

\begin{conjecture}\label{conj:symp} The algebra $\tilde{B}(\cal{V})$ is quasi-isomorphic to the endomorphism algebra of a canonical generating Lagrangian in a wrapped Fukaya category of $\cal{X}_\cal{V}$, where we take full wrapping around the hyperplanes of the arrangement and an appropriate partial wrapping at infinity.
\end{conjecture}

\subsection{Cyclic arrangements}
The first main result of this paper is a proof of Conjecture \ref{conj:symp} in the special case of cyclic hyperplane arrangements, where the partial wrapping at infinity is specified as in Remark~\ref{rem:AurouxStops}. The special feature of cyclic arrangements we use in the proof comes from the first bullet point above:  the complexified hyperplane complement is isomorphic to a symmetric product of a multiply-punctured plane (or disc).  This is crucial to our proof of Conjecture \ref{conj:symp} in the cyclic case, as it connects Fukaya categories of complexified hyperplane complements to algebras appearing in Heegaard Floer theory.  In particular, by \cite{Auroux} the wrapped Fukaya categories of these symmetric products underlie bordered Floer homology~\cite{LOTBorderedOrig}, the extended TQFT approach to Heegaard Floer homology.  When the surface is a multiply-punctured disc, the results of \cite{LP, MMW2} together with those of \cite{Auroux} imply that the Fukaya category for a certain sutured structure on the surface (determining stops for the wrapping) is described by an algebra used recently by Ozsv{\'a}th--Szab{\'o} for very fast knot Floer homology (HFK) computations in their theory of bordered HFK~\cite{OSzNew,HFKcalc}.   Conjecture~\ref{conj:symp} for cyclic arrangements is then implied by the following theorem, which we prove in Section~\ref{sec:OSz}.

\begin{theorem}\label{thm:iso}[cf. Theorem~\ref{thm:LeftCyclicIsom}, Corollary~\ref{cor:MainUnpolarizedThm}, Theorem~\ref{thm:RightCyclicIsom}]
The universal deformation $\tilde{B}(\cal{V})$ associated to a cyclic hyperplane arrangement is isomorphic to the Ozsv{\'a}th--Szab{\'o} algebra associated to a symmetric product of a multiply-punctured disc.
\end{theorem}

The proof of Theorem~\ref{thm:iso} makes contact with the second bullet point above.  In particular, the proof builds on work of Karp--Williams~\cite{KarpWilliams}, who relate the geometry of the $m=1$ amplituhedron and the positive Grassmannian to the combinatorics of sign sequences.  The combinatorial theory we develop in service of the proof of Theorem~\ref{thm:iso} is perhaps of independent interest; in particular, we formulate a ``positive Gale duality" in linear programming, by modifying the usual Gale duality in such a way that it preserves positivity by exchanging left and right cyclic polarized arrangements (see Section~\ref{sec:Cyclic}).
We also show that the cyclic arrangements associated to the ``dual" symmetric products $\Sym^k(\C - \{z_1, \ldots, z_n\})$ and $\Sym^{n-k}(\C - \{z_1, \ldots, z_n\})$ are positive Gale duals; see Theorem~\ref{thm:AltGale}. 

\subsection{Bordered Floer homology and the representation theory of $\mf{gl}(1|1)$}

The connection between Fukaya categories of symmetric products and convolution algebras associated to cyclic arrangements is interesting in both directions.
In \cite{LaudaManion}, the first and third authors used the algebras of Ozsv{\'a}th--Szab{\'o} to construct categorical representations of $\gl(1|1)$; these constructions are a small part of a larger program of the third author and Rouquier to develop the foundations of the higher representation theory of $\gl(1|1)$  (see also \cite{ManionDecat}, \cite{EPV}, and \cite{TianUq,TianUT} for closely related work).  Thus, in addition to being basic ingredients to Heegaard Floer homology, the Ozsv{\'a}th--Szab{\'o} algebras are also basic objects in $\gl(1|1)$ representation theory, and
Theorem~\ref{thm:iso} implies several interesting facts about these algebras.  For example, it follows from Theorem~\ref{thm:iso} that the center of Ozsv{\'a}th--Szab{\'o}'s algebra is isomorphic to the torus-equivariant cohomology of the associated hypertoric variety, and the endomorphism algebras of projective modules over Ozsv{\'a}th--Szab{\'o}'s algebra are isomorphic to the equivariant cohomologies of toric varieties; see Corollary~\ref{cor:center}.

In the companion paper \cite{LLM2} we prove that for general arrangements, the universal deformations $\tilde{B}(\cal{V})$ are affine highest weight categories, a notion recently introduced by Kleshchev~\cite{Klesh-affine} in order to extend ideas from finite-dimensional quasi-hereditary theory to infinite-dimensional settings.  We also prove that these algebras are homologically smooth.  This homological smoothness is relevant for Conjecture~\ref{conj:symp}, as one expects that the wrapped Fukaya categories of complexified hyperplane complements are equivalent to (dg or $A_\infty$-) module categories of homologically smooth algebras.

In the cyclic case, Theorem \ref{thm:iso}, together with \cite{LLM2}, therefore gives the following.

\begin{corollary}
The wrapped Fukaya categories of  $\Sym^k(\C - \{z_1, \ldots, z_n\})$ are affine highest weight categories.
\end{corollary}

This corollary is useful as a tool to further understand the categorification of bases for the $U_q(\gl(1|1))$ representations initiated in \cite{LaudaManion}.  First, the affine quasi-hereditary structure of $\Bt(\cal{V})$ includes as part of the structure a family of standard modules over Ozsv{\'a}th--Szab{\'o}'s algebras categorifying the standard tensor-product basis of $V^{\otimes n}$.  There are also natural classes of modules categorifying the canonical and dual canonical basis of $V^{\otimes n}$.  As a result we get a geometric construction of canonical bases for $V^{\otimes n}$, wherein each canonical basis element corresponds to an irreducible component of the relative core of the hypertoric variety $\M_{\cal{V}}$.

Existing work relating Floer theory and $U_q(\gl(1|1))$, e.g. \cite{LaudaManion}, can also be reformulated in the language of hypertoric category $\cal O$. In particular, when reformulated in the language of cyclic hyperplane arrangements, the bimodule constructions of \cite{LaudaManion} are closely related to the basic hyperplane operations of deletion and restriction (see Section~\ref{sec:qg-bimodules}).  In Section \ref{sec:GeneralDeletionRestriction}, we generalize these bimodules from the cyclic case to the case of general arrangements.  We expect these deletion/restriction bimodules to be part of a larger functorial invariant of complexified hyperpane complements, a subject we hope to revisit later.

\subsection{Positive geometries, hypertoric varieties, and the amplituhedron} \label{intro:amp}

Cyclic hyperplane arrangements are connected by the results of \cite{KarpWilliams} to the $m=1$ amplituhedron defined by Arkani-Hamed and Trnka in their work on scattering theory~\cite{AHT14}. Amplituhedron geometry arises in physics as a tool to understand scattering amplitudes in gauge theories; these amplitudes exhibit symmetries and recursion relations in both realistic situations (e.g. gluon scattering in particle colliders) and in cases of theoretical interest (e.g. planar $\cal{N} = 4$ super Yang Mills). These symmetries are obscured if one wants locality and unitarity of the physics to be manifest, as it is with Feynman diagrams, suggesting the possibility of an alternate perspective on physics in which locality and unitarity emerge from more fundamental principles. The work of Arkani--Hamed and Trnka shows how the mathematics surrounding scattering amplitudes and their symmetries can be organized in such a way that locality and unitarity follow as consequences, not assumptions.

The different amplituhedra depend on integers $n, k, m$ as well as external scattering data.  The $m=4$ amplituhedra are most relevant for scattering amplitudes, but mathematically, $m=1$ is the more basic. In fact, recent work of Arkani-Hamed--Thomas--Trnka \cite{Binary} gives a characterization of higher-$m$ amplituhedra in terms of projections down to the $m=1$ amplituhedron, rederiving locality and unitarity as elementary consequences of basic features in the $m=1$ case.

These amplituhedra are also examples of ``positive geometry" \cite{PosGeom}. The crucial object needed for deriving physics from amplituhedra---and a crucial ingredient in the definition of the positive geometry coming from amplituhedra---is a top-degree differential form, which one integrates to get scattering amplitudes. For $k=1$ amplituhedra, which are cyclic polytopes, the amplituhedron form is a moment-map pushforward of a natural form on a toric variety having the cyclic polytope as its moment polytope.

It is an open problem to find analogous descriptions of other amplituhedra, that is, to describe their natural form as a pushforward of a form coming from some other positive geometry.  For $m=1$, this problem is likely related to the geometry of hypertoric varieties: the $m=1$ amplituhedron is a moment-map image of the core of the cyclic hypertoric variety.  Just as the $k=1$ amplituhedron form is a pushforward via a moment map of a form on a toric variety, the $m=1$ amplituhedron form should be the pushforward via a moment map of a form on the associated cyclic hypertoric variety.

\subsection{Outline of paper}
We briefly summarize the contents of this paper. In Section~\ref{sec:hyp-var-alg} we review the theory of hypertoric varieties, including their connection with polarized arrangements $\cal{V}$.  We review two algebras $\At(\cal{V})$ and $\Bt(\cal{V})$ naturally associated to a polarized arrangement.  Subsection ~\ref{sec:GeneralDeletionRestriction} introduces bimodules for these algebras associated with the hyperplane operations of deletion and restriction.  Section~\ref{sec:Cyclic} focuses on cyclic hyperplane arrangements;  we begin by extending the work of \cite{KarpWilliams} by developing the theory of (left/right) cyclic polarized arrangements in terms of subspaces in the positive Grassmannian.  We introduce left/right cyclic arrangements in Subsection~\ref{sec:AlternateCyclic} and formulate a positive variant of Gale duality that exchanges left and right cyclic arrangements in Subsection~\ref{sec:AltGale}.   In Subsection~\ref{sec:SymmetricPowers} we review the connection between the complexified complement of a cyclic arrangement and symmetric products of the punctured plane.  Subsection~\ref{sec:DelRestCyclic} shows that the deletion and sign-modified restriction operations for polarized arrangements preserve left/right cyclic arrangements. In Section~\ref{sec:OSz} we establish the connection between hypertoric convolution algebras associated to cyclic polarized arrangements and Ozsv{\'a}th--Szab{\'o} algebras appearing in the theory of bordered Heegaard Floer homology, proving Theorem~\ref{thm:iso}. As a consequence, we show in Subsection~\ref{sec:center} that the center of the Ozsv{\'a}th--Szab{\'o} algebra is isomorphic to the torus equivariant cohomology of the hypertoric variety $\mf{M}_{\cal{V}}$ of a (left/right) cyclic polarized arrangement.    Finally, Section~\ref{sec:qg-bimodules} connects the bimodules arising from deletion and restriction of cyclic polarized arrangements with an action of a variant of quantum $\mf{gl}(1|1)$ as bimodules over Ozsv{\'a}th--Szab{\'o} algebras;  we also discuss a Heegaard Floer interpretation of the factorization of these bimodules as the composition of deletion and restriction bimodules.

\subsection*{Acknowledgements} The authors are grateful to  Mina Aganagi{\'c}, Nima Arkani-Hamed, Sheel Ganatra, Sasha Kleshchev, Ciprian Manolescu, Nick Proudfoot, Raphael Rouquier, and Joshua Sussan  for helpful conversations.   A.D.L. was partially supported by NSF grant DMS-1664240, DMS-1902092 and Army Research Office W911NF2010075.
A.M.L. was supported by an Australian Research Council Future Fellowship.

\section{Hypertoric varieties and algebras} \label{sec:hyp-var-alg}

\subsection{Hyperplane arrangements and hypertoric varieties}
\subsubsection{Hyperplane arrangements}\label{sec:ArrangementDefs}
We use linear-algebraic data to specify real affine hyperplane arrangements, which we refer to as  arrangements, following the general framework used in \cite{Gale}.

\begin{definition}\label{def:VEta}
An \emph{arrangement} of $n$ hyperplanes in $k$-space is a pair $\cal{V} = (V,\eta)$ where $V \subset \R^n$ is a linear subspace of dimension $k$ and $\eta$ is an element of $\R^n / V$; we require that an element of $\R^n$ representing $\eta$ has at least $n-k$ nonzero entries\footnote{The condition on nonzero elements representing $\eta$ is sometimes omitted, in which case $(V,\eta)$ satisfying this condition is called \emph{simple} or \emph{regular}.}. We say that $\cal{V} = (V,\eta)$ is \emph{rational} if $V$ arises (uniquely) from a subspace $V_{\Q} \subset \Q^n$ and $\eta$ arises (uniquely) from an element $\eta_{\Q} \in \Q^n / V_{\Q}$.
\end{definition}

The intersections of $V$ with the coordinate hyperplanes of $\R^n$ give an honest affine hyperplane arrangement in the real vector space $V$ which we will denote $\cal{H}_{\cal{V}}$ (however, some hyperplanes of the arrangement might be empty). If $V$ is presented as the column span of an $n \times k$ matrix $A$, we can use the columns of the matrix as a basis for $V$ to identify $V$ (and thus its affine translate $V + \eta$) with $\R^k$; if $w \in \R^n$ represents $\eta$, then under this identification, the hyperplanes of the arrangement take the form
\[
H_i = \left\{ x \in \R^n \mid   w_i + \sum_{j=1}^k a_{ij} x_j = 0\right\}, \qquad \text{$1 \leq i \leq n$}
\]
for $x = (x_1,\ldots,x_k) \in \R^k$. The positive half-spaces of $\R^n$ induce a co-orientation on $\cal{H}_{\cal{V}}$, using which we can associate a region of the arrangement (possibly empty) to each length-$n$ sequence $\alpha$ of signs in $\{+,-\}$: given $\alpha = (\alpha_1,\ldots,\alpha_n)$ and  a matrix $A$ presenting $V$ as above, the corresponding region $\Delta_{\alpha} \subset V + \eta$ is the set of points $x$ such that
\[
\alpha_i \left(w_i + \sum_{j=1}^k a_{ij} x_j \right) \geq 0
\]
for $1 \leq i \leq n$. In what follows we sometimes write $\alpha(i)=\alpha_i$ to denote the $i^{th}$ term of the sequence $\alpha$.  If $\Delta_{\alpha}$ is nonempty, $\alpha$ is called a \emph{feasible} sign sequence. We let $\cal{F} = \cal{F}(\cal{V})$ denote the set of feasible sign sequences for $\cal{V}$, and we let $\cal{K} = \cal{K}(\cal{V}) \subset \cal{F}(\cal{V})$ denote the set of feasible sign sequences $\alpha$ such that $\Delta_{\alpha}$ is compact.

We define equivalence of arrangements by saying that $\cal{V} \sim \cal{V'}$ if they have the same affine oriented matroid as discussed in \cite[Chapter 4.5]{BLVSWZ}; if $\cal{V}$ and $\cal{V}'$ are equivalent, it follows that $\cal{F}(\cal{V}) = \cal{F}(\cal{V}')$ and $\cal{K}(\cal{V}) = \cal{K}(\cal{V}')$. Concretely, if we let $\phi$ be the unique linear functional on $V + \langle \eta \rangle$ such that $\phi(V) = 0$ and $\phi(\eta) = 1$, and define $\phi'$ similarly for $(V',\eta')$, then $(V,\eta) \sim (V',\eta')$ if and only if the Pl{\"u}cker coordinates of the $(k+1)$-dimensional subspaces $(\id,\phi)(V + \langle \eta \rangle)$ and $(\id,\phi')(V' + \langle \eta' \rangle)$ of $\R^{n+1}$ have the same (projective) signs. If $(A,w)$ and $(A',w')$ represent $(V,\eta)$ and $(V',\eta')$ as above, then $(V,\eta) \sim (V',\eta')$ if and only if the column spans of the matrices
\[
\left[
\begin{array}{c|c}
A & w \\
\hline
0 & 1
\end{array}
\right],
\left[
\begin{array}{c|c}
A' & w' \\
\hline
0 & 1
\end{array}
\right]
\]
have Pl{\"u}cker coordinates of the same (projective) signs.

\subsubsection{Hypertoric varieties} \label{sec:hyeprertoric}

Below we follow \cite{ProudfootThesis} with some minor expositional changes. Let $(V,\eta)$ be rational. Let $\mf t^n = (\C^n)^*$ with coordinate basis $\{\epsilon_1,\dots,\epsilon_n\}$. Let $\mf t^d = V^{\perp} \subset \mf t^n$, the complex perpendicular to $V$ ($\mf t^d$ has complex dimension $d := n-k$), and let $\mf t^k = \mf t^n / \mf t^d$. We have full integer lattices $\tf^n_{\Z} = (\Z^n)^* \subset \tf^n$, $\tf^d_{\Z} = \tf^d \cap \tf^n_{\Z} \subset \tf^d$, and $\tf^k_{\Z} =\tf^n_{\Z} / \tf^d_{\Z} \subset \tf^k$. There is an exact sequence of abelian Lie algebras
\[
	0 \rightarrow \mf t^d \xrightarrow{\iota} \mf t^n \rightarrow \mf t^k \rightarrow 0
\]
which exponentiates to an exact sequence of tori
\[
	1 \rightarrow T^d \rightarrow T^n \rightarrow T^k \rightarrow 1.
\]
The torus $T^n = (\C^\times)^n$ acts by coordinate-wise multiplication on $\C^n$, and we regard $T^d$ as acting on $\C^n$ via the inclusion of $T^d$ into $T^n$.  This, in turn, gives rise to a hamiltonian action of $T^d$ on the cotangent bundle $T^*(\C^n)$ via $t\dot (x,y) = (tx,t^{-1}y)$. The action of the maximal compact subtorus $T^d_{\R}$ of $T^d$ is hyperhamiltonian with hyperk{\"a}hler moment map given by
\[
\mu_{\R}(x,y) \oplus \mu_{\C}(x,y) = \left(\iota_{\R}^* \left( \frac{1}{2} \sum_i (|x_i|^2 - |y_i|^2) \epsilon_i \right), \iota^* \left(\sum_i (x_iy_i)  \epsilon_i \right) \right) \in (\tf^d_{\R})^* \oplus (\tf^d)^* = (\R^n / V) \oplus (\C^n / V_{\C}).
\]

The \emph{hypertoric variety} associated to $\cal{V}$ is defined to be
\[
\M_{\cal{V}} = \M_{V,\eta} = \left( \mu_{\R}^{-1}(\eta) \cap \mu_{\C}^{-1}(0) \right) / T^d_{\R};
\]
one can also define $\M_{\cal{V}}$ as an algebraic symplectic quotient.  These varieties (also called toric hyperk{\"a}hler manifolds or toric hyperk{\"a}hler varieties) were introduced in the smooth case by Goto \cite{Goto2}, unifying examples studied in \cite{EguchiHanson, GibbonsHawking, Calabi}, and in the singular case by Bielawski and Dancer \cite{BielawskiDancer}.

It follows from our assumptions in Definition~\ref{def:VEta} that $\M_{\cal{V}}$ is an algebraic symplectic orbifold of complex dimension $2k$; it is smooth if and only if $V \cap \Z^n$ is a unimodular lattice. There is a residual hyperhamiltonian action of the compact torus $T^k_{\R}$ on $\M_{\cal{V}}$ which can be extended to a hamiltonian action of the complex torus $T^k$.  We consider a variant of the hyperk{\"a}hler moment map for the $T^k_{\R}$ action defined by
\[
\bar{\mu}_{\R}([x,y]) \oplus \bar{\mu}_{\C}([x,y]) := \left( \frac{1}{2} \sum_i (|x_i|^2 - |y_i|^2) \epsilon_i, \sum_i (x_iy_i)  \epsilon_i \right) \in (V + \eta) \oplus V_{\C}.
\]
By taking appropriate linear combinations of $\bar{\mu}_{\R}$ and the real and imaginary parts of $\bar{\mu}_{\C}$, one can also construct variants of the moment map which take values in the complexification $(V + \eta)_{\C}$ of $V + \eta$.

Equivalent rational arrangements give rise to isomorphic hypertoric varieties, respecting the additional structure described below.

\subsubsection{Additional structure}

The hypertoric variety $\M = \M_{V,\eta}$ makes sense even when $\eta$ does not satisfy the simplicity condition of Definition~\ref{def:VEta}; in particular, we can consider $\M_0 = \M_{V,0}$, and for general $\eta$ there is a canonical morphism $\nu: \M \to \M_0$ that is a resolution of singularities when $\M$ is smooth. The action of $\mathbb{S} := \C^{\times}$ on $T^*(\C^n)$ by $s\dot(x,y) = (s^{-1}x, s^{-1}y)$ gives actions on $\M$ and $\M_0$ such that $\nu$ is equivariant; we have $s \cdot \omega = s^2 \omega$ where $\omega$ is the symplectic form on $\M$. The $\mathbb{S}$ action on $\C[\M_0]$ gives $\M_0$ the structure of an affine cone; it has nonnegative weights with zero weight space consisting of constant functions. Thus, when $\M$ is smooth, $\M \xrightarrow{\nu} \M_0$ is a conical symplectic resolution in the sense of \cite{BLPW-I}. In general, the map $\M \xrightarrow{\nu} \M_0$ with the $\mathbb{S}$ actions is invariant under equivalences of rational arrangements $\cal{V}$.

The subvariety
\[
\bar{\mu}_{\C}^{-1}(0) = \{ [x,y] \in \M : x_i y_i = 0 \textrm{ for all } i \}
\]
is called the \emph{extended core} of $\M$. The irreducible components of the extended core of $\M$ are $X_{\alpha}$ for $\alpha \in \cal{F}$, where
\[
X_{\alpha} = \{ [x,y] \in \M: y_i = 0 \textrm{ when } \alpha(i) = + \textrm{ and } x_i = 0 \textrm{ when } \alpha(i) = - \}.
\]
The varieties $X_{\alpha}$ can also be defined as toric varieties using the Cox construction; see \cite[Section 4.2]{Gale}. The image of $X_{\alpha}$ under $\bar{\mu}_{\R}$ is $\Delta_{\alpha} \subset V + \eta$. If we add to $\bar{\mu}_{\R}$ any complex linear combination of the real and imaginary parts of $\bar{\mu}_{\C}$, the image of $X_{\alpha}$ under this new map is still $\Delta_{\alpha} \subset V + \eta \subset (V + \eta)_{\C}$.

The \emph{core} of $\M_{V,\eta}$ is the union of $X_{\alpha}$ for $\alpha \in \cal{K}$; it can also be defined as $\nu^{-1}(0)$ where $0 \in \M_0$ is the cone point.

\subsection{Polarizations}

\subsubsection{Definitions}\label{sec:PolarizationDefs}

We now consider polarized arrangements, in which an affine hyperplane arrangement is equipped with an objective function as in linear programming. We recall the basic definitions here and refer to \cite{Gale} and \cite[Section 5]{HypertoricCatO} for further details.

\begin{definition}
A \emph{polarized arrangement} of $n$ hyperplanes in $k$-space is a triple $\cal{V} = (V,\eta,\xi)$ where $(V,\eta)$ is an arrangement as in Definition~\ref{def:VEta}
and $\xi$ is an element of $V^* = (\R^n)^* / V^{\perp}$ such that each element of $(\R^n)^*$ representing $\xi$ has at least $k$ nonzero entries.
We say that $(V,\eta,\xi)$ is \emph{rational} if $(V,\eta)$ is rational and
$\xi\in V_{\Q}^* = (\Q^n)^*/V_{\Q}^{\perp}$.
\end{definition}

If $V$ is presented as the column span of an $n \times k$ matrix $A$ as above, then in the basis of columns, $\xi$ is expressed as a $1 \times k$ matrix.  Thus we can specify the polarized arrangement $(V,\eta,\xi)$ (non-uniquely) as a single matrix:
\[
(V,\eta,\xi) \xrightarrow{V = \mathrm{col}(A)}
\left[
\begin{array}{c|c}
A & w \\
\hline
x^T & *
\end{array}
\right],
\]
where $A$ has size $n \times k$, $w$ has size $n \times 1$, $x$ has size $k \times 1$ with $x^T$ its transpose, and the bottom-right entry $*$ is left unspecified. From this data, $(V,\eta)$ is defined as above, and $\xi$ is defined to have matrix $x$ with respect to the columns of $A$. If we define a \emph{strong polarized arrangement} to be a polarized arrangement $(V,\eta,\xi)$ equipped with a lift of $\xi$ to $\bar{\xi} \in (V+\langle \eta \rangle)^*$, then we can specify a strong polarized arrangement by a matrix of the above form in which $*$ has been replaced by a real number $c$.

A strong polarized arrangement gives an affine hyperplane arrangement in the affine space $V + \eta$ with a well-defined objective function $\bar{\xi}$; from this data we can extract an oriented matroid program as in \cite[Chapter 10]{BLVSWZ}. For strong polarized arrangements, we say that $(V,\eta,\bar{\xi}) \sim (V',\eta',\bar{\xi}')$ if they have the same oriented matroid program. Concretely, if $(V,\eta,\bar{\xi})$ and $(V',\eta',\bar{\xi}')$ are represented by $(A,w,x,c)$ and $(A',w',x',c')$ as above, then $(V,\eta,\bar{\xi}) \sim (V',\eta',\bar{\xi}')$ if and only if the column spans of the matrices
\[
\left[
\begin{array}{c|c}
A & w \\
\hline
x^T & c \\
\hline
0 & 1
\end{array}
\right],
\left[
\begin{array}{c|c}
A' & w' \\
\hline
(x')^T & c' \\
\hline
0 & 1
\end{array}
\right]
\]
have Pl{\"u}cker coordinates of the same (projective) signs. For ordinary polarized arrangements, we say that $(V,\eta,\xi) \sim (V',\eta',\xi')$ if they have strong lifts that are equivalent.

Given a polarized arrangement $\cal{V} = (V,\eta,\xi)$, we say $\alpha \in \{+,-\}^n$ is \emph{bounded} if the affine-linear functional $\bar{\xi}$ on $V + \eta$ is bounded above on $\Delta_{\alpha}$ for some (equivalently, any) strong lift $\bar{\xi}$ of $\xi$. We let $\cal{B} = \cal{B}(\cal{V})$ be the set of $\alpha \in \{+,-\}^n$ that are bounded; we have $\cal{K} \subset \cal{B}$ where $\cal{F} = \cal{F}(\cal{V})$ and $\cal{K} = \cal{K}(\cal{V})$ are defined from $(V,\eta)$ as in Section~\ref{sec:ArrangementDefs}. We let $\cal{P} = \cal{F} \cap \cal{B}$ denote the set of \emph{bounded feasible regions}. The subsets $(\cal{F}, \cal{K}, \cal{B}, \cal{P})$ of $\{+,-\}^n$ are preserved under equivalence of polarized arrangements.

\subsubsection{Polarizations and hypertoric varieties}

Assume that $\cal{V} = (V,\eta,\xi)$ is rational and that $\xi \in \tf^k_{\Z}$. Exponentiating the map $\C \xrightarrow{ \cdot \xi} \tf^k$, we get a homomorphism $\mathbb{T} := \C^{\times} \to T^k$ and thus a hamiltonian action of $\mathbb{T}$ on $\M = \M_{V,\eta}$. This action commutes with the action of $\mathbb{S}$ and has finite fixed point set $\M^{\mathbb{T}}$, and it is preserved under equivalences of polarized arrangements. We will write $\M_{\cal{V}} = \M_{V,\eta,\xi}$ when we want to consider $\M_{V,\eta}$ equipped with this additional $\mathbb{T}$ action (together with $\M_0, \nu$, and the actions of $\mathbb{S}$).

For $\M = \M_{V,\eta,\xi}$, the \emph{relative core} $\M^+$ of $\M$ is the set of $p \in \M$ such that $\displaystyle \lim_{\mathbb{T} \ni t \to 0} t \cdot p$ exists; it is the union of the toric varieties $X_{\alpha}$ for the bounded feasible regions $\alpha \in \cal{P}$. The relative core contains the core and is contained in the extended core (see \cite[Section 4.2]{Gale}).

\subsubsection{Dualities}

A central feature of linear programming is the existence of a duality on linear programs, referred to as Gale duality in \cite{Gale}.

\begin{definition}
If $\cal{V} = (V,\eta,\xi)$ is a polarized arrangement, its \emph{Gale dual} is the polarized arrangement
\[
\cal{V}^{\vee} := (V^{\perp}, -\xi, -\eta).
\]
\end{definition}
Gale duality squares to the identity and preserves rationality and equivalence of polarized arrangements. A sequence $\alpha \in \{+,-\}^n$ is feasible for $\cal{V}$ if and only if it is bounded for $\cal{V}^{\vee}$ (and vice-versa). That is, $\cal{B}=\cal{F}^{\vee}$, $\cal{F}=\cal{B}^{\vee}$ and $\cal{P}=\cal{P}^{\vee}$.

\begin{remark}
As discussed in \cite{BLPW-II}, conical symplectic resolutions with $\mathbb{T}$ actions as above admit a duality known as symplectic duality, a mathematical incarnation of 3d mirror symmetry. Gale dual (rational) polarized arrangements give symplectically dual hypertoric varieties.
\end{remark}

The Pl{\"u}cker coordinates of $V$ and $V^{\perp}$ are indexed by the same set of $\binom{n}{k} = \binom{n}{n-k}$ elements, but they do not agree in general. However, for a subspace $V$, we can obtain a related subspace $\alt(V)$ by mapping $V$ through the automorphism of $\R^n$ that flips the sign of all even-index coordinates; note that $\alt(V^{\perp}) = (\alt (V))^{\perp}$. It is a standard result that the Pl{\"u}cker coordinates of $V$ and $\alt V^{\perp}$ do agree; thus, when discussing cyclic arrangements in Section~\ref{sec:Cyclic} below, it will be useful to consider alt-variants of Gale duality for polarized arrangements. Correspondingly, we define
\[
\alt(V,\eta,\xi) := (\alt(V), \alt(\eta), \alt(\xi))
\]
where $\alt(\eta)$ and $\alt(\xi)$ are obtained from any representatives of $\eta$ and $\xi$ by flipping the signs of all even-index coordinates.

Finally, we define the \emph{polarization reversal} of $\cal{V}=(V,\eta,\xi)$ to be $p(\cal{V}):=(V,\eta,-\xi)$; geometrically, polarization reversal precomposes the action of $\mathbb{T}$ with the automorphism $t \mapsto t^{-1}$ of $\mathbb{T}$. The polarization reversal operation will be important in Section~\ref{sec:AltGale}.

\subsubsection{Partial orders} \label{subsec:partial_order}
Let $\cal{V}$ be a polarized arrangement. Let $\mathbb{B}$ denote the set of $k$-element subsets $\mathbbm{x}$ of $I=\{1,\dots, n\}$  such that
\[
H_{\mathbbm{x}} = \bigcap_{i \in \mathbbm{x}} H_i \neq \emptyset.
\]
Equivalently, $\mathbb{B}$ is the set of bases of the matroid associated to $\cal{V}$.
There is a bijection $\mu\maps \mathbb{B} \to \cal{P}$ sending $\mathbbm{x}$ to the unique sign sequence $\alpha_{\mathbbm{x}}$ such that $\xi$ obtains its maximum on $\Delta_{\alpha}$ at the point $H_{\mathbbm{x}}$. We write $\mathbbm{x}_{\beta} = \mu^{-1}(\beta)$ for the subset associated to a sign sequence $\beta$. The covector $\xi$ induces a partial order\footnote{This partial order is the transitive closure of the relation $\preceq$, where $\mathbbm{x} \preceq \mathbbm{x}'$ if $|\mathbbm{x}\cap \mathbbm{x}'| = |\mathbbm{x}|-1 = k-1$ and $\xi(H_{\mathbbm{x}})<\xi(H_{\mathbbm{x}'})$. The first condition ensures that $H_{\mathbbm{x}}$ and $H_{\mathbbm{x}'}$ lie on the same one dimensional flat, so that $\xi$ cannot take the same value at these two points.   } $\leq $ on $\mathbb{B}\cong \cal{P}$.

Write $\mathbbm{x}^c$ for the complement in $I$ of the subset $\mathbbm{x}$.  Let $\mathbb{B}^{\vee}$ denote the set of bases of $\cal{V}^{\vee}$, i.e. $(n-k)$ element subsets of $I$.  Then $\mathbbm{x} \mapsto \mathbbm{x}^c$ defines a bijection from $\mathbb{B} \to \mathbb{B}^{\vee}$.  The bijection $\mu^{\vee} \maps \mathbb{B}^{\vee} \to \cal{P}^{\vee}$ is compatible with the equality $\cal{P}=\cal{P}^{\vee}$, so that $\mu(\mathbbm{x})=\mu^{\vee}(\mathbbm{x}^c)$~\cite[Lemma 2.9]{Gale}.

\subsection{Convolution algebras}\label{sec:ABAlgebraDefs}

This section introduces finite-dimensional Koszul algebras $A(\cal{V})$ and $B(\cal{V})$ associated to arrangements, and universal flat graded deformations $\tilde{A}(\cal{V}) $ and $\tilde{B}(\cal{V})$ of them.  With the exception of the deletion and restriction bimodules of Section~\ref{sec:bimodules}, which have not been explicitly discussed elsewhere, almost all of the material in this section is taken directly from the original sources \cite{Gale,HypertoricCatO,BLPPW}.

\subsubsection{Definitions and basic properties}\label{sec:conv}

We now recall the algebras associated to a polarized arrangement $\cal{V}=(V,\eta,\xi)$. Related algebras can be defined for unpolarized arrangements $(V,\eta)$, although these do not play an explicit role in \cite{Gale, BLPPW}. We will start with the polarized case, where the algebras satisfy interesting duality relationships, and then discuss the necessary modifications in the unpolarized case.   Recall the notation $\cal{F}$, $\cal{B}$, $\cal{P}$, $\cal{K}$ for feasible, bounded, bounded feasible, and compact feasible sign sequences of $\cal{V}$ from Section~\ref{sec:PolarizationDefs}.

\subsubsection{The $A$ algebras}

For sign sequences $\alpha,\beta\in \{\pm\}^n$, we write
 \[
	\alpha \leftrightarrow \beta \iff \alpha \text{ and } \beta \text{ differ in exactly one entry.}
\]
If $\alpha , \beta \in \cal{F}$ this means that $\Delta_{\alpha}$ and $\Delta_{\beta}$ are related by crossing a single hyperplane $H_i$, in which case we write $\beta=\alpha^i$.

Define a quiver $Q = Q(\cal{V})$ whose vertex set is $\cal{F}$ and arrows $p(\alpha,\beta)$ from $\alpha$ to $\beta$ and $p(\beta,\alpha)$ from $\beta$ to $\alpha$ if and only if $\alpha \leftrightarrow \beta$. Let $P(Q)$ be the path algebra of this quiver over $\Z$; $P(Q)$ has a distinguished idempotent $e_{\alpha}$ for all $\alpha \in \cal{F}$.

\begin{definition}[Definition 3.1 and Remark 3.1 of \cite{Gale}]
The $\Z$-algebra $\tilde{A}(\cal{V})$ is defined to be $P(Q) \otimes_{\Z} \Z[t_1,\ldots,t_n]$ modulo the two-sided ideal generated by the following relations:
\begin{enumerate}
\item[A1]: $e_{\alpha}$ for all $\alpha \in \cal{F}\setminus \cal{B}$, that is those feasible $\alpha$ that are not bounded,
\item[A2]: $p(\alpha,\beta) p(\beta,\gamma) - p(\alpha,\delta) p(\delta,\gamma)$ for all distinct $\alpha,\beta,\gamma,\delta \in \cal{F}$ with $\alpha \leftrightarrow \beta \leftrightarrow \gamma \leftrightarrow \delta \leftrightarrow \alpha$,
\item[A3]: $p(\alpha,\beta,\alpha) - t_i e_{\alpha}$ for all $\alpha,\beta \in \cal{F}$ with $\alpha \leftrightarrow \beta$ via a sign change in coordinate $i$.
\end{enumerate}
We give $\tilde{A}(\cal{V})$ a grading by setting $\deg(p(\alpha,\beta)) = 1$ and $\deg(t_i) = 2$. We can refine this grading to a multi-grading by $\Z\langle e_1, \ldots, e_n \rangle$ by letting
\begin{itemize}
\item $\deg(p(\alpha,\beta)) = e_i$ if $\alpha \to \beta$ changes a sign in position $i$,
\item $\deg (t_i) = 2e_i$;
\end{itemize}
we recover the single grading by sending $e_i$ to $1$ for all $i$. While $\tilde{A}(\cal{V})$ is a $\Z$-algebra a priori, we can view it as a $\Z[t_1,\ldots,t_n]$-algebra.
\end{definition}

Over $\R$ (or $\Q$ given a rational arrangement), the infinite-dimensional algebra $\tilde{A}(\cal{V})$ can be viewed as the universal graded flat deformation in the sense of \cite{BLPPW} of a finite-dimensional quasi-hereditary Koszul algebra $A(\cal{V})$; see \cite[Remark 4.5]{Gale}. We briefly recall the definition of $A(\cal{V})$ below.

We have $\R[t_1,\ldots,t_n] \cong \Sym((\R^n)^*)$, with the isomorphism identifying $t_i$ with the $i$-th coordinate function on $\R^n$.  We can then identify $V^*$ with $(\R^n)^*/V^{\perp}$. It follows that $\Sym(V^*)$ is the quotient of $\R[t_1,\ldots,t_n]$ by the ideal generated by all linear combinations of $t_1,\ldots,t_n$ whose coefficient vectors annihilate $V$; equivalently, we have $\Sym(V^*) \cong \frac{\R[t_1,\ldots,t_n]}{\Sym(V^{\perp})}$. The algebra $A(\cal{V})$ is defined similarly to $\tilde{A}(\cal{V})$, except that we take a quotient of $P(Q) \otimes_{\R} \Sym(V^*)$ instead of $P(Q) \otimes_{\Z} \Z[t_1,\ldots,t_n]$. Only the single $\Z$ grading descends to a grading on $A(\cal{V})$. It follows from \cite[Theorem 8.7]{BLPPW} that
\[
\Sym(V^{\perp}) \xrightarrow{j} \tilde{A}(\cal{V}) \xrightarrow{\pi} A(\cal{V})
\]
is a graded flat deformation that is universal in the sense of \cite[Remark 4.2]{BLPPW}, where $j$ includes an element of $\Sym(V^{\perp})$ into $\R[t_1,\ldots,t_n]$ and then multiplies by $1 \in \tilde{A}(\cal{V})$, while $\pi$ is the natural quotient map from $\tilde{A}(\cal{V})$ to $A(\cal{V})$.

\subsubsection{The $B$ algebras}

For $S = \{i_1, \ldots, i_m\} \subset \{1,\ldots n\}$, let $u_S := u_{i_1} \cdots u_{i_m} \in \Z[u_1, \ldots, u_n]$, and let $H_S \subset V + \eta$ denote the intersection of the hyperplanes corresponding to elements of $S$. For $\alpha, \beta \in \cal{P}$, set
\[
\tilde{R}_{\alpha \beta} := \frac{\Z[u_1,\ldots,u_n]}{(u_S: S \subset \{1,\ldots,n\} \textrm{ with } \Delta_{\alpha} \cap \Delta_{\beta} \cap H_S = \emptyset)}.
\]
Let $f_{\alpha,\beta} \in \tilde{R}_{\alpha \beta}$ be the element corresponding to $1 \in \Z[u_1,\ldots,u_n]$. For $\alpha, \beta, \gamma \in \cal{P}$, let $S(\alpha \beta \gamma) = \{i \in \{1,\ldots,n\} : \alpha(i) = \gamma(i) \neq \beta(i)\}$, where $\alpha(i), \beta(i), \gamma(i)$ denote the $i^{th}$ sign of $\alpha, \beta, \gamma$ respectively.

\begin{definition}\label{def:BStyle}
The $\Z$-algebra $\tilde{B}(\cal{V})$ is defined to be
$
\tilde{B}(\cal{V}):= \bigoplus_{\alpha,\beta \in \cal{P}} \tilde{R}_{\alpha,\beta}
$
with multiplication given by
\[
f_{\alpha,\beta} \cdot f_{\beta,\gamma} := u_{S(\alpha \beta \gamma)} f_{\alpha,\gamma}
\]
and extended bilinearly over $\Z[u_1,\ldots,u_n]$. The algebra $\tilde{B}(\cal{V})$ admits a single grading by setting $\deg(f_{\alpha,\beta}) = d_{\alpha,\beta}$, where $d_{\alpha,\beta}$ is the number of sign changes required to turn $\alpha$ into $\beta$, and $\deg(u_i) = 2$. We can refine to a multi-grading by $\Z\langle e_1,\ldots,e_n \rangle$ by letting $\deg(f_{\alpha, \beta}) := e_{i_1} + \cdots + e_{i_m}$ if $\beta$ is obtained from $\alpha$ by changing the signs in positions $i_1,\ldots,i_m$. We define the multi-degree of $u_i$ to be $2e_i$; we recover the single grading by sending $e_i$ to $1$ for all $i$. We can view $\tilde{B}(\cal{V})$ as an algebra over $\Z[u_1,\ldots,u_n]$.
\end{definition}

To define the finite-dimensional version $B(\cal{V})$ over $\R$ (or $\Q$ if $\cal{V}$ is rational), write $\R[u_1,\ldots,u_n] = \Sym(\R^n)$ by identifying $u_i$ with the $i^{th}$ coordinate function on $(\R^n)^*$. The inclusion of $V$ into $\R^n$ gives us a ring homomorphism from $\Sym(V)$ into $\R[u_1,\ldots,u_n]$ and thus into the quotient $\tilde{R}^{\R}_{\alpha \beta}$. For $\alpha, \beta \in \cal{P}$ we set
\[
R_{\alpha \beta} := \tilde{R}^{\R}_{\alpha \beta} \otimes_{\Sym(V)} \R,
\]
where the action of $\Sym(V)$ on $\R$ has all elements of $V$ acting as zero, so that $R_{\alpha \beta}$ can be viewed as a further quotient of $\tilde{R}^{\R}_{\alpha \beta}$ by $(c_1 u_1 + \cdots + c_n u_n : (c_1,\ldots, c_n) \in V)$.
We define $B(\cal{V})$ using $R_{\alpha \beta}$ in place of $\tilde{R}_{\alpha \beta}$ in the definition of $\tilde{B}(\cal{V})$; the multi-grading does not make sense on $B(\cal{V})$ but the single grading does. By \cite[Theorem 8.7]{BLPPW},
\begin{equation} \label{eq:Bt-flat}
  \Sym(V) \xrightarrow{j} \tilde{B}(\cal{V}) \xrightarrow{\pi} B(\cal{V})
\end{equation}
is a universal graded flat deformation.

\begin{remark}\label{rem:ExtendedBDef}
We could alternatively define $\tilde{B}(\cal{V})$ using rings $\tilde{R}_{\alpha \beta}$ for all bounded (but possibly infeasible) sign sequences $\alpha$; let $\B$ denote the set of such sequences, so that $\cal{P} = \cal{F} \cap \B$. The rest of the definition would be unchanged, since $\tilde{R}_{\alpha \beta}$ would be zero if $\alpha$ or $\beta$ is infeasible (the ideal in the quotient defining $\tilde{R}_{\alpha \beta}$ would contain $1 = u_{\emptyset}$). The product still makes sense without modification and agrees with the product on $\tilde{B}(\cal{V})$; note that if $\alpha, \gamma$ are feasible but $\beta$ is infeasible, then
\[
\Delta_{\alpha} \cap \Delta_{\gamma} \cap H_{S(\alpha \beta \gamma)} \subset \Delta_{\alpha} \cap \Delta_{\beta} \cap \Delta_{\gamma} = \emptyset.
\]
\end{remark}

\begin{theorem}[Theorem 4.14 and Corollary 4.15 \cite{Gale}]
For a polarized arrangement $\cal{V}$, we have graded algebra isomorphisms $\tilde{B}(\cal{V}) \cong \tilde{A}(\cal{V}^{\vee})$ and $B(\cal{V}) \cong A(\cal{V}^{\vee})$.
\end{theorem}

As a consequence, we have the following description of $\tilde{B}(\cal{V})$.
\begin{proposition}\label{prop:BTildeGensRels}
For a polarized arrangement $\cal{V} = (V,\eta,\xi)$, let $Q$ be the quiver with vertices $e_{\alpha}$ given by $\alpha\in \cal{B}$ and arrows $p(\alpha,\beta)$ from $\alpha$ to $\beta$ when $\alpha \leftrightarrow \beta$. The algebra $\tilde{B}(\cal{V})$ is $P(Q) \otimes_{\Z} \Z[u_1,\ldots,u_n]$ modulo the two-sided ideal generated by the following relations:
\begin{enumerate}
\item[B1]: $e_{\alpha}$ if $\alpha \in \cal{B}\setminus \cal{F}$, that is $\alpha$ bounded and infeasible,
\item[B2]: $p(\alpha,\beta) p(\beta,\gamma) - p(\alpha,\delta) p(\delta,\gamma)$ for all distinct bounded $\alpha,\beta,\gamma,\delta$ with $\alpha \leftrightarrow \beta \leftrightarrow \gamma \leftrightarrow \delta \leftrightarrow \alpha$,
\item[B3]: $p(\alpha,\beta,\alpha) - u_i e_{\alpha}$ for all bounded $\alpha,\beta$ with $\alpha \leftrightarrow \beta$ via a sign change in coordinate $i$.
\end{enumerate}
The gradings on $\tilde{B}(\cal{V})$ defined above match the ones defined as for $\tilde{A}(\cal{V})$.
\end{proposition}

The natural operations on $(V,\eta,\xi)$ on arrangements interact with the finite-dimensional algebras $A$ and $B$ (see \cite{Gale}).
\begin{itemize}
\item Gale duality of arrangements becomes Koszul duality of algebras $A(\cal{V}) \cong A(\cal{V}^\vee)$, $B(\cal{V}) \cong B(\cal{V}^\vee)$.
\item Polarization reversal of $(V,\eta,\xi)$ gives Ringel duality of algebras.
\item Applying alt to $(V,\eta,\xi)$ induces isomorphisms of algebras (this also holds for $\tilde{A}$ and $\tilde{B}$).
\end{itemize}

\subsubsection{Geometric aspects}\label{sec:AlgebrasGeometricAspects}

As discussed in \cite[Section 8]{BLPPW}, in the rational case $B(\cal{V})$ has an interpretation as a convolution algebra whose underlying vector space is the direct sum of cohomology spaces of $X_{\alpha \beta} := X_{\alpha} \cap X_{\beta}$, where $\alpha$ and $\beta$ are bounded feasible sign vectors, equipped with a convolution product:
\[
B(\cal{V})  \cong \bigoplus_{\alpha,\beta \in \cal{I}} H^{\ast}(X_{\alpha\beta})[-d_{\alpha\beta}].
\]
The algebra $\tilde{B}(\cal{V})$ has a similar interpretation as a convolution algebra built from equivariant cohomology spaces: we have
\[
\tilde{B}(\cal{V})  \cong \bigoplus_{\alpha,\beta \in \cal{I}} H^{\ast}_{T^k}( X_{\alpha\beta})[-d_{\alpha\beta}].
\]

As an upshot of the geometric definitions of $\tilde{B}(\cal{V})$ (resp. $B(\cal{V})$), one can identify the center of these convolution algebras with the equivariant (resp. ordinary) cohomology of the associated hypertoric variety (see \cite[Theorem 8.3 \& Proposition 8.5]{BLPPW}):
\[
Z(B(\cal{V})) \cong H^{\ast}(\M_{\cal{V}}) \,\, \textrm{ and } \,\, Z(\tilde{B}(\cal{V})) \cong H^{\ast}_{T^k}(\M_{\cal{V}}).
\]
The graded flat deformation \eqref{eq:Bt-flat} comes from forgetting the equivariant structure and ${\rm Sym}(V) = H^{\ast}_T({\rm pt})$.

It is expected that the convolution algebra $B(\cal{V})$ is an endomorphism algebra of the relative core in an appropriately defined Fukaya category of $\mathfrak{M}_{\cal{V}}$, see e.g. \cite[Remark 4.12]{Gale}. The results of this paper suggest the following Fukaya interpretation for the universal deformations $\tilde{B}(\cal{V})$ directly in terms of hyperplane data.

\begin{conjecture}
For a polarized arrangement $\cal{V}$ (not necessarily rational), the algebra $\tilde{B}(\cal{V})$ is the homology of the endomorphism algebra of the interiors of regions $\Delta_{\alpha}$ for $\alpha \in \cal{P}$ in a suitably defined wrapped Fukaya category of the complement of $\cal{H}_{\cal{V}} \subset (V + \eta)_{\C}$.
\end{conjecture}

Theorems~\ref{thm:LeftCyclicIsom} and \ref{thm:RightCyclicIsom} establish this conjecture when $\cal{V}$ is left or right cyclic as defined in Section~\ref{sec:Cyclic}; in this case we describe the stops for the wrapping in more detail in Section~\ref{rem:AurouxStops} below.

When $\cal{V}$ is rational, $\Delta_{\alpha}$ is the image under $\bar{\mu}_{\R}$ (plus any linear combination of $\mathrm{Re}(\bar{\mu}_{\C})$, $\mathrm{Im}(\bar{\mu}_{\C})$) of the relative-core toric Lagrangian $X_{\alpha} \subset \M_{\cal{V}}$. This observation, together with the geometric interpretations of the centers of $B(\cal{V})$ and $\tilde{B}(\cal{V})$, makes it tempting to speculate further that $\tilde{B}(\cal{V})$ admits an alternative interpretation in terms of some sort of algebraically-equivariant Fukaya category of $\M_{\cal{V}}$.  Thus we speculate that the algebra $\tilde{B}(\cal{V})$ arises as an endomorphism algebra in two ways---in the Fukaya category of the complexified hyperplane complement $\cal X_\cal{V}$, and in some equivariant Fukaya category of a hypertoric variety $\M_{\cal{V}}$.

We will not go further into Fukaya categories here; however, we can consider Grothendieck groups associated to $B(\cal{V})$ and $\tilde{B}(\cal{V})$, which given the Fukaya interpretations should be related to the middle cohomology of $\M_{\cal{V}}$. In fact, by \cite{HypertoricCatO} we have
\[
K_0(B(\cal{V}){\rm -mod})_{\C} \cong H^{2k}_{\mathbb{T}}(\mathfrak{M}_{\cal{V}}; \C)
\]
when $\M_{\cal{V}}$ is smooth, where classes $[P_{\alpha}]$ of indecomposable projectives over $B(\cal{V})$ correspond to classes $[X_{\alpha}]$ in cohomology.

\subsubsection{Unpolarized case}\label{sec:UnpolarizedHypertoricAlgs}

Absent a polarization and given only $(V,\eta)$, one can define variants $B'$, $\tilde{B}'$ of the algebras $B, \tilde{B}$   whose idempotents correspond to feasible $\alpha$ such that $\Delta_{\alpha}$ is compact (rather than just bounded above with respect to $\xi$, which has not been chosen here). Defining analogues of the $A$ algebras in this setting is more complicated, and will not be discussed here. It will be convenient for us to define
\[
\tilde{B}'(V,\eta) := \left( \sum_{\alpha \in \cal{K}} e_{\alpha} \right) \tilde{B}(V,\eta,\xi) \left( \sum_{\alpha \in \cal{K}} e_{\alpha} \right)
\]
for any choice of polarization $\xi$ on $(V,\eta)$, and similarly for $B'(V,\eta)$. Using the definition of $B$ and $\tilde{B}$ from Definition~\ref{def:BStyle}, it is clear that $\tilde{B}'(V,\eta)$ and $B'(V,\eta)$ admit definitions requiring no choice of $\xi$, and are thus independent of the choice of $\xi$. The idempotents $e_{\alpha}$ such that $\Delta_{\alpha} \neq \emptyset$ is compact are precisely those for which the indecomposable projective module $\tilde{B}(\cal{V}) e_{\alpha}$ is also injective.

\subsection{Deletion and restriction}\label{sec:GeneralDeletionRestriction}

\subsubsection{Operations on $(V,\eta)$ and $(V,\eta,\xi)$}
Two of the most natural operations one can perform on hyperplane arrangements are deletion of a hyperplane and restriction to a hyperplane; we briefly discuss how to view these operations as acting on $(V,\eta)$. Below, suppose $(V,\eta)$ is an arrangement.

For restriction, let $1 \leq i \leq n$ and consider the inclusion $\iota_i: \R^{n-1} \to \R^n$ of the $i^{th}$ coordinate hyperplane. Assume that $V + \iota(\R^{n-1}) = \R^n$ (i.e. that $V$ is not contained in the $i^{th}$ coordinate hyperplane of $\R^n$).  Define the \emph{restriction} of $(V,\eta)$ to the $i^{th}$ hyperplane to be
\[
	(V^i,\eta^i) := (\iota_i^{-1}(V),\iota_i^{-1}(\eta))
\]
(note that $\iota_i$ induces an isomorphism from $\R^{n-1} / V^i$ to $\R^n / V$).   The restriction is an arrangement of $n-1$ hyperplanes in $k-1$-space, naturally identified with the restriction of $\cal{H}_{\cal{V}}$ to its $i^{th}$ hyperplane.

Now, let $1 \leq i \leq n$ and consider the coordinate projection $\pi_i: \R^n \to \R^{n-1}$ that omits the $i^{th}$ coordinate. Assume that $V$ does not contain the $i^{th}$ coordinate axis of $\R^n$.
The \emph{deletion} of the $i^{th}$ hyperplane is defined by
\[
(V_i,\eta_i) := (\pi_i(V),\pi_i(\eta))
\]
(note that $\pi_i$ induces a map from $\R^n / V$ to $\R^{n-1} / V_i$). The deletion is an arrangement of $n-1$ hyperplanes in $k$-space, naturally identified with the deletion of the $i^{th}$ hyperplane from $\cal{H}_{\cal{V}}$. One can check that restriction and deletion preserve rationality of arrangements.

We now discuss the polarized case.

\begin{itemize}
\item (Restriction)
The \emph{restriction} of $\cal{V} = (V,\eta,\xi)$ to the $i^{th}$ hyperplane is defined by $\cal{V}^i := (V^i,\eta^i,\xi^i)$, where $(V^i,\eta^i)$ is the restriction of $(V,\eta)$ as before, and $\xi^i = \xi|_{V^i}$.

\item (Deletion)
The \emph{deletion} of the $i^{th}$ hyperplane from $\cal{V} = (V,\eta,\xi)$ is defined as
$\cal{V}_i:= (V_i,\eta_i,\xi_i)$, where $(V_i,\eta_i)$ is the deletion as before, and $\xi_i = \xi \circ \pi_i^{-1}$,
where $\pi$ is the isomorphism $V \cong V_i$.
\end{itemize}

Deletion and restriction are exchanged by Gale duality; if $\cal{V} = (V,\eta,\xi)$ is a polarized arrangement, then the restriction of $\cal{V}$ to its $i^{th}$ hyperplane $H_i$ is defined if and only if the deletion of the $i^{th}$ hyperplane $H_i^{\vee}$ of $\cal{V}^{\vee}$ is defined, and in this case we have $(\cal{V}^i)^{\vee} = (\cal{V}^{\vee})_i$ (see \cite[Lemma 2.6]{Gale}).

\subsubsection{Homomorphisms and bimodules for deletion and restriction} \label{sec:bimodules}

The deformed algebras $\tilde{A}(\cal{V})$ and $\tilde{B}(\cal{V})$ interact well with deletion and restriction.  Namely, to a pair of arrangements related by deletion or restriction, there is an associated (non-unital) algebra homomorphism that maps distinguished idempotents to distinguished idempotents.  Interestingly, these homomorphisms are only defined for the infinite-dimensional algebras $\tilde{A}(\cal{V})$ and $\tilde{B}(\cal{V})$, not for their finite-dimensional Koszul quotients $A(\cal{V})$ and $B(\cal{V})$.

\begin{definition}
Let $\cal{V} = (V,\eta,\xi)$ be a polarized arrangement of $n$ hyperplanes in $k$-space such that the restriction $\cal{V}^i$ of $\cal{V}$ to the $i^{th}$ hyperplane is well-defined, and choose a sign $s \in \{+,-\}$. Define an algebra homomorphism
\[
\rest_{\tilde{A}}^i(\cal{V},s): \tilde{A}(\cal{V}^i) \to \tilde{A}(\cal{V})
\]
by sending
\begin{itemize}
\item $e_{\alpha} \mapsto e_{\iota_{i,s}(\alpha)}$ where $\iota_{i,s}(\alpha)$ is $\alpha$ with sign $s$ inserted in position $i$,
\item $p(\alpha,\beta) \mapsto p(\iota_{i,s}(\alpha), \iota_{i,s}(\beta))$,
\item $t_j \mapsto t_j$ for $j < i$ and $t_j \mapsto t_{j+1}$ for $j \geq i$.
\end{itemize}
Note that if $\alpha$ is feasible, then $\iota_{i,s}(\alpha)$ is feasible for $s \in \{+,-\}$ (although it may be unbounded even if $\alpha$ is bounded; in this case, $e_{\iota_{i,s}(\alpha)} = 0$). One can check that the relations defining $\tilde{A}(\cal{V})$ are sent to zero under this homomorphism. The homomorphism is compatible with the map between multi-grading groups $\Z\langle e_1, \ldots, e_{n-1} \rangle \to \Z\langle e_1,\ldots,e_n \rangle$ sending $e_j$ to $e_j$ for $j < i$ and sending $e_j$ to $e_{j+1}$ for $j \geq i$. It preserves the single grading.
\end{definition}

We can obtain a $\tilde{B}$ version of the homomorphism $\rest_{\tilde{A}}^i(\cal{V},s)$ using Gale duality
\[
\tilde{B}(\cal{V}_i) = \tilde{A}((\cal{V}_i)^{\vee}) = \tilde{A}((\cal{V}^{\vee})^i).
\]
We define
\[
\del^{\tilde{B}}_i(\cal{V},s): \tilde{B}(\cal{V}_i) \to \tilde{B}(\cal{V})
\]
to be the homomorphism $\rest_{\tilde{A}}^i(\cal{V}^{\vee},s)$ under the above identifications.

For deletion and $\tilde{A}(\cal{V})$ (equivalently by duality, restriction and $\tilde{B}(\cal{V})$), we will consider two closely related homomorphisms $\del^{\tilde{A}}_i(\cal{V},s)$ and $\del'^{\tilde{A}}_i(\cal{V},s)$.

\begin{definition}
Let $\cal{V} = (V,\eta,\xi)$ be as above such that the deletion $\cal{V}_i$ of the $i^{th}$ hyperplane of $\cal{V}$ is well-defined, and choose a sign $s \in \{+,-\}$. Define an algebra homomorphism
\[
\del^{\tilde{A}}_i(\cal{V}, s): \tilde{A}(\cal{V}) \to \tilde{A}(\cal{V}_i)
\]
by sending
\begin{itemize}
\item $e_{\alpha} \mapsto e_{\rho_{i,s}(\alpha)}$ where $\rho_{i,s}(\alpha)$ is $\alpha$ with sign $s$ removed from position $i$, if the $i^{th}$ sign of $\alpha$ is $s$, and $e_{\rho_{i,s}(\alpha)} := 0$ otherwise,
\item $p(\alpha,\beta) \mapsto p(\rho_{i,s}(\alpha), \rho_{i,s}(\beta))$,
\item $t_j \mapsto t_j$ for $j < i$, $t_i \mapsto 0$, and $t_j \mapsto t_{j-1}$ for $j > i$.
\end{itemize}
One can check that this homomorphism is well-defined. It is compatible with the map between multi-grading groups $\Z\langle e_1, \ldots, e_n \rangle \to \Z\langle e_1, \ldots, e_{n-1} \rangle$ sending $e_j$ to $e_j$ for $j < i$, sending $e_i$ to zero, and sending $e_j$ to $e_{j-1}$ for $j > i$. It preserves the single grading.
\end{definition}

As above, define
\[
\rest_{\tilde{B}}^i(\cal{V}, s): \tilde{B}(\cal{V}) \to \tilde{B}(\cal{V}^i)
\]
using the identifications of $\tilde{B}(\cal{V})$ with $\tilde{A}(\cal{V}^{\vee})$ and of $\tilde{B}(\cal{V}^i)$ with $\tilde{A}((\cal{V}^{\vee})_i)$. For the homomorphism $(\del')^{\tilde{A}}_i(\cal{V}, s)$, it is convenient to define $(\rest')_{\tilde{B}}^i(\cal{V},s)$ first.

\begin{definition}
Let $\cal{V} = (V,\eta,\xi)$ be as above such that the restriction $\cal{V}^i$ to the $i^{th}$ hyperplane of $\cal{V}$ is well-defined, and choose a sign $s \in \{+,-\}$. Let
\[
\tilde{B}^s(\cal{V}) = \left( \sum_{\alpha : \alpha(i) = s} e_{\alpha} \right) \cdot \tilde{B}(\cal{V}) \cdot \left( \sum_{\alpha : \alpha(i) = s} e_{\alpha} \right).
\]
Define an algebra homomorphism
\[
(\rest')_{\tilde{B}}^i(\cal{V}, s): \tilde{B}^s(\cal{V}) \to \tilde{B}(\cal{V}^i)
\]
by sending
\begin{itemize}
\item $e_{\alpha} \mapsto e_{\rho_{i,s}(\alpha)}$,
\item $f_{\alpha, \beta} \mapsto f_{\rho_{i,s}(\alpha), \rho_{i,s}(\beta)}$,
\item $u_j \mapsto u_j$ for $j < i$, $u_i \mapsto 1$, and $u_j \mapsto u_{j-1}$ for $j > i$.
\end{itemize}
\end{definition}
One can check that this map sends the ideals defining $\tilde{R}_{\alpha \beta}$ on the left to the ideals defining $\tilde{R}_{\rho_{i,s}(\alpha) \rho_{i,s}(\beta)}$ on the right (this would not be true if we tried to define the homomorphism on the full algebra $\tilde{B}(\cal{V})$) and that it respects the products on each side, so it defines an algebra homomorphism. It is compatible with the map between multi-grading groups $\Z\langle e_1, \ldots, e_n \rangle \to \Z\langle e_1, \ldots, e_{n-1} \rangle$ sending $e_j$ to $e_j$ for $j < i$, sending $e_i$ to zero, and sending $e_j$ to $e_{j-1}$ for $j > i$. However, it does not preserve the single grading (note that $\deg(u_i) = 2$ while $\deg(1) = 0$). Define
\[
(\del')^{\tilde{A}}_i(\cal{V}, s): \tilde{A}^s(\cal{V}) \to \tilde{A}(\cal{V}_i)
\]
using the identifications of $\tilde{A}^s(\cal{V})$ with $\tilde{B}^s(\cal{V}^{\vee})$ and of $\tilde{A}(\cal{V}_i)$ with $\tilde{B}((\cal{V}^{\vee})^i)$, where
\[
\tilde{A}^s(\cal{V}) = \left( \sum_{\alpha : \alpha(i) = s} e_{\alpha} \right) \cdot \tilde{A}(\cal{V}) \cdot \left( \sum_{\alpha : \alpha(i) = s} e_{\alpha} \right).
\]
We have:
\begin{itemize}
\item $e_{\alpha} \mapsto e_{\rho_{i,s}(\alpha)}$,
\item $p(\alpha,\beta) \mapsto p(\rho_{i,s}(\alpha), \rho_{i,s}(\beta))$,
\item $t_j \mapsto t_j$ for $j < i$, $t_i \mapsto 1$, and $t_j \mapsto t_{j-1}$ for $j > i$.
\end{itemize}

From the above homomorphisms, one can define bimodules over the algebras in question by starting with the identity bimodule over the domain and inducing the left action via the homomorphism. Tensor products with these bimodules on the left send projectives to projectives.

\subsubsection{Compositions}

Suppose we are given $(V,\eta,\xi)$ arising as the $i^{th}$ deletion of another polarized arrangement $(V', \eta', \xi')$. Define a third polarized arrangement $(V'',\eta'',\xi'')$ as the $i^{th}$ restriction of $(V',\eta',\xi')$, assuming this makes sense. For any $s', s'' \in \{+,-\}$, we can consider the composite homomorphism
\[
\tilde{A}(\cal{V}'') \xrightarrow{\rest_{\tilde{A}}^i(\cal{V}'',s'')} \tilde{A}(\cal{V}') \xrightarrow{\del^{\tilde{A}}_i(\cal{V}', s')} \tilde{A}(\cal{V}).
\]
If $s' \neq s''$ this is zero; if $s' = s''$, then the resulting homomorphism is independent of $s' = s''$. For the $\tilde{B}$ algebras, if $s',s'' \in \{+,-\}$ we have the composite
\[
\tilde{B}(\cal{V}) \xrightarrow{\del^{\tilde{B}}_i(\cal{V},s')} \tilde{B}(\cal{V}') \xrightarrow{\rest_{\tilde{B}}^i(\cal{V}', s'')} \tilde{B}(\cal{V}''),
\]
with the same properties. In terms of hyperplanes in $V + \eta$, one can think of the $\tilde{B}$ homomorphism as being determined by adding an additional hyperplane, then restricting to it. If one composes two such addition-restriction homomorphisms, the result is the same as adding both hyperplanes first, then restricting to their intersection; if the intersection is empty, the composite of the two addition-restriction homomorphisms is zero.

One can obtain the same composite homomorphism using the variants $(\del')^{\tilde{A}}_i(\cal{V}', s')$ and $(\rest')_{\tilde{B}}^i(\cal{V}', s'')$. Note that in the above compositions, $\rest_{\tilde{A}}^i(\cal{V}'',s'')$ has image contained in the non-unital subalgebra $\tilde{A}^s(\cal{V}')$ on which $(\del')^{\tilde{A}}_i(\cal{V}', s')$ is defined. Similarly, $\del^{\tilde{B}}_i(\cal{V},s')$ has image contained in the non-unital subalgebra $\tilde{B}^s(\cal{V}')$ on which $(\rest')_{\tilde{B}}^i(\cal{V}', s'')$ is defined. Composing using these variant homomorphisms, one can check that we get the same composite homomorphism as above. This composite preserves both the single grading and the multi-grading by $\Z \langle e_1, \ldots, e_n \rangle$.

\section{Cyclic arrangements}\label{sec:Cyclic}

\subsection{Definitions}

\subsubsection{Cyclic arrangements}

We let $\Grr_{k,n}^{> 0}$ denote the positive Grassmannian consisting of positive (i.e. totally positive) $k$-dimensional subspaces of $\R^n$, i.e. the set of subspaces whose Pl{\"u}cker coordinates are all nonzero and have the same sign.
An element in $\Grr_{k,n}^{> 0}$ can be represented as the column span of  an $n\times k$ matrix with strictly positive maximal minors.

\begin{definition}[cf. \cite{Shannon, Ziegler, RamirezAlfonsin, FRA, KarpWilliams}]
An arrangement $(V,\eta)$ is called \emph{cyclic} if:
\begin{itemize}
\item $V \in \Grr_{k,n}^{> 0}$,
\item $V + \langle \eta \rangle \in \Grr_{k+1,n}^{> 0}$, and
\item $\eta$ is \emph{positively oriented} with respect to $V$, which means that the first coordinate of the orthogonal projection of some, or equivalently every, representative $w \in \R^n$ of $\eta$ onto $V^{\perp}$ is positive.
\end{itemize}
\end{definition}

\begin{theorem}[Theorem 6.16 of \cite{KarpWilliams}]
Let $(V,\eta)$ be a cyclic arrangement. The map from the affine $k$-dimensional space $V + \eta$ to the projectivization of the linear $k+1$-dimensional space $W := V + \langle \eta \rangle$ sending $v + \eta$ to $[v + \eta]$ restricts to a homeomorphism from the union of the compact regions of $\cal{H}_{(V,\eta)}$, a subset of $V + \eta$, to the $m=1$ ``B-amplituhedron'' $B_{n,k,1}(W) \subset \mathbb{P}(W)$.
\end{theorem}

Karp--Williams also show that given an explicit $n \times (k+1)$ matrix $Z^T$ representing $(V,\eta)$ as above, the map from $\mathbb{P}(W)$ to $\Grr_{k,k+1}$ sending $X$ to $Z(X^{\perp})$ restricts to a homeomorphism from $B_{n,k,1}(W)$ to the $m=1$ amplituhedron $A_{n,k,1}(Z)$ as defined by Arkani--Hamed and Trnka~\cite{AHT14} (in fact, they show an analogous result for general $m$).

\subsubsection{Left and right cyclic polarized arrangements}\label{sec:LRCyclicDefs}

We propose that there are two natural analogues of the definition of cyclicity in the world of polarized arrangements; below, we define ``left cyclic'' and ``right cyclic'' polarized arrangements.

\begin{definition}
Let $\cal{V} = (V,\eta,\xi)$ be a polarized arrangement. We say that $\V$ is \emph{left cyclic} if:
\begin{itemize}
\item $V + \langle \eta \rangle \in \Grr_{k+1,n}^{>0}$,
\item $(\xi,\id)(V) \in \Grr_{k,n+1}^{>0}$, and
\item $\eta$ is positively oriented with respect to $V$.
\end{itemize}
where $(\xi,\id)$ is the linear map from $V$ to $\R^{n+1}$ whose first coordinate is given by the linear functional $\xi$ on $V$. Similarly, we say that $\V$ is \emph{right cyclic} if:
\begin{itemize}
\item $V + \langle \eta \rangle \in \Grr_{k+1,n}^{>0}$,
\item $(\id,(-1)^k \xi)(V) \in \Grr_{k,n+1}^{>0}$, and
\item $\eta$ is positively oriented with respect to $V$.
\end{itemize}
\end{definition}
The conditions $(\xi,\id)(V) \in \Grr_{k,n+1}^{>0}$ and $(\id,(-1)^k \xi)(V) \in \Grr_{k,n+1}^{>0}$ both imply that $V \in \Grr_{k,n}^{> 0}$, so if $(V,\eta,\xi)$ is left or right cyclic then $(V,\eta)$ is cyclic. In Section~\ref{sec:AltGale} we will show that left and right cyclicity are related by the combination of Gale duality, alt, and polarization reversal.

\subsection{Background results}

\subsubsection{Sign variation}

The results of Karp--Williams make extensive use of an explicit identification of the compact nonempty regions $\Delta_{\alpha}$ for a cyclic arrangement $(V,\eta)$ as those for which $\alpha$ has ``sign variation'' $k$, i.e. the signs in $\alpha$ change from $+$ to $-$ or $-$ to $+$ exactly $k$ times when reading from left to right (or from right to left). We review some properties of sign variation and cyclic arrangements here; we note that sign variation also plays a crucial role in Arkani-Hamed--Thomas--Trnka's ``binary code'' reformulation of higher-$m$ amplituhedra in terms of the $m=1$ amplituhedron \cite{Binary}.

\begin{definition}
For $\alpha \in \{+,-,0\}^n$, let $\var(\alpha)$ denote the number of sign changes in $\alpha$ as above, ignoring any zeroes. Let $\overline{\var}(\alpha)$ denote the maximum value of $\var(\alpha')$ over all $\alpha' \in \{+,-\}^n$ obtained from $\alpha$ by replacing each zero with either plus or minus (different zeroes may be replaced with different signs).
\end{definition}

\begin{proposition}[Proposition 6.14, Definition 5.1 of \cite{KarpWilliams}]\label{prop:SignPatternCpct}
If $(V,\eta)$ is cyclic, a sign sequence $\alpha$ represents a (nonempty) compact region of $(V,\eta)$ if and only if $\var(\alpha) = k$ and $\alpha$ starts with a plus (equivalently, $\var(\alpha) = k$ and $\alpha$ ends with $(-1)^k$). It represents a noncompact region of $(V,\eta)$ if and only if $\var(\alpha) < k$.
\end{proposition}

If $z$ is a vector in $\R^n$, we can define $\var(z)$ and $\overline{\var}(z)$ by taking $\alpha$ to be the signs of the coordinates of $z$.

\begin{lemma}\label{lem:VarOfAlt}
For $z \in \R^n \setminus \{0\}$, we have
$
\var(\alt(z)) = n-1-\overline{\var}(z),
$
or equivalently
$
\overline{\var}(\alt(z)) = n-1-\overline{\var}(z).
$
Also, $V$ is positive if and only if $\alt(V^{\perp})$ is positive.
\end{lemma}

\begin{proof}
This is \cite[Lemma 3.3]{KarpWilliams}, following \cite{Hilbert,GantmakherKrein,Hochster,Ando}.
\end{proof}

\begin{lemma}\label{lem:VarOfU}
For $(V,\eta)$ with $V \in \Grr_{k,n}^{> 0}$ and $V + \langle \eta \rangle \in \Grr_{k+1,n}^{>0}$, if $u := \proj_{V^{\perp}}(\eta)$, then $\var(u) = \overline{\var}(u) = k$.
\end{lemma}

\begin{proof}
This is a consequence of \cite[Theorem 3.4]{KarpWilliams}, which follows \cite{GantmakherKrein} (see also \cite[Definition 6.6]{KarpWilliams}).
\end{proof}

\begin{lemma}\label{lem:VarOfAv}
Let $A$ be a totally positive $m \times n$ matrix and let $v \in \R^n$. We have $\var(Av) \leq \var(v)$; if equality holds, then the signs of the first (and thus last) nonzero entries of $Av$ and $v$ are equal.
\end{lemma}

\begin{proof}
This is stated in the proof of \cite[Proposition 6.8]{KarpWilliams}, following \cite{Schoenberg,GantmakherKrein}\footnote{In the literature this result appears with the restriction $m \geq n$. In general, we can add rows to $A$ so it has more rows than columns and remains totally positive; let $A'$ denote the resulting matrix. We have $\var(Av) \leq \var(A'v) \leq \var(v)$. If $\var(Av) = \var(v)$, then the first nonzero entries of $A'v$ and $v$ have the same sign. If $Av = 0$, then $\var(v) = 0$, so $v = 0$; it follows that the first nonzero entries of $v$ and $Av$ have the same sign.}.
\end{proof}

\subsubsection{Sign variation for polarized arrangements}
If $(V,\eta)$ is a cyclic arrangement, then all deletions and restrictions from Section~\ref{sec:GeneralDeletionRestriction} are defined for $(V,\eta)$. We can thus obtain new arrangements by deleting the first and last hyperplanes of $(V,\eta)$; remembering the deleted hyperplanes as polarizations gives us two polarized arrangements $(V', \eta', \xi')$ and $(V'', \eta'', \xi'')$.

\begin{definition}
Given cyclic $(V,\eta)$,  set $(V', \eta', \xi')$ to be the deleted arrangement $(V_1, \eta_1)$, with $\xi'$   the unique linear functional on $V \cong V'$ whose level sets on the affine space $V + \eta \cong V' + \eta'$ are parallel to the deleted hyperplane $H_1$ and which is increasing in the positive normal direction to $H_1$. We define $(V'', \eta'', \xi'')$ similarly, with $(V'', \eta'') = (V_n, \eta_n)$ and $\xi''$ increasing in $(-1)^k$ times the positive normal direction to $H_n$.
\end{definition}

\begin{remark}
In contrast to Section~\ref{sec:GeneralDeletionRestriction}, here we start with an unpolarized arrangement and obtain a polarized arrangement after deletion. The unpolarized part of this polarized arrangement, though, comes from Section~\ref{sec:GeneralDeletionRestriction}.
\end{remark}

Let $\cal{V}' = (V', \eta', \xi')$ and $\cal{V}'' = (V'', \eta'', \xi'')$. One can check that these polarized arrangements are left and right cyclic respectively, and that all left and right cyclic arrangements arise in this manner. Reversing the perspective, given $(V,\eta,\xi)$ left cyclic, we will write $(V^l, \eta^l)$ for an (arbitrary) choice of cyclic arrangement $(V^l, \eta^l)$ producing $(V,\eta,\xi)$ as its polarized arrangement $(V',\eta',\xi')$. Similarly, we write $(V^r, \eta^r)$ for a choice of cyclic arrangement producing a given right cyclic polarized arrangement $(V,\eta,\xi)$.

In fact, the bounded feasible regions of a left cyclic polarized arrangement $(V,\eta,\xi)$ naturally correspond to the (nonempty) compact regions of $(V^l, \eta^l)$ (an analogous statement holds in the right cyclic case). To see this, first note that by Proposition~\ref{prop:SignPatternCpct}, the sign sequence of any nonempty compact region of $(V^l, \eta^l)$ starts with a plus, so it suffices to determine when $+\alpha$ is nonempty and compact for sign sequences $\alpha \in \{+,-\}^n$.

\begin{lemma}\label{lem:HH'helper}
The following statements hold for a sign sequence $\alpha \in \{+,-\}^n$.
\begin{itemize}
\item $\Delta_{+\alpha}$ is empty if and only if $\Delta_{\alpha}$ is empty (i.e. $\alpha$ is infeasible), in which case $\alpha$ is bounded.
\item $\Delta_{+\alpha}$ is nonempty and compact if and only if $\Delta_{\alpha}$ is nonempty and bounded (i.e. $\alpha$ is bounded feasible).
\item $\Delta_{+\alpha}$ is noncompact if and only if $\Delta_{\alpha}$ is nonempty and unbounded (i.e. $\alpha$ is feasible but unbounded).
\end{itemize}
\end{lemma}

\begin{proof}
We have $\Delta_{+\alpha} \subset \Delta_{\alpha}$, so if $\Delta_{\alpha}$ is empty then so is $\Delta_{+\alpha}$. Conversely, if $\Delta_{\alpha}$ is nonempty, then either $\var(\alpha) \leq k-1$, or $\var(\alpha) = k$ and $\alpha$ starts with a plus by Proposition~\ref{prop:SignPatternCpct}. In either case, we have $\var(+\alpha) \leq k$ and $+\alpha$ starts with a plus, so $\Delta_{+\alpha}$ is nonempty by Proposition~\ref{prop:SignPatternCpct}.

If $\Delta_{+\alpha}$ is nonempty and compact then $\Delta_{\alpha}$ is nonempty by above. The affine functional $\xi$ on $\Delta_{\alpha}$ is bounded above on $\Delta_{+\alpha} \subset \Delta_{\alpha}$ by compactness. Since the hyperplane $H_1 \subset V^l + \eta^l$ is a level set of $\xi$ and $\xi$ is larger on the positive side of $H_1$ than on the negative side, we see that $\xi$ is also bounded above on $\Delta_{-\alpha} \subset \Delta_{\alpha}$, so $\xi$ is bounded above on $\Delta_{\alpha} = \Delta_{+\alpha} \cup \Delta_{-\alpha}$.

Conversely, suppose $\Delta_{\alpha}$ is nonempty and $\xi$ is bounded above on $\Delta_{\alpha}$. By above, $\Delta_{+\alpha}$ is nonempty; if $\Delta_{+\alpha}$ is noncompact, then it contains some semi-infinite ray $\rho$. Without loss of generality we may take $\rho$ to be in the interior of $\Delta_{\alpha}$, which must then contain an open cone $C(\rho)$ of semi-infinite rays centered around $\rho$. By construction, $\Delta_{+\alpha}$ is contained between level sets of $\xi$ acting on $V + \eta$, so $\rho$ must be parallel to $H_1$. Since $\xi$ is constant on $\rho$, $\xi$ must be unbounded above on some rays in the cone $C(\rho) \subset \Delta_{\alpha}$, a contradiction. The final item of the lemma follows from the first two items.
\end{proof}

For $\alpha \in \{+,-\}^n$, let $\var_l(\alpha) := \var(+\alpha)$.

\begin{corollary}\label{cor:LeftCyclicBoundedFeasible}
A sign sequence $\alpha \in \{+,-\}^n$ is feasible for the left cyclic arrangement $(V,\eta,\xi)$ if and only if $\var_l(\alpha) \leq k$ and is bounded if and only if $\var_l(\alpha) \geq k$.
\end{corollary}

\begin{proof}
By Lemma~\ref{lem:HH'helper}, $\alpha$ is feasible if and only if $\Delta_{+\alpha}$ is nonempty, which by Proposition~\ref{prop:SignPatternCpct} happens if and only if $\var(+\alpha) \leq k$. Similarly, $\alpha$ is bounded if and only if $\Delta_{+\alpha}$ is compact, which happens if and only if $\var(+\alpha) \geq k$.
\end{proof}

We give the corresponding statements in the right cyclic case without proof; let $(V,\eta,\xi)$ be a right cyclic polarized arrangement.

\begin{lemma}\label{lem:HH'helper2}
The following statements hold for a sign sequence $\alpha \in \{+,-\}^n$:
\begin{itemize}
\item $\Delta_{\alpha(-1)^k}$ is empty if and only if $\Delta_{\alpha}$ is empty (i.e. $\alpha$ is infeasible), in which case $\alpha$ is bounded.
\item $\Delta_{\alpha(-1)^k}$ is nonempty and compact if and only if $\Delta_{\alpha}$ is nonempty and bounded (i.e. $\alpha$ is bounded feasible).
\item $\Delta_{\alpha(-1)^k}$ is noncompact if and only if $\Delta_{\alpha}$ is nonempty and unbounded (i.e. $\alpha$ is feasible but unbounded).
\end{itemize}
\end{lemma}

For $\alpha \in \{+,-\}^n$, let $\var_r(\alpha) := \var(\alpha(-1)^k)$.

\begin{corollary}
A sign sequence $\alpha \in \{+,-\}^n$ is feasible for the right cyclic arrangement $(V,\eta,\xi)$ if and only if $\var_r(\alpha) \leq k$ and is bounded if and only if $\var_r(\alpha) \geq k$.
\end{corollary}

\subsubsection{Consequences for hypertoric varieties and algebras}

For cyclic arrangements and left or right cyclic polarized arrangements, the above results give us a nice parametrization of the Lagrangians $X_{\alpha} \subset \M_{\cal{V}}$ and ${\rm int}(\Delta_{\alpha}) \subset (V + \eta)_{\C} \setminus \cal{H}_{\cal{V}}$ appearing in Section~\ref{sec:AlgebrasGeometricAspects}, for $\alpha \in \cal{P}$ as well as $\alpha \in \cal{K}$. When $\cal{V}$ is left cyclic, relative core components of $\M_{\cal{V}}$ are those $X_{\alpha}$ with $\var_l(\alpha) = k$, and similarly for right cyclic $\cal{V}$. The core components of $\M_{\cal{V}}$ are those $X_{\alpha}$ with $\alpha(1) = +$ and $\var(\alpha) = k$.

Correspondingly, we can use the above results to describe the algebra $\tilde{B}(\cal{V})$ more explicitly in the case of interest; we first discuss the case where $\cal{V}$ is left cyclic. Let $Q$ be the quiver whose vertices are sign sequences $\alpha \in \{+,-\}^n$ with $\var_l(\alpha) \geq k$, with arrows $p(\alpha,\beta)$ from $\alpha$ to $\beta$ and $p(\beta,\alpha)$ from $\beta$ to $\alpha$ whenever $\alpha \leftrightarrow \beta$.

\begin{corollary}\label{cor:SimpleBTilde}
If $\cal{V}$ is left cyclic, the $\Z$-algebra $\tilde{B}(\cal{V})$ can be naturally identified with $P(Q) \otimes_{\Z} \Z[u_1,\ldots,u_n]$ modulo the two-sided ideal generated by the following relations:
\begin{enumerate}
\item[A1]\label{it:SimpleA1L}: $e_{\alpha}$ for all $\alpha$ with $\var_l(\alpha) > k$,
\item[A2]\label{it:SimpleA2L}: $p(\alpha,\beta) p(\beta,\gamma) - p(\alpha,\delta) p(\delta,\gamma)$ for all distinct $\alpha,\beta,\gamma,\delta$ with $\var_l \geq k$ and $\alpha \leftrightarrow \beta \leftrightarrow \gamma \leftrightarrow \delta \leftrightarrow \alpha$,
\item[A3]\label{it:SimpleA3L}: $p(\alpha,\beta,\alpha) - u_i e_{\alpha}$ for all $\alpha,\beta $ with $\var_l \geq k$ and $\alpha \leftrightarrow \beta$.
\end{enumerate}
\end{corollary}

The right cyclic case has a similar description with $\var_l$ replaced by $\var_r$ everywhere.

\subsection{Cyclic arrangements as an equivalence class}\label{sec:AlternateCyclic}

The following alternative characterization of cyclic arrangements will be useful.

\begin{proposition}\label{prop:AlternateCyclicDef}
Given $(V,\eta)$, let $\phi \in (V + \langle \eta \rangle)^*$ be the unique functional with $\phi(V) = 0$ and $\phi(\eta) = 1$. The arrangement $(V,\eta)$ is cyclic if and only if $(\phi, \id)(V + \langle \eta \rangle) \in \Grr^{> 0}_{k+1,n+1}$.
\end{proposition}

\begin{proof}
We first claim that given either of the conditions in the statement, we have $V \in \Grr^{> 0}_{k,n}$ and $V + \langle \eta \rangle \in \Grr^{> 0}_{k+1,n}$. This is immediate if $(V,\eta)$ is left cyclic. Assuming that $(\phi, \id)(V + \langle \eta \rangle) \in \Grr^{> 0}_{k+1,n+1}$, represent $V$ as the column span of a matrix $A'$, and represent $\eta$ by a vector $w' \in \R^n$. Then $(\phi, \id)(V + \langle \eta \rangle)$ is the column span of
$\left[
\begin{array}{c|c}
1 & 0 \\
\hline
w' & A'
\end{array}
\right]$,
so the maximal minors of this matrix all have the same sign. It follows that the maximal minors of $A'$ and of $\left[ \begin{array}{c|c} w' &  A' \end{array} \right]$ all have the same sign, so $V \in \Grr^{> 0}_{k,n}$ and $V + \langle \eta \rangle \in \Grr^{> 0}_{k+1,n}$. It thus suffices to show that $\eta$ is positively oriented with respect to $V$ if and only if $(\phi, \id)(V + \langle \eta \rangle) \in \Grr^{> 0}_{k+1,n+1}$, assuming that $V \in \Grr^{> 0}_{k,n}$ and $V + \langle \eta \rangle \in \Grr^{> 0}_{k+1,n}$, and we will make these assumptions below.

If $A$ is a matrix with $p$ columns, we will write $\widetilde{A}$ for $A$ with its columns permuted by the longest permutation in the symmetric group $\mathfrak{S}_p$. We let $i$ label the rows and $j$ label the columns of a given matrix, so that an expression like $(-1)^j \widetilde{A}$ means ``$\widetilde{A}$ with its $j^{th}$ column multiplied by $(-1)^j$ for $1 \leq j \leq k$,'' and similarly for expressions like $(-1)^i \widetilde{A}$.

Since $V \in \Grr^{>0}_{k,n}$, there exists a unique totally positive matrix $A$ of size $(n-k) \times k$ such that $V$ is the column span of the matrix
$\left[
\begin{array}{c}
I_k \\
\hline
(-1)^{j+k} \widetilde{A}
\end{array}
\right]$ (see \cite[Lemma 3.9]{Postnikov}),
where $I_k$ is the identity matrix of size $k$ (note that all maximal minors of the above block matrix are positive). There exists a unique vector $w' \in \R^{n-k}$ such that
$\left[
\begin{array}{c}
0 \\
\hline
w'
\end{array}
\right]$
represents $\eta \in \R^n / V$;
then
\[
V + \langle \eta \rangle = {\rm colspan}\left[
\begin{array}{c|c}
0 & I_k \\
\hline
w' & (-1)^{j+k} \widetilde{A}
\end{array}
\right], \quad \text{and}
\qquad
(\phi, \id)(V + \langle \eta \rangle)
= {\rm colspan}\left[
\begin{array}{c|c}
1 & 0 \\
\hline
0 & I_k \\
\hline
w' & (-1)^{j+k} \widetilde{A}
\end{array}
\right]
\]
Since $V + \langle \eta \rangle \in \Grr^{> 0}_{k+1,n}$, there exists a unique vector $w \in \{w', -w'\}$ such that the minors of
$\left[
\begin{array}{c|c}
0 & I_k \\
\hline
(-1)^k w & (-1)^{j+k} \widetilde{A}
\end{array}
\right]$
are all positive. It follows that the maximal minors of
$\left[
\begin{array}{c|c}
1 & 0 \\
\hline
0 & I_k \\
\hline
(-1)^k w & (-1)^{j+k} \widetilde{A}
\end{array}
\right]$
are all positive, so that $\left[ \begin{array}{c|c} A & w \end{array} \right]$ is a totally positive matrix. Writing $w' = (-1)^{\ell} w$ for some $\ell$ defined modulo $2$, we want to show that $\eta$ is positively oriented with respect to $V$ if and only if $\ell = k$ modulo $2$.

To do so, let $u$ be the orthogonal projection of $\eta$ onto $V^{\perp}$. Since $u$ is equivalent to
$\left[
\begin{array}{c}
0 \\
\hline
(-1)^{\ell} w
\end{array}
\right]$
modulo $V$, we can write
\[
u =
\left[
\begin{array}{c|c}
0 & I_k \\
\hline
(-1)^{\ell} w & (-1)^{j+k} \widetilde{A}
\end{array}
\right]
\left[
\begin{array}{c}
1 \\
\hline
z
\end{array}
\right]
\]
for some $z \in \R^k$. Expanding out this product of block matrices, we get
$\left[
\begin{array}{c}
z \\
\hline
v
\end{array}
\right]$
where
\begin{align*}
v &=
\left[
\begin{array}{c|c}
(-1)^{\ell} w & (-1)^{j+k} \widetilde{A}
\end{array}
\right]
\left[
\begin{array}{c}
1 \\
\hline
z
\end{array}
\right]
\;=\;
\left[
\begin{array}{c|c}
A & w
\end{array}
\right]
\left[
\begin{array}{c}
z_k \\
-z_{k-1} \\
\vdots \\
(-1)^{k-1} z_1 \\
(-1)^k (-1)^{\ell-k}
\end{array}
\right], \qquad
\text{writing $z =
\left[
\begin{array}{c}
z_1 \\
\vdots \\
z_k
\end{array}
\right]$.}
\end{align*}
By Lemma~\ref{lem:VarOfAv}, we have
\[
\var(v) \leq \var \left(
\left[
\begin{array}{c}
z_k \\
-z_{k-1} \\
\vdots \\
(-1)^{k-1} z_1 \\
(-1)^k (-1)^{\ell-k}
\end{array}
\right]
\right)
\]

Now, we have $\var(u) = \overline{\var}(u) = k$ by Lemma~\ref{lem:VarOfU}. The first coordinate $u_1$ of $u$ is equal to $z_1$. If this coordinate were zero, we would not have $\var(u) = \overline{\var}(u)$, so either $u_1 = z_1 > 0$ (if $\eta$ is positively oriented with respect to $V$) or $u_1 = z_1 < 0$ (if $\eta$ is negatively oriented with respect to $V$). We want to show that
\[
u_1 = z_1 > 0 \,\, \textrm{ if and only if } \,\, \ell = k \textrm{ mod } 2.
\]

Assume one of the above two statements holds without the other; we will derive a contradiction. It follows that $(-1)^{k-1} z_1$ and $(-1)^k (-1)^{\ell-k}$ have the same sign, so
\[
\var \left(
\left[
\begin{array}{c}
z_k \\
-z_{k-1} \\
\vdots \\
(-1)^{k-1} z_1 \\
(-1)^k (-1)^{\ell-k}
\end{array}
\right]
\right)
=
\var \left(
\left[
\begin{array}{c}
z_k \\
-z_{k-1} \\
\vdots \\
(-1)^{k-1} z_1
\end{array}
\right]
\right)
=:
\var \left(
z'
\right)
\]
and we get

\begin{align}\label{eq:varv}
\var(v) &\leq
\var \left(
z'
\right)
\;\leq \;
\overline{\var} \left(
z'
\right)
\; = \; k-1-\var(z).
\end{align}
Since $u =
\left[
\begin{array}{c}
z \\
\hline
v
\end{array}
\right]$, we have
$
k = \var(u)
\leq \var(z) + \var(v) + 1
\leq \var(z) + (k-1-\var(z)) + 1
= k.
$
The inequalities in \eqref{eq:varv} must therefore be equalities, so that $\var(v)=\var(z')$.

By Lemma~\ref{lem:VarOfAv}, the first nonzero entries of $v$ and $z'$ must have the same sign, and since the vectors have the same value of $\var$, their last nonzero entries must also have the same sign. The last entry $v_{n-k}$ of $u$ is nonzero (otherwise $\var(u) \neq \overline{\var}(u)$ as before), so the sign of $v_{n-k}$ is the sign of $(-1)^{k-1} z_1 = (-1)^{k-1} u_1$. This contradicts $\var(u) = k$, proving the proposition.
\end{proof}

\begin{corollary}\label{cor:AnotherCyclicDefCHANGED}
Given $(V,\eta)$, let $\phi$ be defined as in Proposition~\ref{prop:AlternateCyclicDef}. The arrangement $(V,\eta)$ is cyclic if and only if $(\id, (-1)^k \phi)(V + \langle \eta \rangle) \in \Grr^{> 0}_{k+1,n+1}$.
\end{corollary}

\begin{proof}
By Proposition~\ref{prop:AlternateCyclicDef}, it suffices to show that $(\phi,\id)(V + \langle \eta \rangle) \in \Grr^{> 0}_{k+1,n+1}$ if and only if $(\id, (-1)^k \phi)(V + \langle \eta \rangle) \in \Grr^{> 0}_{k+1,n+1}$. Picking a matrix $A'$ and a vector $w'$ representing $V$ and $\eta$ respectively, $(\phi,\id)(V + \langle \eta \rangle)$ is the column span of
$\left[
\begin{array}{c|c}
1 & 0 \\
\hline
w' & A'
\end{array}
\right]$. Comparing this matrix with
$\left[
\begin{array}{c|c}
(-1)^k w' & A' \\
\hline
1 & 0
\end{array}
\right]$, the maximal minors not involving the top row of the first matrix are $(-1)^{k(k+1)} = 1$ times the maximal minors not involving the bottom row of the second matrix. The maximal minors involving the top row of the first matrix are $(-1)^{k^2 + k} = 1$ times the maximal minors involving the bottom row of the second matrix. Since the column span of the second matrix is $(\id, (-1)^k \phi)(V + \langle \eta \rangle) \in \Grr^{> 0}_{k+1,n+1}$, the corollary follows.
\end{proof}

\begin{corollary}
For a given $(n,k)$, the cyclic arrangements $(V,\eta)$ form an equivalence class of arrangements.
\end{corollary}

\begin{proof}
It follows from Proposition~\ref{prop:AlternateCyclicDef} that an arrangement $(V,\eta)$, with $V$ the column span of $A'$ and $\eta$ represented by $w'$, is cyclic if and only if the maximal minors of
$\left[
\begin{array}{c|c}
1 & 0 \\
\hline
w' & A'
\end{array}
\right]$, or equivalently of
$\left[
\begin{array}{c|c}
0 & 1 \\
\hline
A' & w'
\end{array}
\right]$,
all have the same sign. This holds if and only if the maximal minors of
$\left[
\begin{array}{c|c}
A' & w' \\
\hline
0 & 1
\end{array}
\right]$
not involving the bottom row all have one sign and the maximal minors involving the bottom row all have $(-1)^k$ times this sign. Such arrangements $(V,\eta)$ constitute an equivalence class.
\end{proof}

\begin{corollary}\label{cor:PolarizedCyclicAlternateDef1}
Given a polarized arrangement $(V,\eta,\xi)$, let $\phi$ be defined as in Proposition~\ref{prop:AlternateCyclicDef}. Then $(V,\eta,\xi)$ is left cyclic if and only $(\id, (-1)^k \phi)(V + \langle \eta \rangle) \in \Grr^{> 0}_{k+1,n+1}$ and $(\xi,\id)(V) \in \Grr^{> 0}_{k,n+1}$. Similarly, $(V,\eta,\xi)$ is right cyclic if and only $(\phi,\id)(V + \langle \eta \rangle) \in \Grr^{> 0}_{k+1,n+1}$ and $(\id, (-1)^k \xi)(V) \in \Grr^{> 0}_{k,n+1}$.
\end{corollary}

\begin{proof}
By definition, $(V,\eta,\xi)$ is left cyclic if and only if $(V,\eta)$ is cyclic and $(\xi,\id)(V) \in \Grr^{> 0}_{k,n+1}$; by Corollary~\ref{cor:AnotherCyclicDefCHANGED}, this holds if and only if $(\id, (-1)^k \phi)(V + \langle \eta \rangle) \in \Grr^{> 0}_{k+1,n+1}$ and $(\xi,\id)(V) \in \Grr^{> 0}_{k,n+1}$. The argument in the right cyclic case is similar, using Proposition~\ref{prop:AlternateCyclicDef}.
\end{proof}

\begin{corollary}\label{cor:PolarizedCyclicAlternateDef2}
Given a polarized arrangement $(V,\eta,\xi)$, let $\phi$ be defined as in Proposition~\ref{prop:AlternateCyclicDef}. Then $(V,\eta,\xi)$ is left cyclic if and only if there exists a strong lift $\bar{\xi}$ of $\xi$ such that $(\bar{\xi},\id,(-1)^k \phi)(V + \langle \eta \rangle) \in \Grr_{k+1,n+2}^{> 0}$. Similarly, $(V,\eta,\xi)$ is right cyclic if and only if there exists a strong lift $\bar{\xi}$ of $\xi$ such that $(\phi, \id, (-1)^k \bar{\xi})(V + \langle \eta \rangle) \in \Grr_{k+1,n+2}^{> 0}$.
\end{corollary}

\begin{proof}
We will give a proof in the left cyclic case; the right cyclic case is similar. If a strong lift $\bar{\xi}$ exists as described, let $A', w'$ be representatives for $V,\eta$, and let $\left[\begin{array}{c|c} (x')^T & c \end{array} \right]$ be the matrix for $\bar{\xi}$ in the columns of $\left[ \begin{array}{c|c} A' & (-1)^k w' \end{array} \right]$. The maximal minors of
$\left[
\begin{array}{c|c}
(x')^T & c \\
\hline
A' & (-1)^k w' \\
\hline
0 & 1
\end{array}
\right]$ all have the same sign; it follows from Corollary~\ref{cor:PolarizedCyclicAlternateDef1} that $(V,\eta,\xi)$ is left cyclic.

Conversely, assume $(V,\eta,\xi)$ is left cyclic (with $A', w', x'$ chosen as above), and consider the above matrix with $c$ left unspecified. By assumption, maximal minors involving the bottom row all have the same sign, and maximal minors not involving the top row all have the same sign (the signs must thus agree in these two cases). Each of the finitely many maximal minors involving the top row, but not the bottom, can be written as $(-1)^k c$ times a maximal minor of $A'$ (all of which have the same sign), plus terms that are independent of $c$. Thus, for $c >> 0$ or $c << 0$, we can ensure that these minors have the same sign as the rest of the minors of this matrix, so a strong lift $\bar{\xi}$ exists as specified in the statement.
\end{proof}

\begin{corollary}
For a given $(n,k)$, the left cyclic and the right cyclic polarized arrangements $(V,\eta,\xi)$ form equivalence classes of polarized arrangements.
\end{corollary}

\begin{proof}
It follows from Corollary~\ref{cor:PolarizedCyclicAlternateDef2} that $(V,\eta,\xi)$ is left cyclic if and only if it has a strong lift $(V,\eta,\bar{\xi})$, with $V$ the column span of $A'$, $\eta$ represented by $w'$, and $\bar{\xi}$ having matrix $\left[ \begin{array}{c|c} (x')^T & c \end{array} \right]$ in the columns of $\left[ \begin{array}{c|c} A' & (-1)^k w'\end{array} \right]$, such that the maximal minors of
$\left[
\begin{array}{c|c}
(x')^T & c \\
\hline
A' & (-1)^k w' \\
\hline
0 & 1
\end{array}
\right]$
all have the same sign. This condition on the signs of maximal minors is equivalent to a condition on the signs of maximal minors of the matrix
$\left[
\begin{array}{c|c}
A' & w' \\
\hline
(x')^T & c \\
\hline
0 & 1
\end{array}
\right]$
that specifies an equivalence class of strong polarized arrangements, and thus an equivalence class of polarized arrangements.
\end{proof}

\subsection{Vandermonde arrangements}\label{sec:Vandermonde}

Let $z_1 < \cdots < z_n \in \R \subset \C$ and let $V$ be the column span of the Vandermonde matrix
\begin{equation} \label{eq:VandermondeM}
\begin{bmatrix}
1 & z_1 & \cdots & z_1^{k-1} \\
1 & z_2 & \cdots & z_2^{k-1} \\
\vdots & & & \\
1 & z_n & \cdots & z_n^{k-1}
\end{bmatrix}.
\end{equation}
Let $\eta$ be the element of $\R^n / V$ represented by $w := (-1)^k (z_1^k, \ldots, z_n^k)$; then $(V,\eta)$ is an arrangement. If all the $z_i$ are rational then $(V,\eta)$ is rational; one can check that $(V,\eta)$ is cyclic using Section~\ref{sec:AlternateCyclic} or the proof of \cite[Proposition 6.8]{KarpWilliams}. If we write $v_1,\ldots,v_k$ for the columns of the above matrix and identify $V + \eta$ with $\R^k$ by sending $(a_1,\ldots,a_k) \in \R^k$ to $a_1 v_1 + \cdots + a_k v_k + w \in V + \eta$, then the $i^{th}$ hyperplane of the arrangement has equation
\[
a_1 + a_2 z_i + \cdots + a_k z_i^{k-1} + (-z_i)^k = 0.
\]

It follows from Section~\ref{sec:AlternateCyclic} that all cyclic arrangements are equivalent to ones arising from this Vandermonde construction in the sense defined above, and that given $n$ and $k$ they are all equivalent to each other (i.e. the choice of $z_i$ does not matter up to equivalence).

We now give a polarized analogue of this construction. Given points $z_0 < z_1 < \cdots < z_n$ in $\R \subset \C$, we define a left cyclic polarized arrangement $(V,\eta,\xi)$ with $(V,\eta)$ obtained from $z_1, \ldots, z_n$ as above. We let $\xi$ be the linear functional on $V$ whose matrix in the columns of the above Vandermonde matrix is $\begin{bmatrix} 1 & z_0 & \cdots & z_0^{k-1} \end{bmatrix}$. One can check that $(V,\eta,\xi)$ is left cyclic. Similarly, given $z_1 < \cdots < z_n < z_{n+1}$ in $\R \subset \C$, we define a right cyclic polarized arrangement $(V,\eta,\xi)$ where $(V,\eta)$ are defined as above and $\xi$ has matrix $(-1)^k \begin{bmatrix} 1 & z_{n+1} & \cdots & z_{n+1}^{k-1} \end{bmatrix}$. Again, all left and right cyclic polarized arrangements are equivalent to Vandermonde ones, which (given $n$, $k$, and a choice of left versus right) are all equivalent to each other.

\subsection{Symmetric powers}\label{sec:SymmetricPowers}

\subsubsection{Cyclic arrangements and \texorpdfstring{$\Sym^k(\C)$}{Sym-k(C)}}\label{sec:CyclicAndSymPowers}

Besides their relationship to amplituhedra, cyclic arrangements are also special in that their complexified complements $\cal{X}_{\cal{V}}$ are symmetric products of the punctured plane, as we explain below. While much of the material in this section is standard, we give a detailed exposition due to its conceptual importance in understanding our results.

\begin{proposition}
When $(V,\eta)$ is the Vandermonde arrangement of $n$ points $z_1 < \cdots < z_n$ in $\R \subset \C$ from Section~\ref{sec:Vandermonde}, the map from $\C^k$ to $\Sym^k(\C)$ sending $(a_1,\ldots,a_n)$ to the multi-set of roots of the polynomial $f(z) = (-z)^k + a_k z^{k-1} + \cdots + a_2 z + a_1$ restricts to a bijection from the complexified complement of $(V,\eta)$ (a subset of $(V + \eta)_{\C} \cong (\R^k)_{\C} = \C^k$) to $\Sym^k(\C \setminus \{z_1,\ldots,z_n\})$.
\end{proposition}

\begin{proof}
We can identify $\Sym^k(\C)$ with the space of degree $k$ complex polynomials in a single variable $z$, with leading term $(-z)^k$ for reasons we will see below, by sending a polynomial to its (unordered) multi-set of roots. The subset $\Sym^k(\C \setminus \{z_1,\ldots,z_n\})$ of $\Sym^k(\C)$ gets identified with those polynomials that do not vanish at $z_1, \ldots, z_n$. On the other hand, the same set of degree $k$ complex polynomials can be identified with $\C^k$ by sending a polynomial $f(z) = (-z)^k + a_k z^{k-1} + \cdots + a_2 z + a_1$ to its vector of coefficients. Under this identification, the polynomials $f$ vanishing at $z_i$ correspond to the coefficient vectors $(a_1,\ldots,a_k)$ satisfying the equation
\[
(-z_i)^k + a_k z_i^{k-1} + \cdots + a_2 z_i + a_1 = 0,
\]
the complexified equation for the $i^{th}$ hyperplane of the Vandermonde arrangement $(V,\eta)$. Thus, the complexified complement of $(V,\eta)$ is identified with polynomials not vanishing at any $z_i$, and thus with $\Sym^k(\C \setminus \{z_1,\ldots,z_n\})$.
\end{proof}

For any cyclic arrangement $(V,\eta)$, an equivalence of $(V,\eta)$ with the Vandermonde arrangement for $z_1 < \cdots < z_n$ gives an identification of the complement of the complexification of $(V,\eta)$ with $\Sym^k(\C \setminus \{z_1,\ldots,z_n\})$.

\subsubsection{Distinguished Lagrangians}\label{sec:SymProdLagrangians}

Recall from Section~\ref{sec:AlgebrasGeometricAspects} that for an arrangement $(V,\eta)$, the complexified complement of $(V,\eta)$ has a distinguished family of noncompact Lagrangians given by the interiors of the compact regions of $(V,\eta)$, and that given a polarization $(V,\eta,\xi)$, we get a larger family consisting of interiors of bounded feasible regions. When $(V,\eta)$ is cyclic (resp. $(V,\eta,\xi)$ is left or right cyclic), we can view the interiors of compact (resp. bounded feasible) regions as Lagrangians in $\Sym^k(\C \setminus \{z_1,\ldots,z_n\})$.

\begin{proposition}\label{prop:LagrangiansCorrespond}
For cyclic $(V,\eta)$, the Lagrangians in $\Sym^k(\C \setminus \{z_1,\ldots,z_n\})$ given by interiors of compact regions of $(V,\eta)$ are the symmetric products of the straight-line Lagrangians connecting $z_i$ to $z_{i+1}$ for $1 \leq i \leq n-1$. For left cyclic $(V,\eta,\xi)$, the Lagrangians in $\Sym^k(\C \setminus \{z_1,\ldots,z_n\})$ given by interiors of bounded feasible regions are the same as above, except that one includes the straight-line Lagrangian between $-\infty$ and $z_1$ in $\R \subset \C$. For right cyclic $(V,\eta,\xi)$, one includes the straight-line Lagrangian between $z_n$ and $+\infty$ instead.
\end{proposition}

\begin{proof}
Let $(V,\eta)$ be cyclic and let $\alpha \in \cal{K}$; by Proposition~\ref{prop:SignPatternCpct}, we have $\var(\alpha) = k$ and $\alpha$ starts with a plus. We can view points in the interior of $\Delta_{\alpha}$ as polynomials $f(z) = (-z)^k + a_k z^{k-1} + \cdots + a_2 z + a_1$, with real coefficients, such that $\alpha(i) f(z_i) > 0$ for all $i$. Such polynomials $f$ have $k$ sign changes on the real axis because $\var(\alpha) = k$, so they have $k$ real roots. More precisely, if $\alpha$ changes sign after index $i_j$ for $1 \leq i_1 < \cdots < i_k < n$, then $f$ has a root between $z_{i_j}$ and $z_{i_j + 1}$ for $1 \leq j \leq k$. It follows that the multi-set of roots of $f$ lies in the symmetric product of straight lines from $z_{i_j}$ to $z_{i_j + 1}$ inside $\Sym^k(\C \setminus \{z_1,\ldots,z_n\})$ for $1 \leq j \leq k$.

Conversely, assume $f(z) = (-z)^k + a_k z^{k-1} + \cdots + a_2 z + a_1$ for arbitrary complex coefficients $a_i$, with $f(z_i) \neq 0$ for all $i$, and that the multi-set of roots of $f$ lies in the symmetric product of straight lines from $z_{i_j}$ to $z_{i_j + 1}$ for some $1 \leq i_1 < \cdots < i_k < n$. For $1 \leq i \leq n$, let $\alpha(i)$ denote the sign of $f(z_i)$; then $(a_1,\ldots,a_k) \in \Delta_{\alpha}$ and we have $\var(\alpha) = k$. Furthermore, $f(z_1) = (r_1 - z_1) \cdots (r_k - z_1) > 0$ since each root $r_j$ is greater than $z_1$, so the first sign of $\alpha$ is a plus. Thus, $\alpha \in \cal{K}$.

Now let $(V,\eta,\xi)$ be left cyclic and let $\alpha \in \cal{P}$; by Corollary~\ref{cor:LeftCyclicBoundedFeasible}, we have $\var_l(\alpha) = k$. For a point in the interior of $\Delta_{\alpha}$ viewed as a polynomial $f$ as above, if $\alpha$ starts with a plus then the above argument goes through and $f$ is in a symmetric product of finite-length straight lines. If $\alpha$ starts with a minus, then $\var_l(\alpha) = k-1$, so there exist $1 \leq i_1 < \cdots < i_{k-1} < n$ such that $f$ has a root between $z_{i_j}$ and $z_{i_j + 1}$ for $1 \leq j \leq k-1$. Furthermore, since $\lim_{z \to -\infty \in \R} f(z) = +\infty$ but $f(z_1) < 0$, $f$ must have a root on the real axis to the left of $z_1$. Thus, the multi-set of roots of $f$ lies in a member of the extended family of symmetric-product Lagrangians in $\Sym^k(\C \setminus \{z_1,\ldots,z_n\})$ from the statement. Conversely, if the multi-set of roots of $f$ lies in one of these Lagrangians, then $f$ has real coefficients and the region $\Delta_{\alpha}$ containing $(a_1,\ldots, a_k)$ satisfies $\var_l(\alpha) = k$; we thereby have $\alpha \in \cal{P}$.

Finally, let $(V,\eta,\xi)$ be right cyclic and let $\alpha \in \cal{P}$; we have $\var_r(\alpha) = k$. We consider two cases depending on whether the last sign of $\alpha$ is $\pm (-1)^k$; if it is $(-1)^k$, then $\var(\alpha) = \var_r(\alpha) = k$ and the argument proceeds as usual. If the last sign of $\alpha$ is $-(-1)^k$, note that $\lim_{z \to \infty \in \R} f(z) = (-1)^k$; the argument proceeds as before.
\end{proof}

\subsection{Dots in regions and partial orders}

\subsubsection{Sign sequences and dots in regions}\label{sec:SignSeqAndOSzIdems}

We introduce an alternate combinatorial description of bounded feasible sign sequences $\alpha$ in the left and right cyclic cases, closely mirroring the structure of the straight-line Lagrangians in Proposition~\ref{prop:LagrangiansCorrespond}. Let $V_l(n,k)$ denote the set of $k$-element subsets of $\{0,\ldots,n-1\}$ and let $V_r(n,k)$ denote the set of $k$-element subsets of $\{1,\ldots,n\}$. Following Ozsv{\'a}th--Szab{\'o} (see Section~\ref{sec:OSz} below), we draw elements of $V_l(n,k)$ as sets of $k$ dots in the regions $\{0,\ldots,n-1\}$ on the left of Figure~\ref{fig:Idems}; see the right of Figure~\ref{fig:Idems} for an example. Elements of $V_r(n,k)$ are drawn similarly as sets of $k$ dots in the regions $\{1,\ldots,n\}$ on the left of Figure~\ref{fig:Idems}.

\begin{figure}
\includegraphics[scale=0.62]{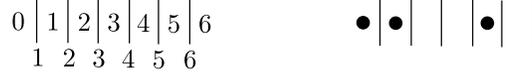}
\caption{Left: $n$ lines and $n+1$ regions between them. Right: a set of $3$ dots in the regions $\{0,1,4\} \subset \{0,\ldots,6\}$.}
\label{fig:Idems}
\end{figure}

\begin{definition}\label{def:StatesFromAlpha} \hfill
\begin{enumerate}[(i.)]
  \item \label{def:LeftIStatesFromAlpha} Let $\cal{V}$ be a left cyclic polarized arrangement so that $\cal{P}$ is the set of $\alpha \in \{+,-\}^n$ with $\var_l(\alpha) = k$.  There is a bijection  $\kappa_l\maps \cal{P}\to V_l(n,k)$ given by sending $\alpha \in \cal{P}$ to $\x_{\alpha} \subset \{0,\ldots,n-1\}$ with  $i \in \x_{\alpha}$ if there is a change after sign $i+1$ in the sequence $+\alpha$; the inverse sends  $\x \in V_l(n,k)$ to $\alpha_{\x}$, defined from $+\alpha_{\x}$   by starting with a $+$ (``step zero'') and writing $n$ signs to the right, introducing a sign change at step $i$ if and only if $i-1 \in \x$.

      \item \label{def:RightIStatesFromAlpha}Let $\cal{V}$ be a right cyclic polarized arrangement so that $\cal{P}$ is the set of $\alpha \in \{+,-\}^n$ with $\var_r(\alpha) = k$.  There is a bijection
           $\kappa_r\maps \cal{P} \to V_r(n,k)$ sending $\alpha$ to $\x_{\alpha} \subset \{1,\ldots,n\}$ defined by $i \in \x_{\alpha}$ if there is a change after sign $i$  in the sequence $\alpha(-1)^k$; the inverse sends $\x \in V_r(n,k)$ to $\alpha_{\x}$, defined from   $\alpha_{\x}(-1)^k$   by starting with $(-1)^k$ as the rightmost entry (``step zero'') and writing $n$ signs from right to left, introducing a sign change at step $i$ if and only if $n-i+1 \in \x$.
\end{enumerate}
\end{definition}

Let $V'(n,k)$ denote the set of $k$-element subsets of $\{1,\ldots,n-1\}$; we have $V'(n,k) \subset V_l(n,k), V_r(n,k)$. The above constructions give a bijection between $V'(n,k)$ and the set of $\alpha \in \cal{K}$ for a cyclic arrangement. In all cases (cyclic, left cyclic, and right cyclic), Proposition~\ref{prop:LagrangiansCorrespond} identifies the interior-of-region Lagrangian for a given $\alpha$ with the symmetric-product Lagrangian for $\x_{\alpha}$.

\subsubsection{The partial order for cyclic arrangements} \label{sec:partial-dots}

Let $z_0 < z_1 < \cdots < z_n \in \R \subset \C$, and let $(V,\eta,\xi)$ be the associated left cyclic Vandermonde arrangement from Section~\ref{sec:LRCyclicDefs}. We have a partial order on the set $\cal{P}$ of bounded feasible sign sequences for $(V,\eta,\xi)$ from Section~\ref{subsec:partial_order}. Identifying $\cal{P}$ with the set $V_l(n,k)$ of $k$-element subsets of $\{0,\ldots,n-1\}$ from Section~\ref{sec:SignSeqAndOSzIdems}, viewed as sets of dots in regions, we also have the lexicographic partial order on $V_l(n,k)$ generated by the relations $\x < \y$ when $\y$ is obtained from $\x$ by moving a dot one step to the right.

\begin{proposition}\label{prop:VandermondeOrders}
 For a left cyclic Vandermonde arrangement, the partial order on $\cal{P}$ induced by $\xi$ agrees with the order induced from the lexicographic order on $V_l(n,k)$ from the bijection $\kappa_l \maps \cal{P} \to V_l(n,k)$.
\end{proposition}

\begin{proof}
Let $\alpha \in \cal{P}$, and identify $V$ with $\R^k$ using the columns of the Vandermonde matrix from \eqref{eq:VandermondeM}.
 The points $(a_1,\ldots,a_k) \in \R^k$ lying in the region $\Delta_{\alpha}$ are those satisfying the inequalities
\[
\alpha(i) \left( (-z_i)^k + a_k z_i^{k-1} + \cdots + a_2 z_i + a_1 \right) \geq 0.
\]
On the other hand, we can view $(a_1,\ldots,a_k)$ as a degree-$k$ polynomial $f_{a_1,\ldots,a_k}(z) = (-z)^k + a_k z^{k-1} + \cdots + a_2 z + a_1$ of one variable $z$, and this identification is a bijection between $\R^k$ and the set of real-coefficient degree-$k$ polynomials in $z$ with leading term $(-z)^k$. Under this identification, $\Delta_{\alpha}$ is the set of such polynomials $f$ such that $\alpha(i) f(z_i) \geq 0$ for all $i$, i.e. that either $f(z_i) = 0$ or the sign of $f(z_i)$ is $\alpha_i$ for all $i$. The interior of $\Delta_{\alpha}$ is the set of $f$ such that $f(z_i)$ is nonzero and has the same sign as $\alpha_i$ for all $i$. Note that if $f$ is in the interior of $\Delta_{\alpha}$, then since $\var(\alpha) = k$, all roots of $f$ are real and lie in the regions between the $z_i$ coming from the element $\x$ of $V_l(n,k)$ corresponding to $\alpha$ (see Figure~\ref{fig:Idems}). Thus, if $\x = \{i_1 < \cdots < i_k\}$, then we can write $f(z) = (r_1 - z) \cdots (r_k - z)$ where
\begin{itemize}
\item if $i_1 = 0$, then $r_1 < z_1$;
\item if $i_j \neq 0$, then $z_{i_j} < r_j < z_{i_j + 1}$.
\end{itemize}

Now, up to an additive constant that does not depend on $f$, the value of $\xi$ at $f$ is the evaluation $f(z_0)$. Taking the constant to be zero, the maximum value attained by $\xi$ on $\Delta_{\alpha}$ is the supremum of $f(z_0)$ over all $f$ in the interior of $\Delta_{\alpha}$. By the above, we have $f(z_0) = (r_1 - z_0)\cdots(r_k - z_0)$. This quantity approaches its supremum (over the interior of $\Delta_{\alpha}$) as $r_j \to z_{i_j}$ for all $j$, so the supremum is $(z_{i_1} - z_0) \cdots (z_{i_k} - z_0)$. Thus, for $\alpha \leftrightarrow \beta$ (corresponding to $\x,\y$ such that $\y$ is obtained from $\x$ by moving a dot one step), we have $\alpha < \beta$ if and only if $\y$ is obtained from $\x$ by increasing the value of $i_j$ by one for some $j$ (i.e. moving a dot of $\x$ one step to the right), and the proposition follows.
\end{proof}

An analogous result holds in the right cyclic case with the following modifications. We start with $z_1 < \cdots < z_n < z_{n+1}$. For $f$ in the interior of $\Delta_{\alpha}$ we have $f(z) = (r_1 - z) \cdots (r_k - z)$ where
\begin{itemize}
\item if $i_k = n$, then $r_k > z_n$;
\item if $i_j \neq n$, then $z_{i_j} < r_j < z_{i_j + 1}$
\end{itemize}
(analogously to above, we let $\alpha$ correspond to $\x = \{i_1 < \cdots < i_k\} \subset \{1,\ldots,n\}$). The value of $\xi$ at $f$ is $(-1)^k f(z_{n+1})$ up to an additive constant; the supremum of this value over $f$ in the interior of $\Delta_{\alpha}$ is $(z_{n+1} - z_{i_1}) \cdots (z_{n+1} - z_{i_k})$. We conclude that for $\alpha \leftrightarrow \beta$ corresponding to $\x,\y \subset \{1,\ldots,n\}$, we have $\alpha < \beta$ if and only if $\y$ is obtained from $\x$ by moving a dot one step to the left.

The proof of Proposition~\ref{prop:VandermondeOrders} also lets us conclude that our bijections $\cal{P} \leftrightarrow V_l(n,k)$ from Section~\ref{sec:SignSeqAndOSzIdems} and $\cal{P} \leftrightarrow \mathbb{B}$ from Section~\ref{subsec:partial_order} are related straightforwardly.

\begin{corollary}\label{cor:BijectionsAgree}
If $(V,\eta,\xi)$ is the left cyclic Vandermonde arrangement associated to $z_0 < z_1 < \cdots < z_n \in \R \subset \C$ and we have $\alpha \in \cal{P}$, then the element $\mathbbm{x} \in \mathbb{B}$ associated to $\alpha$ in Definition~\ref{def:StatesFromAlpha} is obtained from the element $\x \in V_l(n,k)$ associated to $\alpha$ in Section~\ref{subsec:partial_order} by adding one to each $i \in \x$.
\end{corollary}

\begin{proof}
By definition, $\mathbbm{x}$ is the set of indices of the $k$ hyperplanes at whose (unique) intersection point the functional $\xi$ takes its maximum value on $\Delta_{\alpha}$. By the proof of Proposition~\ref{prop:VandermondeOrders}, the point of $\Delta_{\alpha}$ maximizing the value of $\xi$ corresponds to a polynomial $f$ whose roots lie at the right endpoints of the regions containing the dots of $\x$. The hyperplane $H_i$ consists of those polynomials vanishing at $z_i$, and the right endpoint of a region labeled $i \in \{0,\ldots,n-1\}$ is $i+1$.
\end{proof}

If $(V,\eta,\xi)$ is a right cyclic Vandermonde arrangement, one can show similarly that the elements $\mathbbm{x} \in \mathbb{B}$ and $\x \in V_r(n,k)$ associated to $\alpha$ agree as subsets of $\{1,\ldots,n\}$.

Proposition~\ref{prop:VandermondeOrders} and Corollary~\ref{cor:BijectionsAgree} (resp. their right cyclic analogues) hold for general left cyclic (resp. right cyclic) polarized arrangements, since their claims are preserved under equivalence and all left and right cyclic polarized arrangements are equivalent to Vandermonde ones.

\subsection{Cyclicity and Gale duality}\label{sec:AltGale}

For a polarized arrangement $(V,\eta,\xi)$, its Gale dual is $(V^{\perp}, -\xi, -\eta)$, so its alt Gale dual is $(\alt(V^{\perp}), -\alt(\xi), -\alt(\eta))$. Thus, the polarization reversal of its alt Gale dual is $(\alt(V^{\perp}), -\alt(\xi), \alt(\eta))$. Similarly, the alt Gale dual of its polarization reverse is $(\alt(V^{\perp}), \alt(\xi), -\alt(\eta))$. Note that alt commutes with Gale duality and polarization reversal; the relevant question is the ordering of Gale duality and polarization reversal.

\begin{theorem}\label{thm:AltGale}
A polarized arrangement $(V,\eta,\xi)$ is right cyclic if and only if the polarization reversal of its alt Gale dual $(\alt(V^{\perp}), -\alt(\xi), \alt(\eta))$ is left cyclic.
\end{theorem}

\begin{proof}
Given either condition we have $V \in \Grr^{>0}_{k,n}$, so there exists a unique totally positive matrix $A$ of size $(n-k) \times k$ such that $V$ is the column span of the matrix
$\left[
\begin{array}{c}
I_k \\
\hline
(-1)^{j+k} \widetilde{A}
\end{array}
\right]$. Let $w \in \R^{n-k}$ be the unique vector such that $\left[ \begin{array}{c} 0 \\ \hline (-1)^k w \end{array} \right]$ represents $\eta \in \R^n / V$. Let $[(-1)^{j} x^T]$ be the matrix of $\xi$ in the columns of the matrix representing $V$.
Then
\[
(\phi,\id)(V + \langle \eta \rangle) = {\rm colspan}\left[
\begin{array}{c|c}
1 & 0 \\
\hline
0 & I_k \\
\hline
(-1)^k w & (-1)^{j+k} \widetilde{A}
\end{array}
\right], \qquad
(\id, (-1)^k \xi)(V) ={\rm colspan} \left[
\begin{array}{c}
I_k \\
\hline
(-1)^{j+k} \widetilde{A} \\
\hline
(-1)^{j+k} x^T
\end{array}
\right].
\]

Note that $V^{\perp}$ can be viewed as the column span of
$\left[
\begin{array}{c}
(-1)^{i+k} (\widetilde{A})^T \\
\hline
-I_{n-k}
\end{array}
\right]$
where $()^T$ denotes the transpose (we are viewing elements of both $\R^n$ and $(\R^n)^*$ as column vectors). We can multiply column $j$ by $(-1)^{k+j}$ to view $V^{\perp}$ as the column span of
$\left[
\begin{array}{c}
(-1)^{i+j} (\widetilde{A})^T \\
\hline
(-1)^{k+j-1} I_{n-k}
\end{array}
\right]$. Thus, $\alt(V^{\perp})$ is the column span of
$\left[
\begin{array}{c}
(-1)^{j-1} (\widetilde{A})^T \\
\hline
I_{n-k}
\end{array}
\right]$.

Since
$\left[
\begin{array}{c}
0 \\
\hline
(-1)^k w
\end{array}
\right]$
represents $\eta \in \R^n/V$, we can take
$\left[
\begin{array}{c}
0 \\
\hline
(-1)^{i-1} w
\end{array}
\right]$ to represent $\alt(\eta)$. The matrix of $\alt(\eta)$ as a linear functional on $\alt(V^{\perp})$, in the basis for $\alt(V^{\perp})$ given by columns of the above matrix, is thus
$\left[
\begin{array}{c}
(-1)^{j-1} w^T
\end{array}
\right]$. We see that $(\alt(\eta),\id)(\alt(V^{\perp}))$ is the column span of
$\left[
\begin{array}{c}
(-1)^{j-1} w^T \\
\hline
(-1)^{j-1} (\widetilde{A})^T \\
\hline
I_{n-k}
\end{array}
\right]$.

The vector
$\left[
\begin{array}{c}
(-1)^i x \\
\hline
0
\end{array}
\right]$
in $(\R^n)^*$ represents $\xi \in (\R^n)^*/V^{\perp}$; indeed, dot products of this vector with the columns of
$\left[
\begin{array}{c}
I_k \\
\hline
(-1)^{j+k} \widetilde{A}
\end{array}
\right]$
give the matrix
$\left[
\begin{array}{c}
(-1)^{j} x^T
\end{array}
\right]$
for $\xi$ in this basis for $V$. It follows that the vector
$\left[
\begin{array}{c}
x \\
\hline
0
\end{array}
\right]$
represents $-\alt(\xi) \in (\R^n)^*/V^{\perp}$, so
\[
(\id, (-1)^{n-k} \phi)(\alt(V^{\perp}) + \langle -\alt(\xi) \rangle)
= {\rm colspan}
\left[
\begin{array}{c|c}
(-1)^{j-1} \widetilde{A}^T & (-1)^{n-k} x \\
\hline
I_{n-k} & 0 \\
\hline
0 & 1
\end{array}
\right].
\]

By Corollary~\ref{cor:PolarizedCyclicAlternateDef1} and the above setup, $(V,\eta,\xi)$ is right cyclic if and only if the maximal minors of
$\left[
\begin{array}{c|c}
1 & 0 \\
\hline
0 & I_k \\
\hline
(-1)^k w & (-1)^{j+k} \widetilde{A}
\end{array}
\right]$
and of
$\left[
\begin{array}{c}
I_k \\
\hline
(-1)^{j+k} \widetilde{A} \\
\hline
(-1)^{j+k} x^T
\end{array}
\right]$
are all positive, while $(\alt(V^{\perp}), -\alt(\xi), \alt(\eta))$ is left cyclic if and only if the maximal minors of
$\left[
\begin{array}{c}
(-1)^{j-1} w^T \\
\hline
(-1)^{j-1} (\widetilde{A})^T \\
\hline
I_{n-k}
\end{array}
\right]$
and
$\left[
\begin{array}{c|c}
(-1)^{j-1} \widetilde{A}^T & (-1)^{n-k} x \\
\hline
I_{n-k} & 0 \\
\hline
0 & 1
\end{array}
\right]$
are all positive.

It suffices to show that the column spans of the matrices for $(V,\eta,\xi)$ are the alt perpendiculars of the column spans of the matrices for $(\alt(V^{\perp}), -\alt(\xi), \alt(\eta))$. Indeed, the perpendicular of $(\alt(\eta),\id)(\alt(V^{\perp})$ is the column span of
$\left[
\begin{array}{c|c}
-1 & 0 \\
\hline
0 & -I_k \\
\hline
(-1)^{i-1}w & (-1)^{i-1} \widetilde{A}
\end{array}
\right]$, or equivalently of
$\left[
\begin{array}{c|c}
1 & 0 \\
\hline
0 & (-1)^j I_k \\
\hline
(-1)^{i}w & (-1)^{i+j} \widetilde{A}
\end{array}
\right]$, so its alt-perpendicular is the column span of
$\left[
\begin{array}{c|c}
1 & 0 \\
\hline
0 & I_k \\
\hline
(-1)^{k} w & (-1)^{j+k} \widetilde{A}
\end{array}
\right]$. Similarly, the perpendicular of $(\id,(-1)^k\xi)(V)$ is the column span of
$\left[
\begin{array}{c|c}
(-1)^{i+k} \widetilde{A}^T & (-1)^{i+k} x \\
\hline
-I_{n-k} & 0 \\
\hline
0 & -1
\end{array}
\right]$
or equivalently of
$\left[
\begin{array}{c|c}
(-1)^{i+j} \widetilde{A}^T & (-1)^{i+n-k+1} x \\
\hline
(-1)^{k+j-1}I_{n-k} & 0 \\
\hline
0 & (-1)^n
\end{array}
\right]$, so the alt perpendicular of $(\id,(-1)^k\xi)(V)$ is the column span of
$\left[
\begin{array}{c|c}
(-1)^{j-1} \widetilde{A}^T & (-1)^{n-k} x \\
\hline
I_{n-k} & 0 \\
\hline
0 & 1
\end{array}
\right]$.

\end{proof}

\begin{corollary}
The algebras $A(\cal{V})$ and $B(\cal{V})$ for right cyclic polarized arrangements $\cal{V}$ are Koszul dual to the Ringel duals of the algebras $A(\cal{V}')$ and $B(\cal{V}')$ respectively for left cyclic polarized arrangements $\cal{V}'$. Equivalently, $A(\cal{V})$ is Ringel dual to $B(\cal{V}')$ and $B(\cal{V})$ is Ringel dual to $A(\cal{V}')$.
\end{corollary}

\subsection{Deletions and restrictions of cyclic arrangements}\label{sec:DelRestCyclic}

\begin{definition}\label{def:AltRes}
Let $\cal{V}^i = (V^i,\eta^i,\xi^i)$ be obtained from $\cal{V}=(V,\eta,\xi)$ by restricting to the $i^{th}$ hyperplane.  The \emph{signed restriction} $\cal{V}^i_{{\rm sign}}=(V^i_{{\rm sign}},\eta^i_{{\rm sign}},\xi^i_{{\rm sign}})$ of $\cal{V}$ is the polarized arrangement obtained from $\cal{V}^i$ by applying the automorphisms of $\R^{n-1}$ and $(\R^{n-1})^*$ represented by the diagonal matrix with $j^{th}$ entry $1$ if $j < i$ and $-1$ if $j \geq i$.  We say that $\cal{V}^i_{{\rm sign}}$ is obtained by signed-restricting $\cal{V}$ to the $i^{th}$ hyperplane.  We use the same terminology in the unpolarized case, and refer to
$(V^i_{{\rm sign}},\eta^i_{{\rm sign}})$ as the signed restriction of $(V,\eta)$.
\end{definition}

\begin{lemma}  Let $\cal{V}=(V,\eta,\xi)$ be a left cyclic (resp. right cyclic) arrangement. \hfill
\begin{enumerate}[(i)]
  \item \label{lem:DeletedCyclicIsCyclic}
 Let $\cal{V}_i = (V_i,\eta_i,\xi_i)$ be obtained from $\cal{V}$ by deleting the $i^{th}$ hyperplane as in Section~\ref{sec:GeneralDeletionRestriction}. Then $\cal{V}_i$ is left cyclic (resp. right cyclic).

\item \label{lem:RestrictedCyclicIsCyclic}
Let $\cal{V}^i_{{\rm sign}}$ be the arrangement obtained by signed-restricting $\cal{V}$ to the $i^{th}$ hyperplane (see Definition \ref{def:AltRes}).  Then $\cal{V}^i_{{\rm sign}}$ is left cyclic (resp. right cyclic).
\end{enumerate}
\end{lemma}

\begin{proof}
The proof of part~\eqref{lem:DeletedCyclicIsCyclic} is straightforward; we will prove part~\eqref{lem:RestrictedCyclicIsCyclic}. First assume $\V = (V,\eta,\xi)$ is left cyclic; then the alt Gale dual of the polarization reversal of $\V$ (namely $\alt(p(\cal{V})^{\vee})= (V^{\perp}, \alt(\xi), - \alt(\eta))$) is right cyclic. The $i^{th}$ restriction of the polarization reversal of $\V$ is the polarization reversal of the $i^{th}$ restriction of $\V$, i.e $p(\cal{V})^i = p(\cal{V}^i)$.
The Gale dual of this is the $i^{th}$ deletion of the Gale dual of the polarization reversal of $\V$, so that  $(p(\cal{V})^i)^{\vee} = p(\cal{V}^i)^{\vee} = (p(\cal{V})^{\vee})_i$.  Thus, the alt Gale dual of the polarization reversal of the $i^{th}$ restriction of $\V$ is alt of the $i^{th}$ deletion of the Gale dual of the polarization reversal of $\V$, i.e. $\alt(p(\cal{V}^i)^{\vee}) = \alt((p(\cal{V})^{\vee})_i)$. Alt of a deletion and deletion of an alt are related by a sign-change automorphism as in the definition of the signed restriction, so
\[
\alt(p(\cal{V}_{\sign}^i)^{\vee}) = (\alt(p(\cal{V})^{\vee}))_i.
\]
By part~\eqref{lem:DeletedCyclicIsCyclic} and the above, the alt Gale dual $\alt(p(\cal{V}_{\sign}^i)^{\vee})$ of the polarization reversal of the $i^{th}$ signed restriction of $\V$ is right cyclic, so the signed restriction $\cal{V}_{\sign}^i$ of $\V$ in the statement of the lemma is left cyclic. The case when $\V$ is right cyclic, rather than left cyclic, is similar.
\end{proof}

Note that the algebras associated to $\cal{V}^i_{\sign}$ are naturally isomorphic to those associated to $\cal{V}^i$. Pick an ordered basis for $V^i$ and a representative in $\R^{n-1}$ for $\eta'$, so that we can view $\cal{V}^i$ as a polarized arrangement of $n-1$ affine hyperplanes in $\R^{k-1}$. Define an ordered basis for $V^i_{\sign}$ and representative for $\eta^i_{\sign}$ by sign-changing the basis vectors for $V^i$ and the representative of $\eta^i$ as in Definition \ref{def:AltRes}, so that we can also view $\cal{V}^i_{\sign}$ as a polarized arrangement of $n-1$ affine hyperplanes in $\R^{k-1}$. The above automorphism of $\R^{n-1}$ and $(\R^{n-1})^*$, sending $\cal{V}^i$ to $\cal{V}^i_{\sign}$, has the effect of reversing the co-orientation on the hyperplanes $H_j$ of $\cal{V}^i$ for $j \geq i$ (coming from hyperplanes $H_j$ of the original arrangement $\cal{V}$ with $j > i$) while keeping the direction of $\xi^i$ unchanged.

The analogous statements in the unpolarized case are also true.

\begin{lemma}
Let $(V,\eta)$ be a cyclic arrangement.
\begin{enumerate}[(i)]
\item Let $(V_i,\eta_i)$ be obtained from $(V,\eta)$ by deleting the $i^{th}$ hyperplane. Then $(V_i,\eta_i)$ is cyclic.
\item Let $(V^i_{\sign},\eta^i_{\sign})$ be obtained by signed-restricting to the $i^{th}$ hyperplane. Then $(V^i_{\sign},\eta^i_{\sign})$ is cyclic.
\end{enumerate}
\end{lemma}

\subsection{Examples}

\subsubsection{$k=1$}

Choosing real numbers $z_1 < \cdots < z_n$, the $k=1$ Vandermonde arrangement $(V,\eta)$ associated to these numbers has $V$ given by the column span of $\begin{bmatrix} 1 \\ \vdots \\ 1 \end{bmatrix}$ and $\eta$ represented by $\begin{bmatrix} -z_1 \\ \vdots \\ -z_n \end{bmatrix}$. Identifying $V + \eta$ with $\R^1$ using this data, the hyperplane $H_i$ of the arrangement $\cal{H}_{\cal{V}}$ has equation $x = z_i$ with positive region $x > z_i$. Define $\xi_l$ by choosing $z_0 < z_1$, so that $\xi_l$ has matrix $\begin{bmatrix} 1 \end{bmatrix}$. Define $\xi_r$ by choosing $z_{n+1} > z_n$, so that $\xi_r$ has matrix $\begin{bmatrix} -1 \end{bmatrix}$.

\begin{figure}
\includegraphics[width=\textwidth]{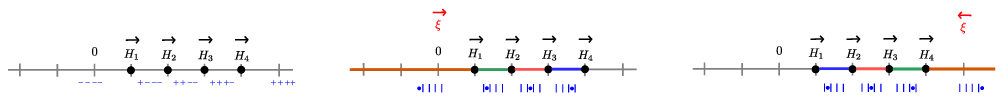}
\caption{A cyclic arrangement with left cyclic and right cyclic polarizations for $n=4, k=1$.}
\label{fig:n4k1}
\end{figure}

The arrangement $\cal{H}_{\cal{V}} \subset \R^1$ is shown in Figure~\ref{fig:n4k1} (left) for $n=4$ and $z_i = i$ for all $i$. The figure also shows the regions $\Delta_{\alpha}$ for $\alpha \in \cal{F}$, labeled by their sign sequences $\alpha$. The middle picture of Figure~\ref{fig:n4k1} indicates the left cyclic polarization arising from the choice of $z_0 = 0 < z_1$; the regions $\Delta_{\alpha}$ for $\alpha \in \cal{P}$ are colored and labeled by sets of dots in regions. The right picture of Figure~\ref{fig:n4k1} does the same for the right cyclic polarization arising from the choice of $z_{5} = 5 > z_4$.

When $\cal{V}$ is a cyclic arrangement with $k=1$, the hypertoric variety $\M_{\cal{V}}$ is isomorphic to the Milnor fiber of the type $A_{n-1}$ Kleinian singularity $\C^2 / \Z_n$; this is the family of varieties studied by Gibbons--Hawking \cite{GibbonsHawking} in the context of gravitational instantons. These are also the varieties appearing in \cite{KS}. Moreover, if we choose a left cyclic polarization of $\cal{V}$, then Khovanov--Seidel's Lagrangians in $\M_{\cal{V}}$ are the relative core Lagrangians $X_{\alpha}$ for $\alpha \in \cal{P}$. The algebra $B(\cal{V})$ in this case is isomorphic to the Khovanov--Seidel quiver algebra $A_{n-1}$ (this is also true for right cyclic polarizations); the algebra $A(\cal{V})$ is isomorphic to its Koszul dual $A_{n-1}^!$. In \cite{Man-KS}, the algebra $B(\cal{V})$ was presented as a quotient of Ozsv{\'a}th--Szab{\'o}'s algebra $\B_l(n,1)$ (see Section~\ref{sec:OSz} below); we will see that $\B_l(n,1) \cong \tilde{B}(\cal{V})$.

\subsubsection{$k=n-1$}

Let $V$ be the column span of the matrix
$\begin{bmatrix}
1 &  \\
1 & 1 & \\
  & 1 & \ddots &  \\
	&   & \ddots & 1  \\
	&   &        & 1 & 1 \\
	&   &        &   & 1
\end{bmatrix}$ of size $n \times (n-1)$,
and choose any $\eta, \xi$ such that $(V,\eta,\xi)$ is left or right cyclic. The form of this matrix implies that the algebra $B(\cal{V})$ is isomorphic to$A_{n-1}^!$, the Koszul dual Khovanov--Seidel algebra (thus this is true for all left and right cyclic arrangements for $k=n-1$). In the rational case, $\M_{\cal{V}}$ is isomorphic to $T^* \C P^{n-1}$, which is studied by Calabi \cite{Calabi} in precursor work to the theory of hypertoric varieties.

The isomorphism $\B_l(n,n-1) \cong \tilde{B}(\cal{V})$ below presents $A_{n-1}^!$ as a quotient of $\B_l(n,n-1)$, in close analogy to the $k=1$ case studied in \cite{Man-KS}. It is interesting to compare with \cite{LaudaManion}, which considers certain finite-dimensional quotients of $\B_l(n,k)$ that are related to category $\cal{O}$. While the $k=1$ quotient in \cite{LaudaManion} is $A_{n-1}$, the $k=n-1$ quotient is not $A_{n-1}^!$ but instead a significantly more complicated algebra. Here, unlike in \cite{LaudaManion}, the cases $k=1$ and $k=n-1$ are equally simple.

Unlike for the finite-dimensional algebras, it is not true that $\B_l(n,1)$ and $\B_l(n,n-1)$ are Koszul dual to each other. Rather, as shown by Ozsv{\'a}th--Szab{\'o} in \cite{OSzNew}, $\B_l(n,k)$ is Koszul dual to an algebra formed from $\B_l(n,n-k)$ (or the isomorphic algebra $\B_r(n,n-k)$) by adding additional algebra generators $C_i$, together with a homological grading and a differential.

\subsubsection{$n=4, k=2$}

Consider the $(n=4,k=2)$ Vandermonde arrangement associated to $1 < 2  < 3 < 4$, where $V$ is the column span of $\begin{bmatrix} 1 & 1 \\ 1 & 2 \\ 1 & 3 \\ 1 & 4 \end{bmatrix}$ and $\eta$ is represented by $\begin{bmatrix} 1 & 4 & 9 & 16 \end{bmatrix}$. Define $\xi_l$ by letting $z_0 = 0$, so that $\xi_l$ has matrix $\begin{bmatrix} 1 & 0 \end{bmatrix}$ in the columns of $V$. Define $\xi_r$ by letting $z_5 = 5$, so that $\xi_r$ has matrix $\begin{bmatrix} 1 & 5 \end{bmatrix}$ in the columns of $V$.

\begin{figure}
\includegraphics[width=.5\textwidth]{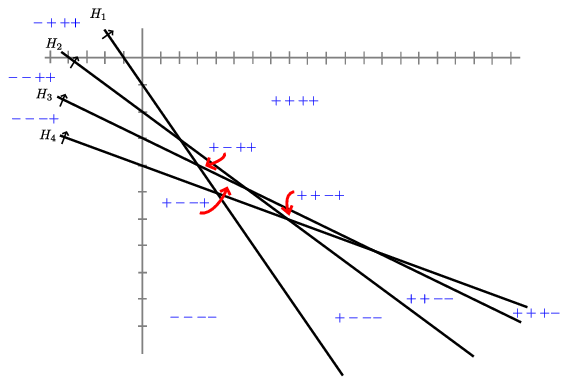}

\includegraphics[width=.5\textwidth]{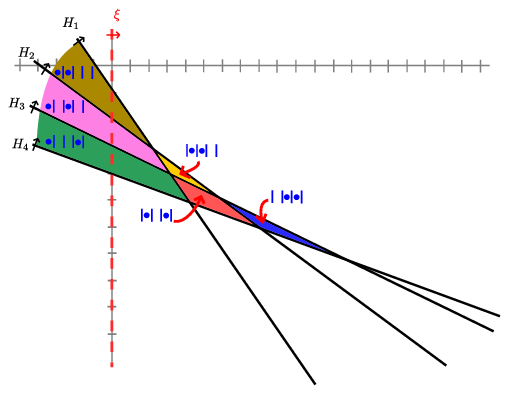}

\includegraphics[width=.5\textwidth]{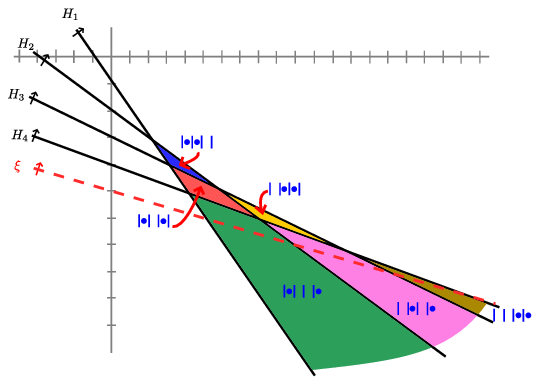}

\caption{A cyclic arrangement with left cyclic and right cyclic polarizations for $n=4, k=2$.}
\label{fig:n4k2}
\end{figure}

The arrangement $\cal{H}_{\cal{V}} \subset (V + \eta) \cong \R^2$ is shown in Figure~\ref{fig:n4k2} (left), with the regions $\Delta_{\alpha}$ for $\alpha \in \cal{F}$. The middle picture of Figure~\ref{fig:n4k2} indicates the left cyclic polarization $\xi_l$, and the right picture of Figure~\ref{fig:n4k2} does the same for $\xi_r$. In both cases, the regions $\Delta_{\alpha}$ for $\alpha \in \cal{P}$ are colored and labeled by sets of dots in regions.

\section{Ozsv{\'a}th--Szab{\'o} algebras as hypertoric convolution algebras}\label{sec:OSz}

\subsection{Definitions}

We define the graded algebra $\B(n,k)$ from \cite{OSzNew} using the generators-and-relations description from \cite{MMW1}. First we introduce some terminology. Let $V(n,k)$ be the set of $k$-element subsets $\x \subset \{0,\ldots,n\}$.

\begin{definition}\label{def:SmallStep}
Let $\B(n,k)$ be the path algebra of the quiver with vertex set $V(n,k)$ and arrows
\begin{itemize}
\item for $1 \leq i \leq n$, $R_i$ from $\x$ to $\y$ and $L_i$ from $\y$ to $\x$ if $\x \cap \{i-1,i\} = \{i-1\}$ and $\y = (\x \setminus \{i-1\}) \cup \{i\}$,
\item for $1 \leq i \leq n$, $U_i$ from $\x$ to $\x$ for all $\x \in V(n,k)$
\end{itemize}
modulo the relations
\begin{enumerate}
\item\label{it:UCommute} $R_i U_j = U_j R_i$, $L_i U_j = U_j L_i$, $U_i U_j = U_j U_i$,
\item\label{it:RLU} $R_i L_i = U_i$, $L_i R_i = U_i$,
\item\label{it:DistantCommute} $R_i R_j = R_j R_i$, $L_i L_j = L_j L_i$, $L_i R_j = R_j L_i$ ($|i-j|>1$),
\item\label{it:TwoLinePass} $R_{i-1} R_{i} = 0$, $L_i L_{i-1} = 0$,
\item\label{it:UVanish} $U_i \Ib_{\x} = 0$ if $\x \cap \{i-1,i\} = \emptyset$.
\end{enumerate}
The relations are assumed to hold for any linear combination of quiver paths with the same starting and ending vertices and labels $R_i,L_i,U_i$ as described; $\Ib_{\x}$ denotes the trivial path at $\x \in V(n,k)$.  The elements $\Ib_{\x} \in B(n,k)$ give a complete set of orthogonal idempotents. We define a grading on $\B(n,k)$ by setting $\deg(R_i) = \deg(L_i) = 1$ and $\deg(U_i) = 2$; we can refine to a multi-grading by $\Z\langle e_1, \ldots, e_n \rangle$ by setting $\deg(R_i) = \deg(L_i) = e_i$ and $\deg(U_i) = 2e_i$. Our single and multiple gradings are two times the single and multiple gradings defined in \cite{OSzNew}.
\end{definition}

Recall from Section~\ref{sec:SignSeqAndOSzIdems} that we let $V_l(n,k)$ denote the subset of $V(n,k)$ consisting of $k$-element subsets of $\{0,\ldots,n-1\}$. Similarly, $V_r(n,k)$ denotes the set of $k$-element subsets of $\{1,\ldots,n\}$, and $V'(n,k)$ denotes the set of $k$-element subsets of $\{1,\ldots,n-1\}$.

\begin{definition} \label{def:variants-OSz}
Let $\B_l(n,k) = \oplus_{\x, \x' \in V_l(n,k)} \Ib_{\x} \cdot \B(n,k) \Ib_{\x'}$, $\B_r(n,k) = \oplus_{\x, \x' \in V_r(n,k)} \Ib_{\x} \cdot \B(n,k) \Ib_{\x'}$, and $\B'(n,k) = \oplus_{\x, \x' \in V'(n,k)} \Ib_{\x} \cdot \B(n,k) \Ib_{\x'}$.
\end{definition}

To build the idempotents $\Ib_{\x}$ into the structure, we can view all of the above algebras as categories (enriched in graded abelian groups) whose objects are $\x \in V(n,k)$, $V_l(n,k)$, $V_r(n,k)$, or $V'(n,k)$ as appropriate. We refer to this definition of Ozsv{\'a}th--Szab{\'o}'s algebras as the small-step quiver description; there is also a ``big-step'' quiver description that is more transparently equivalent to Ozsv{\'a}th--Szab{\'o}'s original definitions.

In \cite[Section 3.6]{OSzNew}, Ozsv{\'a}th--Szab{\'o} define an anti-automorphism of $\B(n,k)$ that restricts to an anti-automorphism of $\B_l(n,k)$, $\B_r(n,k)$ and $\B'(n,k)$ given as follows.

\begin{definition} \label{def:PsiOS}
The anti-automorphism $\psi_{OSz}\colon \B(n,k) \to \B(n,k)^{{\rm opp}}$ sends $R_i \mapsto L_i$, $L_i \mapsto R_i$, and $U_i \mapsto U_i$ in the small-step quiver description of $\B(n,k)$.
\end{definition}

\begin{remark}\label{rem:AurouxStops}
Ozsv{\'a}th--Szab{\'o} introduced $\B(n,k)$ and its relatives in \cite{OSzNew} as part of an algebraic theory that can be used for very efficient computations of knot Floer homology (see also \cite{OSzNewer,OSzHolo,OSzLinks}). Their theory is based on the ideas of bordered Floer homology; given a link (say in $\R^3$), it can be viewed as computing a Heegaard Floer invariant of the link complement by writing the complement as a composition of 3d cobordisms between planes $\C$ with various numbers of punctures $z_1, \ldots, z_n$. By \cite[Theorem 3.25]{LP} or \cite[Corollary 9.10]{MMW2}  plus the relationship between strands algebras and Fukaya categories from \cite[Proposition 11]{Auroux}, Ozsv{\'a}th--Szab{\'o}'s algebras are the homology of formal dg algebras built from morphism spaces between the distinguished Lagrangians in $\Sym^k(\C \setminus \{z_1, \ldots, z_n\})$ of Section~\ref{sec:SymProdLagrangians} in an appropriately-defined partially wrapped Fukaya category of this symmetric product. Specifically, $\B_l(n,k)$ is the homology of the algebra of morphisms between the ``left cyclic'' Lagrangians from that section; similar statements hold for $\B_r(n,k)$ and the ``right cyclic'' Lagrangians, $\B'(n,k)$ and the ``core'' Lagrangians, and $\B(n,k)$ and the union of the left cyclic and right cyclic Lagrangians. In the left cyclic case, the stops for the partial wrapping are the ones specified in \cite{Auroux} given (in the language of that paper) the decorated surface $(F,Z, \underline{\alpha})$ shown in Figure~\ref{fig:DecoratedSurface} with $F$ a disk minus open neighborhoods of $n$ interior points, $Z$ a single point in the outer boundary of $F$, and $\underline{\alpha}$ the system of red arcs shown in Figure~\ref{fig:DecoratedSurface}. The other cases are analogous.
\end{remark}

\begin{figure}
\includegraphics[scale=0.7]{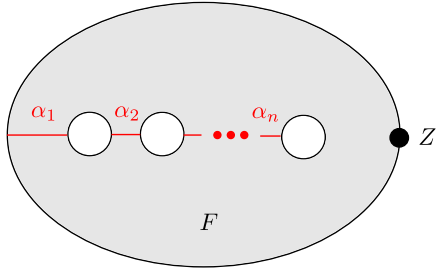}
\caption{The decorated surface $(F,Z,\underline{\alpha})$.}
\label{fig:DecoratedSurface}
\end{figure}

\subsection{Isomorphisms of algebras: left cyclic case}

Let $\cal{V} = (V,\eta,\xi)$ be left cyclic. Note that in the quiver defining $\B_l(n,k)$, $R_i$ and $L_i$ arrows exist between vertices $\x$ and $\y$ if and only if $\alpha_{\x} \leftrightarrow \alpha_{\y}$, where $\alpha_{\x}, \alpha_{\y}$ are the elements of $\cal{P}$ corresponding to $\x,\y$ under the bijection $\kappa_l \maps \cal{P} \to V_l(n,k)$ of Definition~\ref{def:StatesFromAlpha} \eqref{def:LeftIStatesFromAlpha}. By the quiver description of $\Bt(\cal{V})$, we mean its description as $\At(\cal{V}^{\vee})$, i.e. we are using the small-step quiver descriptions everywhere.

\begin{definition}
The homomorphism $\Phi$ from $\B_l(n,k)$ to $\tilde{B}(\cal{V})$ is defined in terms of the quiver descriptions of the algebras by sending
\begin{itemize}
\item vertices $\x$ to vertices $\kappa_l^{-1}(\x) :=\alpha_{\x}$,
\item arrows $\x \xrightarrow{R_i} \y$, $\y \xrightarrow{L_i} \x$ to arrows $p(\alpha_{\x},\alpha_{\y})$, $p(\alpha_{\y},\alpha_{\x})$,
\item arrows $\x \xrightarrow{U_i} \x$ to $u_i e_{\alpha_{\x}}$.
\end{itemize}
One can check that $\Phi$ preserves multi-degrees.
\end{definition}

\begin{proposition}
The map $\Phi$ is well-defined.
\end{proposition}

\begin{proof}
We must check that the relations in Definition~\ref{def:SmallStep} are preserved under $\Phi$; we will use the relations for $\tilde{B}(\cal{V})$ from Corollary~\ref{cor:SimpleBTilde}.
The relations \eqref{it:UCommute} hold after applying $\Phi$ because the $u_i$ variables commute with elements of $P(Q)$ in the tensor product algebra $P(Q) \otimes_{\Z} \Z[u_1,\ldots,u_n]$, even before imposing relations on $\tilde{B}(\cal{V})$.
The relations \eqref{it:RLU} hold after applying $\Phi$ by Corollary~\ref{cor:SimpleBTilde}, item A3.
The relations \eqref{it:DistantCommute} hold after applying $\Phi$ by Corollary~\ref{cor:SimpleBTilde}, item A2.

For the relations \eqref{it:TwoLinePass}, suppose we have a composable pair of arrows $\x \xrightarrow{R_{i-1}} \y \xrightarrow{R_i} \z$ in the quiver description of $\B_l(n,k)$. We have $2 \leq i \leq n-1$ and $\x \cap \{i-2,i-1,i\} = \{i-2\}$. Thus, the signs of $\alpha_{\x}$ in positions $(i-1,i,i+1)$ are either $(+,+,+)$ or $(-,-,-)$; without loss of generality assume they are $(+,+,+)$. The signs of $\alpha_{\y}$ and $\alpha_{\z}$ in these positions are $(-,+,+)$ and $(-,-,+)$ respectively. Let $\beta$ agree with $\alpha$ except that $\beta = (+,-,+)$ in these positions. We have $\var_l(\beta) = k+2$, so $e_{\beta} = 0$ by Corollary~\ref{cor:SimpleBTilde}, item A1. Since $\alpha_{\x} \leftrightarrow \beta$ and $\alpha_{\z} \leftrightarrow \beta$, by Corollary~\ref{cor:SimpleBTilde}, item A2 we have
\[
\Phi(R_{i-1}) \Phi(R_i) = p(\alpha_{\x},\alpha_{\y}) p(\alpha_{\y},\alpha_{\z}) = p(\alpha_{\x},\beta) p(\beta,\alpha_{\z}) = 0.
\]
The relations $L_i L_{i-1}$ are similar.

For the relations \eqref{it:UVanish}, suppose that $\x \cap \{i-1,i\} = \emptyset$. We have $1 \leq i \leq n$; first assume $2 \leq i \leq n-1$. The signs of $\alpha$ in positions $(i-1,i,i+1)$ are either $(+,+,+)$ or $(-,-,-)$; without loss of generality assume they are $(+,+,+)$. Defining $\beta$ to agree with $\alpha$ except in these positions where $\beta = (+,-,+)$, we have $\var_l(\beta) = k+2$ and thus $e_{\beta} = 0$ by Corollary~\ref{cor:SimpleBTilde}, item A1. Since $\alpha_{\x} \leftrightarrow \beta$, by Corollary~\ref{cor:SimpleBTilde}, item A3 we have
\[
\Phi(U_i \Ib_{\x}) = u_i e_{\alpha_{\x}} = p(\alpha_{\x},\beta) p(\beta,\alpha_{\x}) = 0.
\]

Now let $i=1$, so that $\x \cap \{0,1\} = \emptyset$. The signs of $\alpha$ in positions $(1,2)$ are $(+,+)$. If we take $\beta$ to have signs $(-,+)$ in positions $(1,2)$, then $\var_l(\beta) = k+2$ and we get $\Phi(U_1 \Ib_{\x}) = 0$ as before. Similarly, if $i = n$, then $\x \cap \{n-1,n\} = \emptyset$. The sign of $\alpha$ in positions $(n-1,n)$ are either $(+,+)$ or $(-,-)$; without loss of generality assume they are $(+,+)$. If we take $\beta$ to have signs $(+,-)$ in positions $(n-1,n)$, then $\var_l(\beta) = k+1)$ and we get $\Phi(U_n \Ib_{\x}) = 0$.
\end{proof}

\begin{definition}
The homomorphism $\Psi$ from $\tilde{B}(\cal{V})$ to $\B_l(n,k)$ is defined in terms of the quiver descriptions by sending
\begin{itemize}
\item vertices $\alpha$ to vertices $\kappa_l(\alpha)=\x_{\alpha}$ if $\var_l(\alpha) =k$ and to zero if $\var_l(\alpha) > k$, 
\item arrows $p(\alpha,\beta)$, $p(\beta,\alpha)$ to arrows $\x_{\alpha} \xrightarrow{R_i} \x_{\beta}$, $\x_{\beta} \xrightarrow{L_i} \x_{\alpha}$ respectively if $\beta = \alpha^i$ and $\beta$ delays a sign change compared to $\alpha$, and the reverse if $\beta$ advances a sign change compared to $\alpha$, 
\item generators $u_i$ of $\Z[u_1,\ldots,u_n]$ to elements $\sum_{\x \in V_l(n,k)} U_i \Ib_{\x}$ of $\B_l(n,k)$.
\end{itemize}
One can check that $\Psi$ preserves multi-degrees.
\end{definition}

\begin{proposition}
The map $\Psi$ is well-defined.
\end{proposition}

\begin{proof}
The map $\Psi$ is well-defined as a map from $P(Q) \otimes_{\Z} \Z[u_1,\ldots,u_n]$ to $\B_l(n,k)$ by item~\eqref{it:UCommute} of Definition~\ref{def:SmallStep}. We will check that $\Psi$ preserves the relations from Corollary~\ref{cor:SimpleBTilde}.
The relations A1 hold by construction.

For the relations A2, if $\alpha,\beta,\gamma,\delta$ all have $\var_l = k$ then the relations follow from item~\eqref{it:DistantCommute} of Definition~\ref{def:SmallStep} (note that $\alpha$, $\beta$, $\gamma$, $\delta$ are required to be distinct). Without loss of generality, the only other case we need to check is when $\alpha,\beta,\gamma$ have $\var_l = k$ while $\var_l(\delta) > k$. In this case, moving from $\alpha$ to $\beta$ and then to $\gamma$ either delays a sign change and then delays it one step further, or advances a sign change and then advances it one step further. The required relation $\Psi(p(\alpha),p(\beta)) \Psi(P(\beta),P(\gamma)) = 0$ then follows from item~\eqref{it:TwoLinePass} of Definition~\ref{def:SmallStep}.

For the relations A3, if $\var_l(\alpha) = \var_l(\beta) = k$ then $\Psi(p(\alpha,\beta)) \Psi(p(\beta,\alpha)) = \Psi(u_i e_{\alpha})$ follows from item~\eqref{it:RLU} of Definition~\ref{def:SmallStep}. On the other hand, if $\var_l(\alpha) = k$ and $\var_l(\beta) > k$, then changing $\alpha$ to $\beta$ must change initial signs $(+,+)$, terminal signs $(+,+)$ or $(-,-)$, or length-three sign intervals $(+,+,+)$ or $(-,-,-)$ to initial signs $(-,+)$, terminal signs $(+,-)$ or $(-,+)$, or length-three intervals $(+,-,+)$ or $(-,+,-)$ respectively. In any of these cases, $\Psi$ preserves the relation of item A3 by item~\eqref{it:UVanish} of Definition~\ref{def:SmallStep}.

\end{proof}

\begin{theorem}\label{thm:LeftCyclicIsom}
The maps $\Phi$ and $\Psi$ are inverse isomorphisms of (multi-graded) $\mathbb{Z}$-algebras that  intertwine the anti-involution $\psi_{OSz}$ from Definition~\ref{def:PsiOS} with
$\psi^{\cal{V}}$  on $\Bt(\cal{V})$ coming from $\tilde{R}_{\alpha \beta} \mapsto \tilde{R}_{\beta\alpha}$.
\end{theorem}

\begin{proof}
One can check that these are isomorphisms when applied to quiver vertices or arrows on both sides.
\end{proof}

\begin{corollary}\label{cor:MainUnpolarizedThm}
If $(V,\eta)$ is cyclic, the isomorphisms $\Phi, \Psi$ of Theorem~\ref{thm:LeftCyclicIsom} for any left cyclic polarization $\xi$ of $(V,\eta)$ restrict to isomorphisms between $\B'(n,k)$ and the algebra $\Bt'(V,\eta)$ from Section~\ref{sec:UnpolarizedHypertoricAlgs}. These isomorphisms are independent of the choice of $\xi$.
\end{corollary}

\begin{proof}
Note that the isomorphisms $\Phi, \Psi$ interchange $\Ib_{\x_{\alpha}}$ and $e_{\alpha}$. The value of $\Phi$ on any idempotent $\Ib_{\x_{\alpha}}$ with $\alpha \in \cal{K}$, and on the elements $R_i$, $L_i$, $u_i$ between vertices $\x_{\alpha}, \x_{\beta}$ with $\alpha, \beta \in \cal{K}$, are independent of $\xi$; by choosing minimal-length paths between idempotents, any element of $\Bt'(V,\eta)$ is a linear combination of products of such elements.
\end{proof}

\subsection{Isomorphisms of algebras: right cyclic case}
Now let $\cal{V} = (V,\eta,\xi)$ be right cyclic. As before, in the quiver defining $\B_r(n,k)$, $R_i$ and $L_i$ arrows exist between vertices $\x$ and $\y$ if and only if $\alpha_{\x} \leftrightarrow \alpha_{\y}$.

\begin{definition}
The homomorphism $\Phi$ from $\B_r(n,k)$ to $\tilde{B}(\cal{V})$ is defined in terms of the quiver descriptions of the algebras by sending
\begin{itemize}
\item vertices $\x$ to vertices $\kappa_r^{-1}(\x) = \alpha_{\x}$,
\item arrows $\x \xrightarrow{R_i} \y$, $\y \xrightarrow{L_i} \x$ to arrows $p(\alpha_{\x},\alpha_{\y})$, $p(\alpha_{\y},\alpha_{\x})$,
\item arrows $\x \xrightarrow{U_i} \x$ to $u_i e_{\alpha_{\x}}$.
\end{itemize}
One can check that $\Psi$ preserves multi-degrees.
\end{definition}

\begin{definition}
The homomorphism $\Psi$ from $\tilde{A}(\cal{V}^{\vee})$ to $\B_r(n,k)$ is defined in terms of the quiver descriptions by sending
\begin{itemize}
\item vertices $\alpha$ to vertices $\kappa_r(\alpha)=\x_{\alpha}$ if $\var_r(\alpha) =k$ and to zero if $\var_r(\alpha) > k$,
\item arrows $p(\alpha,\beta)$, $p(\beta,\alpha)$ to arrows $\x_{\alpha} \xrightarrow{R_i} \x_{\beta}$, $\x_{\beta} \xrightarrow{L_i} \x_{\beta}$ respectively if $\beta = \alpha^i$ and $\beta$ delays a sign change compared to $\alpha$, and the reverse if $\beta$ advances a sign change compared to $\alpha$,
\item generators $u_i$ of $\Z[u_1,\ldots,u_n]$ to elements $\sum_{\x \in V_r(n,k)} U_i \Ib_{\x}$ of $\B_r(n,k)$.
\end{itemize}
One can check that $\Psi$ preserves multi-degrees.
\end{definition}

\begin{theorem}\label{thm:RightCyclicIsom}
The maps $\Phi$ and $\Phi$ are well defined and inverse  isomorphisms of (multi-graded) $\mathbb{Z}$-algebras that  intertwine the anti-involution $\psi_{OSz}$ from Definition~\ref{def:PsiOS} with
$\psi^{\cal{V}}$  on $\Bt(\cal{V})$ coming from $\tilde{R}_{\alpha \beta} \mapsto \tilde{R}_{\beta\alpha}$.
\end{theorem}

The proof is similar to above.

\subsection{The center of Ozsv{\'a}th--Szab{\'o}'s algebras} \label{sec:center}

Recall from Section~\ref{sec:AlgebrasGeometricAspects} that for a rational polarized arrangement $\cal{V}$ we have $Z(\Bt(\cal{V})) \cong H^{\ast}_{T^k}(\M_{\cal{V}};\C)$.  Explicitly,
\[
Z(\tilde{B}(\cal{V})) = {\rm Sym} (\R^n) / ( u_s \mid S \subset I \text{ such that $H_s \cap V_{\eta}= \emptyset$}).
\]
From Theorem~\ref{thm:LeftCyclicIsom} and \ref{thm:RightCyclicIsom}, we get similar statements for Ozsv{\'a}th--Szab{\'o}'s algebras.

\begin{corollary} \label{cor:center}
The centers of both $\B_l(n,k)$ and $\B_r(n,k)$ are
\[
Z(\B_l(n,k)) = Z(\B_r(n,k)) \cong \frac{\Z[U_1,\ldots,U_n]}{(U_{i_1} \cdots U_{i_{k+1}} : 1 \leq i_1 < \cdots < i_{k+1} \leq n)},
\]
where the element $U_i$ in this polynomial ring corresponds to the sum of $U_i$ generators at all idempotents of $\B_l(n,k)$ or $\B_r(n,k)$.

\end{corollary}

One consequence is that the centers of $\B_l(n,k)$ and $\B_r(n,k)$ have the structure of localization algebras as discussed in \cite{BLPPW}.

\begin{remark}
It follows from the results in this paper and \cite[Theorem 1.2(5)]{HypertoricCatO}
 that for a rational polarized arrangement $\cal{V}$ with $\M_{\cal{V}}$ smooth, we have an isomorphism
\[
K_0(\Bt(\cal{V})) \otimes_{\Z[q,q^{-1}]} \C \cong H^{2k}_{\mathbb{T}}(\M_{\cal{V}}; \C),
\]
where $q$ acts by $1$ on $\C$. A priori, we cannot apply this result to left and right cyclic $\cal{V}$ if $1 < k < n-1$, since $\M_{\cal{V}}$ is not smooth in these cases. Note that we can interpret $K_0(\B_l(n,k))$ in terms of the cohomology of symmetric products: let $F$ be $D^2$ with open neighborhoods of $n$ interior points removed, and let $S_+$ be the union of the internal boundary components of $F$ with a closed interval on the outer boundary of $F$. The cohomology group $H^k(\Sym^k(F), \Sym^k(S_+) ;\C) \cong \wedge^k H^1(F,S_+;\C)$ has a basis given by $k$-fold wedge products of our $n$ distinguished straight-line Lagrangians in $F$, so we can identify it with $K_0(\B_l(n,k)) \otimes_{\Z[q,q^{-1}]} \C$.
\end{remark}

\section{Quantum group bimodules } \label{sec:qg-bimodules}

Here we recall a categorified action of a variant of $\gl(1|1)$ as bimodules over $\B_l(n,k)$, introduced in \cite{LaudaManion} in the style of \cite{Sar-tensor} and equivalent to a particular case of the bimodules defined in \cite{ManionRouquier}. We will explain a bordered Floer perspective on the bimodules over $\B_l(n,k)$, based on Heegaard diagrams, and connect them to deletion and restriction bimodules. In particular, we will see that our factorization of the bimodule $\Fb_k$ into deletion and restriction bimodules has a natural interpretation from the bordered Floer perspective.

\subsection{Basic definitions}

We first recall the bimodule $\Fb_k^{OSz}$ over $(\B_l(n,k), \B_l(n,k+1))$ defined in \cite[Section 9]{LaudaManion}. Let $e_k^{\vee}$ denote the sum of all idempotents $\Ib_{\x}$ for $\x \in V_l(n,k)$ such that $0 \notin \x$, and let $e_k^{\wedge}$ denote the sum of all idempotents $\Ib_{\x}$ for $\x \in V_l(n,k)$ such that $0 \in \x$. For $\x \in V_l(n,k)$ with $0 \in \x$, let $\x^{(\vee)}$ denote $\x \setminus \{0\}$.

Let $P_k^{\vee} = \B_l(n,k) e_k^{\vee}$, a left module over $\B_l(n,k)$. We will define a right action of $\B_l(n,k+1)$ on $P_k^{\vee}$; first, define a surjective ring homomorphism
\[
\Psi' : e_{k+1}^{\wedge} \B_l(n,k+1) e_{k+1}^{\wedge} \to e_k^{\vee} \B_l(n,k) e_k^{\vee}
\]
as follows. For $\x, \y \in V_l(n,k+1)$ with $0 \in \x \cap \y$, we have
\[
\Ib_{\x} \B_l(n,k+1) \Ib_{\y} \cong \frac{\Z[U_1,\ldots,U_n]}{(p_G : G \textrm{ generating interval between } \x \textrm{ and } \y)},
\]
and similarly for $\Ib_{\x^{(\vee)}} \B_l(n,k) \Ib_{\y^{(\vee)}}$, where generating intervals are defined as in \cite[Definition 2.11]{LaudaManion} (following Ozsv{\'a}th--Szab{\'o} \cite{OSzNew}). The generating intervals $p_G$ between $\x$ and $\y$ are contained in the ideal generated by the generating intervals between $\x^{\vee}$ and $\y^{\vee}$, so we get a homomorphism of $\Z[U_1,\ldots,U_n]$-modules
\[
\Ib_{\x} \B_l(n,k+1) \Ib_{\y} \to \Ib_{\x^{\vee}} \B_l(n,k) \Ib_{\y^{\vee}}
\]
by sending $1 \mapsto 1$. Summing over $\x, \y \in V_l(n,k)$ with $0 \in \x \cap \y$, we get a surjective homomorphism of $\Z[U_1,\ldots,U_n]$-modules
\[
\Psi' : e_{k+1}^{\wedge} \B_l(n,k+1) e_{k+1}^{\wedge} \to e_k^{\vee} \B_l(n,k) e_k^{\vee}.
\]
Using Ozsv{\'a}th--Szab{\'o}'s original definition of $\B_l(n,k)$ in terms of these quotients of $\Z[U_1,\ldots,U_n]$ as reviewed e.g. in \cite[Section 2.1]{LaudaManion}, one can check that $\Psi'$ is actually a ring homomorphism, i.e. it respects multiplication. Given $\x, \y, \z \in V_l(n,k+1)$ with $0 \in \x \cap \z$ but $0 \notin \y$, and $b, b' \in \B_l(n,k+1)$, one can check that
\[
\Psi'(\Ib_{\x} b \Ib_{\y} b' \Ib_{\z}) = 0.
\]

We can now define a surjective ring homomorphism
\[
\frac{\B_l(n,k+1)}{\B_l(n,k+1) e_{k+1}^{\vee} \B_l(n,k+1)} \to e_k^{\vee} \B_l(n,k) e_k^{\vee}
\]
by sending
\[
[b] \mapsto \Psi'(e_{k+1}^{\wedge} b e_{k+1}^{\wedge}).
\]
By the above remarks, this map is well-defined; note that
\[
e_{k+1}^{\wedge} b b' e_{k+1}^{\wedge} = e_{k+1}^{\wedge} b e_{k+1}^{\wedge} b' e_{k+1}^{\wedge} + e_{k+1}^{\wedge} b e_{k+1}^{\vee} b' e_{k+1}^{\wedge}
\]
and $e_{k+1}^{\wedge} b e_{k+1}^{\vee} b' e_{k+1}^{\wedge}$ maps to zero under $\Psi'$. Precomposing with the quotient map, we get a surjective ring homomorphism
\[
\B_l(n,k+1) \to e_k^{\vee} \B_l(n,k) e_k^{\vee},
\]
which we can view as a non-unital homomorphism from $\B_l(n,k+1)$ to $\B_l(n,k)$. On the multiplicative generators of $\B_l(n,k+1)$, this homomorphism sends
\begin{itemize}
\item $\Ib_{\x} \mapsto \Ib_{\x^{(\vee)}}$ if $0 \in \x$,
\item $\Ib_{\x} \mapsto 0$ if $0 \notin \x$,
\item $\x \xrightarrow{R_i} \y$ and $\x \xrightarrow{L_i} \y$ map to $\x^{(\vee)} \xrightarrow{R_i} \y^{(\vee)}$ and $\x^{(\vee)} \xrightarrow{L_i} \y^{(\vee)}$ if $0 \in \x \cap \y$, and map to zero otherwise,
\item $\x \xrightarrow{U_i} \x$ maps to $\x^{\vee} \xrightarrow{U_i} \x^{\vee}$ if $0 \in \x$ and zero otherwise.
\end{itemize}
Given $b \in P_k^{\vee} \subset \B_l(n,k)$ and $b' \in \B_l(n,k+1)$, right multiplying $b$ by the image of $b'$ under this homomorphism results in another element of $P_k^{\vee}$.

\begin{definition}
The bimodule $\Fb_k^{OSz}$ over $(\B_l(n,k), \B_l(n,k+1))$ is the left $\B_l(n,k)$-module $P_k^{\vee}$, with an action of $\B_l(n,k+1)$ given by the above homomorphism followed by right multiplication.
\end{definition}

One can alternatively define $\Fb_k^{OSz}$ by inducing the left action of the identity bimodule over $\B_l(n,k+1)$ by the above non-unital homomorphism. The multi-grading by $\Z\langle e_1, \ldots, e_n \rangle$ on $P_k^{\vee} \subset \B_l(n,k)$ is additive with respect to right multiplication by $\B_l(n,k+1)$ as well as with respect to left multiplication by $\B_l(n,k)$, so we can view $\Fb_k$ as a multi-graded bimodule over $(\B_l(n,k), \B_l(n,k+1))$. As with the algebras, we can collapse this $\Z^{n}$ to a grading by $\Z$.

For $0 \leq k \leq n$, the Grothendieck groups (regarded as modules over $\mathbb{C}(q)$) of the compact derived categories of singly-graded left $\B_l(n,k)$-modules are identified with the nonzero weight spaces of the $U_q(\gl(1|1))$ representation $V^{\otimes n}$, where $V$ is the vector representation, see \cite[Section 8.11]{LaudaManion}). Under this identification, the action of $F \in U_q(\gl(1|1))$ on $V^{\otimes n}$ is induced by tensoring with the bimodule $\Fb_k$ above.
In particular, we have $\Fb_{k}\Fb_{k+1} = 0$, categorifying the  $\mf{gl}(1|1)$ relation $F^2=0$.

Following \cite{LaudaManion}, let $\Eb''_k = (\Eb'')_k^{\OSz}$ be $\Hom_{\B_l(n,k)}(\Fb_k, \B_l(n,k))$, the left dual of $\Fb_k$, which is a multi-graded bimodule over $(\B_l(n,k+1), \B_l(n,k))$. We can view $\Eb''_k$ as the identity bimodule over $\B_l(n,k+1)$ with its right action induced by the above homomorphism from $\B_l(n,k+1)$ to $\B_l(n,k)$.

\begin{remark}
We denote the dual bimodule $\Eb''_k$ rather than $\Eb_k$ because its singly graded version does not categorify the action of the standard $E \in U_q(\gl(1|1))$ on $V^{\otimes n}$.  Rather, $\Eb''_k$ categorifies the action of
\[
E'' := (q^{-1} - q)EK,
\]
whose action on $V^{\otimes n}$ is dual to that of $F$ with respect to the bilinear form $(-,-)_{\cal{V}}$.
\end{remark}

It was convenient in \cite{LaudaManion} to focus on $\B_l(n,k)$, but one can define analogous bimodules over $\B_r(n,k)$ by replacing $0$ with $n$ everywhere above. For the larger algebra $\B(n,k)$ with $\binom{n+1}{k}$ idempotents, one can define two pairs of bimodules (one using $0$ and the other using $n$).  There are no such bimodules over the smaller algebras $\B'(n,k)$ (see the next subsection).
\subsection{Strands interpretation}
Recall that the algebras $\B_l(n,k)$, $\B_r(n,k)$, $\B(n,k)$, and $\B'(n,k)$ have interpretations as the homology of (formal) dg strands algebras $\A(\Zc,k)$ of the form appearing in bordered Floer homology \cite{LP, MMW2}. Here $\Zc$ is an arc diagram as considered by Zarev \cite{Zarev}, except that instead of matchings on a collection of oriented intervals, we allow matchings on a collection of oriented intervals and circles. The arc diagrams $\Zc_l$, $\Zc_r$, $\Zc_{\full}$, and $\Zc'$ whose strands algebras have homology $\B_l(n,k)$, $\B_r(n,k)$, $\B(n,k)$, and $\B'(n,k)$ respectively are shown in Figure~\ref{fig:FourArcDiagrams}. A representative element of the strands algebra $\A(\Zc_l) = \oplus_{k=0}^n \A(\Zc_l, k)$ is shown in Figure~\ref{fig:StrandsAlgElt}; the combinatorics of these strands elements is explained in \cite{MMW2} in the case at hand, adapting the general prescriptions of \cite{LOTBorderedOrig,Zarev}.

\begin{figure}[!tbp]
\centering
\begin{minipage}{.5\textwidth}
  \centering
  \includegraphics[width=.5\linewidth
  ]{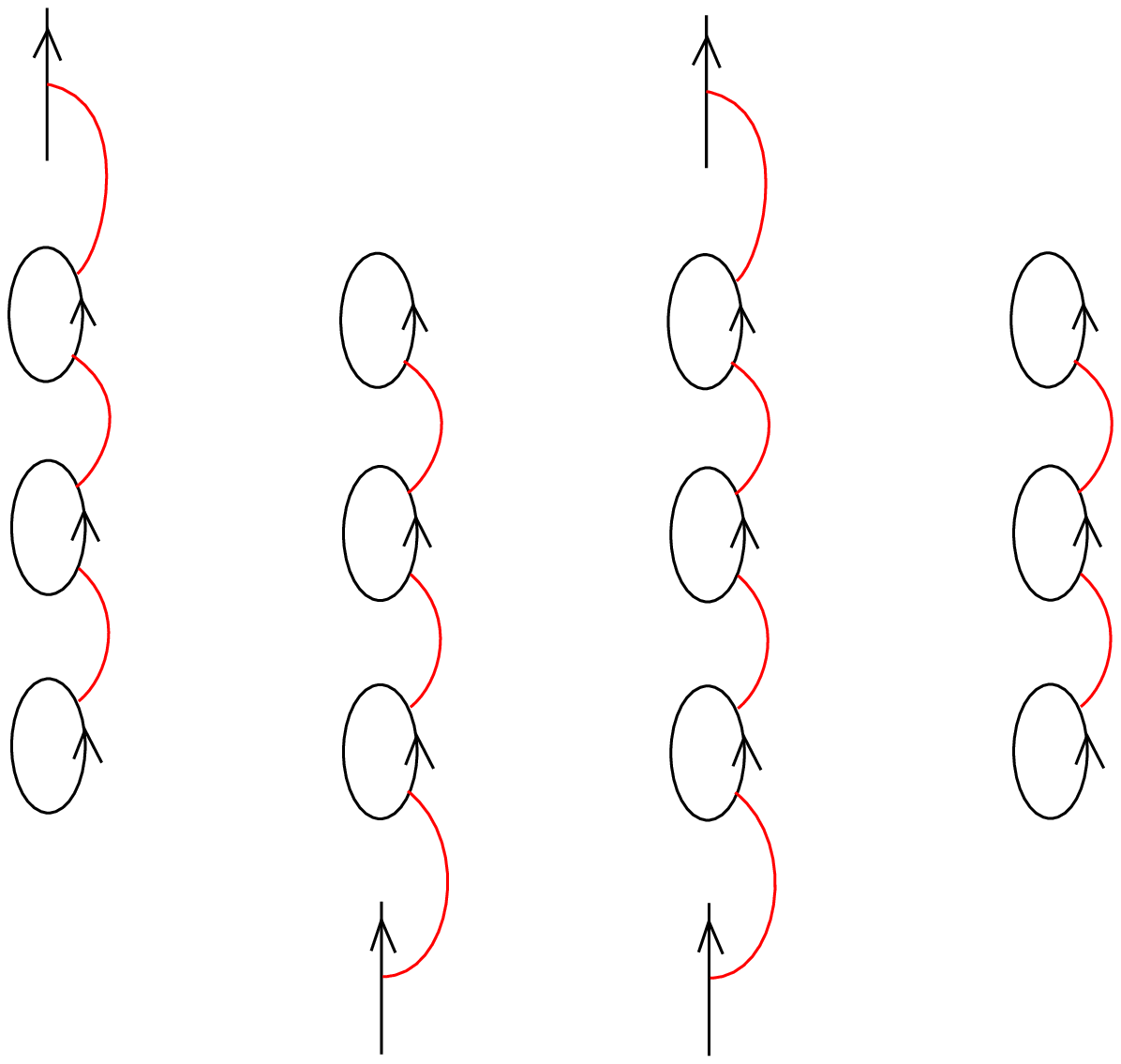}
  \caption{\small From left to right: arc diagrams $\Zc_l$, $\Zc_r$, $\Zc_{\full}$, and $\Zc'$ for $n=3$.}
\label{fig:FourArcDiagrams}
\end{minipage}
\begin{minipage}{.5\textwidth}
  \centering
  \includegraphics[scale=0.4]{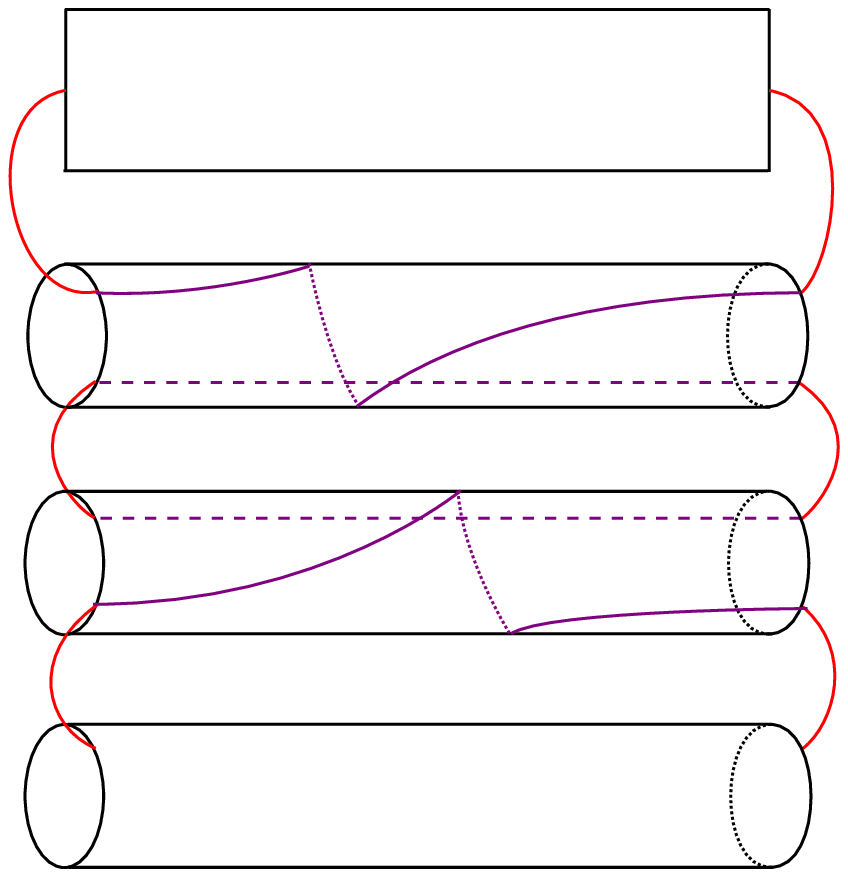}
  \caption{\small An element of the strands algebra $\A(\Zc_l)$.}
\label{fig:StrandsAlgElt}
\end{minipage}
\end{figure}

The constructions of Rouquier and the third author in \cite{ManionRouquier} equip $\A(\Zc)$ with differential bimodules $E$ and $E^{\vee}$.  In our specific setting, the differentials vanish, and the differential bimodules $E$ and $E^{\vee}$ are closely related the bimodules $\Eb''$ and $\Fb$ introduced in the previous subsection.

Representative elements of the bimodules $E$ and $E^{\vee}$ over $\A(\Zc_l)$ are shown in Figure~\ref{fig:StrandsEAndDual}. The left and right actions of elements of $\A(\Zc_l)$ on such elements are defined by concatenation as usual. The $j^{th}$ tensor powers of $E$ and $E^{\vee}$ in general admit similar descriptions using strands pictures in which $j$ strands exit to the top or bottom; in this case, there are no such valid pictures for $j \geq 2$ (see Figure~\ref{fig:ESquared}), so $E$ and $E^{\vee}$ each square to zero.

\begin{remark}
For more general $\Zc$ the bimodules over $\A(\Zc)$ from \cite{ManionRouquier} do not square to zero, but $m^{th}$ powers of the bimodules admit actions of the dg nilCoxeter algebra $\mathfrak{N}_m$ appearing in \cite{DM}, causing the homology of the second and all higher powers to vanish. The bimodules we consider in this page are an especially simple case of this construction.
\end{remark}

For the four arc diagrams in question, one can work with gradings that are simpler than in the general case (see
\cite[Section 6]{MMW2}), and view $E$ and $E^{\vee}$ as graded by $\Z\langle e_1, \ldots, e_n \rangle$ in the unique way such that the following proposition holds with respect to $\Z^{n}$ multi-gradings.

\begin{proposition}\label{prop:gl(11)bim}
There is a natural identification of $\A(\Zc_l,k) \otimes_{\B_l(n,k)} \Fb_k$ with $E^{\vee}$ as dg $(\A(\Zc_l,k), \B_l(n,k+1))$-bimodules, where the right action of $\B_l(n,k+1)$ on $E^{\vee}$ is restricted from the right action of $\A(\Zc_l,k+1)$ via the quasi-isomorphism $\B_l(n,k+1) \xrightarrow{\sim} \A(\Zc_l,k+1)$ from \cite{LP,MMW2}, and similarly for the right action of $\B_l(n,k)$ on $\A(\Zc_l,k)$.

Similarly, there is a natural identification of $\Eb''_l \otimes_{\B_l(n,k)} \A(\Zc_l,k)$ with $E$ as dg $(\B_l(n,k+1), \A(\Zc_l,k))$-bimodules. Similar statements hold for the bimodules over $\B_r(n,k)$ and $\A(\Zc_r,k)$, as well as for the bimodules over $\B(n,k)$ and $\A(\Zc_{\full},k)$.
\end{proposition}

\begin{proof}
As a left $\B_l(n,k)$-module, $\Fb_k$ is a direct sum of the indecomposable projective $\B_l(n,k)$ modules corresponding to elements $\x \in V_l(n,k)$ with $0 \notin \x$. Thus, $\A(\Zc_l,k) \otimes_{\B_l(n,k)} \Fb_k$ is a direct sum of the indecomposable projective $\A(\Zc_l,k)$-modules corresponding to the same elements $\x$. We identify elements of the summand for a given $\x$ with elements of $E^{\vee}$ as in Figure~\ref{fig:BimodIdentification}; this identification is an isomorphism of dg left $\A(\Zc_l,k)$-modules. To see that the right actions of $\B_l(n,k+1)$ agree, it suffices to check that the actions of the idempotents and of the multiplicative generators $R_i$, $L_i$, and $U_i$ of $\B_l(n,k+1)$ agree, which is immediate. The other statements are analogous.
\end{proof}

It follows that the induction equivalences from $D^c(\B_l(n,k))$ to $D^c(\A(\Zc_l, k))$ and from $D^c(\B_l(n,k+1))$ to $D^c(\A(\Zc_l,k+1))$ intertwine the functors $\Fb_k \otimes -$ and $E^{\vee} \otimes -$.

\begin{figure}
\centering
\begin{subfigure}[t]{.5\textwidth}
  \centering
  \includegraphics[width=.9\linewidth]{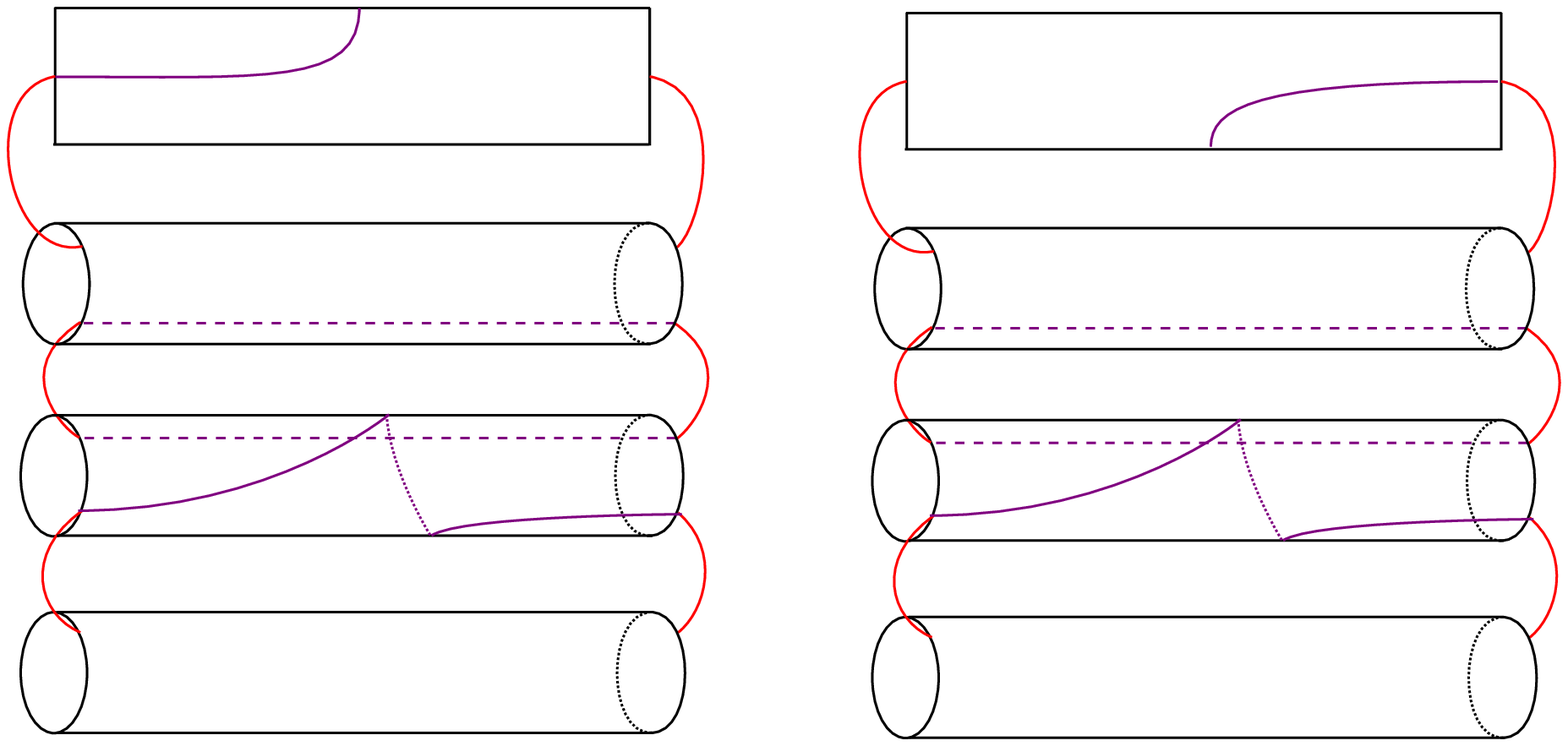}
\caption{Representative elements of $E$ (left) and $E^{\vee}$ (right)}
\label{fig:StrandsEAndDual}
\end{subfigure}\quad
\begin{subfigure}[t]{.45\textwidth}
  \centering
  \includegraphics[width=.45\linewidth]{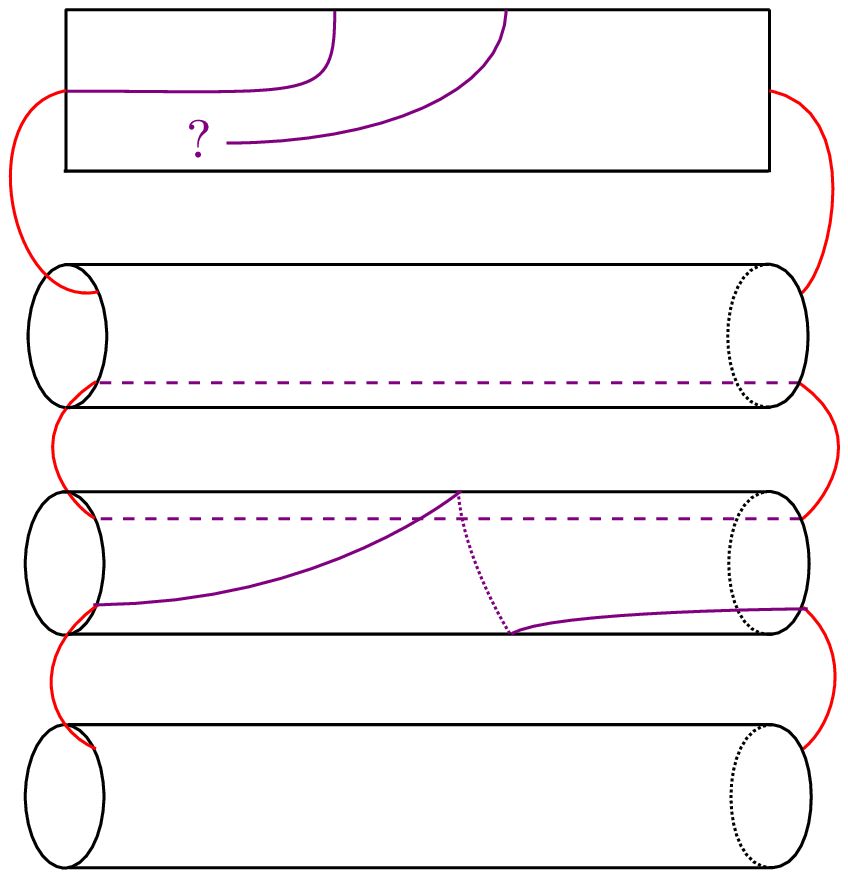}
\caption{$E^2 = 0$ because there are no valid strands pictures with two strands exiting the top of the distinguished rectangle.}\label{fig:ESquared}
\end{subfigure}
\caption{}
\end{figure}

\begin{figure}
\centering
\begin{minipage}{.5\textwidth}
\includegraphics[scale=0.38]{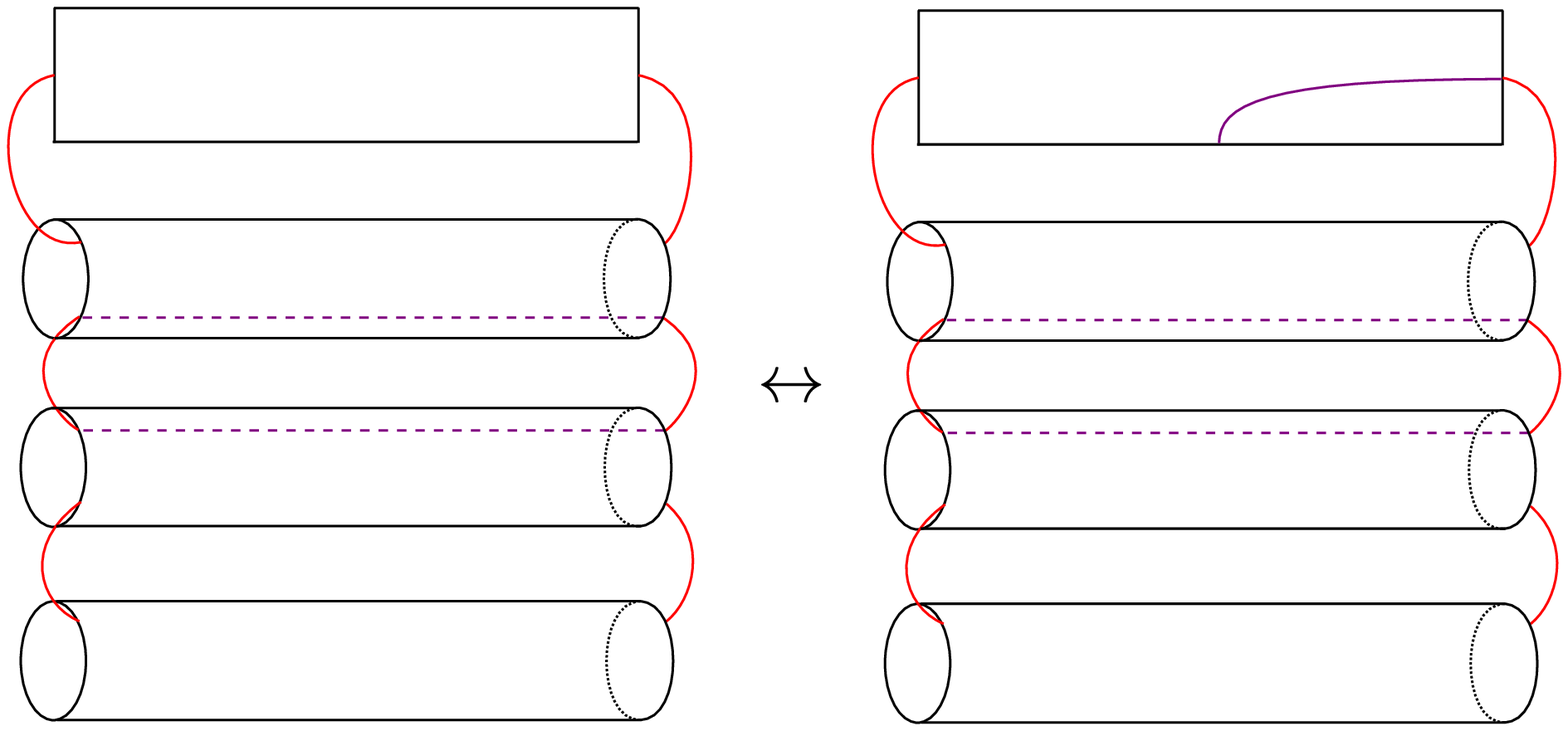}
\end{minipage}
\begin{minipage}{.45\textwidth}
\quad
\includegraphics[scale=0.38]{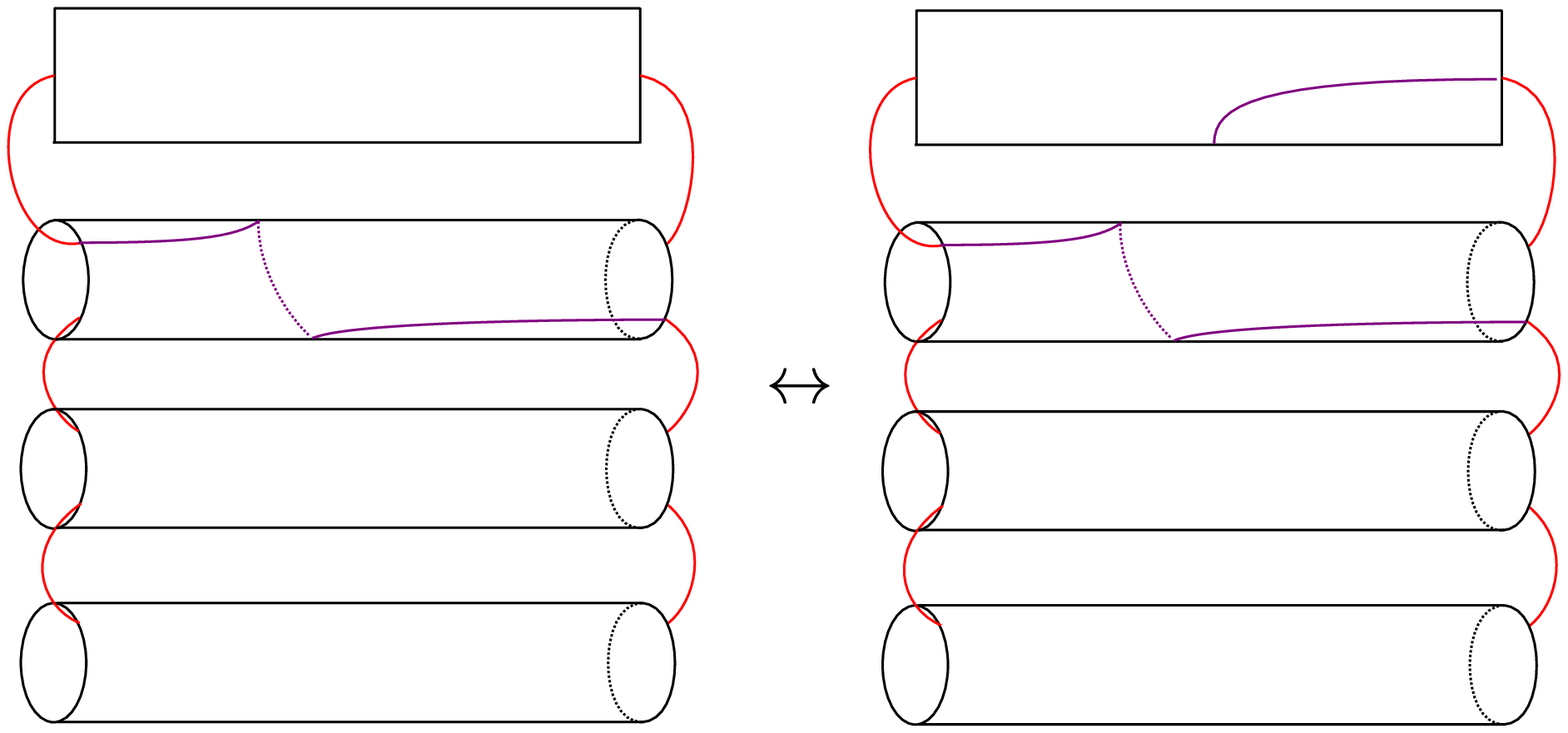}
\end{minipage}
\caption{Elements of $\A(\Zc_l,k) \otimes_{\B_l(n,k)} \Fb_k$ on the left and the corresponding elements of $E^{\vee}$ on the right. }
\label{fig:BimodIdentification}
\end{figure}

\subsection{Heegaard diagram interpretation}\label{sec:FactoringWithHDs}
The bimodules $E$ and $E^{\vee}$ from \cite{ManionRouquier}, which are defined using strands pictures like Douglas--Manolescu's, should admit an alternate description as bimodules associated to (generalized) bordered sutured Heegaard diagrams, reminiscent of the Heegaard diagrams for quantum-group bimodules in \cite[Figure 32]{EPV}. Since these generalized bordered sutured Heegaard diagrams are not so standard in the literature, we review the intended generalization of Zarev's definitions below, before discussing the examples of interest.

\subsubsection{Bordered sutured Heegaard diagrams}\label{sec:HDGeneralities}
Recall that by a generalized arc diagram, we mean an arc diagram as defined in \cite[Definition 2.1.1]{Zarev} except that some or all of the $Z_i$ may be oriented circles rather than oriented intervals (the ordering on the $Z_i$ is also typically irrelevant except for bookkeeping). We do not impose any degeneracy conditions a priori, although developing bordered sutured Floer homology in the full generality of these diagrams is expected to be quite difficult.

In the introduction to \cite{Zarev}, Zarev writes about the $DA$ bimodules $\widehat{BSDA}(Y, \Gamma)$ as if they are associated directly to morphisms in his decorated sutured category $\mathcal{SD}$, which are 3d cobordisms (unparametrized by Heegaard diagrams) between 2d sutured surfaces (parametrized by arc diagrams). This is a common white lie in Heegaard Floer homology; like the sutured Floer complexes they generalize, the bimodules $\widehat{BSDA}(Y, \Gamma)$ are chain-level objects and depend on a parametrization of $(Y,\Gamma)$ in terms of an appropriate Heegaard diagram, along with additional analytic choices (only the homotopy type of the bimodule is an invariant of $(Y,\Gamma)$).

In \cite[Chapter 4]{Zarev}, Zarev explains the Heegaard diagram choices necessary to define the one-sided bordered sutured modules $\widehat{BSD}(Y,\Gamma)$ and $\widehat{BSA}(Y,\Gamma)$ for bordered sutured manifolds $(Y,\Gamma)$ (these $(Y,\Gamma)$ can be viewed as morphisms to or from the empty set in $\mathcal{SD}$). He encodes these choices in what he calls a bordered sutured Heegaard diagram. In \cite[Chapter 8]{Zarev}, where bimodules are discussed, Heegaard diagrams are treated less formally, and it appears that no name is chosen for the type of Heegaard diagram used to represent decorated sutured cobordisms in \cite[Chapter 8.4]{Zarev}. Examining the paragraphs below \cite[Definition 8.4.2]{Zarev}, the relevant type of diagram for a decorated sutured cobordism $(Y,\Gamma)$ from a sutured surface parametrized by $\Zc_1$ to a sutured surface parametrized by $\Zc_2$ is a bordered sutured Heegaard diagram for $(Y,\Gamma)$ viewed as a bordered sutured manifold with boundary parametrized by $\Zc_1 \cup \Zc_2$. Thus, we will also refer to these diagrams for cobordisms as bordered sutured Heegaard diagrams.

We generalize Zarev's definition of bordered sutured Heegaard diagrams in \cite[Definition 4.1.1]{Zarev} by allowing $\Zc$ to be a generalized arc diagram (i.e. to have closed circles as well as closed intervals). We also do not impose homological linear independence in Zarev's terms (his formulation does not correctly extend to generalized diagrams). We treat cobordisms from the sutured surface of $\Zc_1$ to the sutured surface of $\Zc_2$ as described in the above paragraph.

\subsubsection{Ozsv{\'a}th--Szab{\'o}'s middle diagrams}
In \cite[Sections 10 and 11]{OSzHolo}, Ozsv{\'a}th--Szab{\'o} assign $DA$ bimodules to a type of Heegaard diagram that they call a middle diagram, which are defined in turn using the upper diagrams of \cite[Section 2.1]{OSzHolo}. One can think of their middle diagrams as certain bordered sutured Heegaard diagrams in which $\Zc_1$ and $\Zc_2$ are instances of the generalized arc diagram $\Zc'$ of Figure~\ref{fig:FourArcDiagrams} (compare with \cite[Figures 7--10]{OSzHolo}, which we would rotate 90 degrees clockwise in the plane to match our conventions).

Less obviously, it is appropriate to think of the extended middle diagrams appearing in \cite[Section 11.1]{OSzHolo} as bordered sutured Heegaard diagrams in which $\Zc_1$ and $\Zc_2$ are instances of the generalized arc diagram $\Zc_{\full}$ of Figure~\ref{fig:FourArcDiagrams}. Compare \cite[Figure 44]{OSzHolo} with Figure~\ref{fig:ExtendedMiddleDiag}, which shows how we would interpret extended middle diagrams (since Ozsv{\'a}th--Szab{\'o} require holomorphic curves to have zero multiplicity at the boundary circles $Z_0^{||}$ and $Z_1^{||}$, the difference between the versions of the diagram with and without corners is not visible by the holomorphic geometry input to these bimodules).  One can also consider half-extended middle diagrams having either $Z_0^{||}$ or $Z_1^{||}$ but not both, with a similar translation to the bordered sutured language.

\begin{figure}
\includegraphics[scale=0.8]{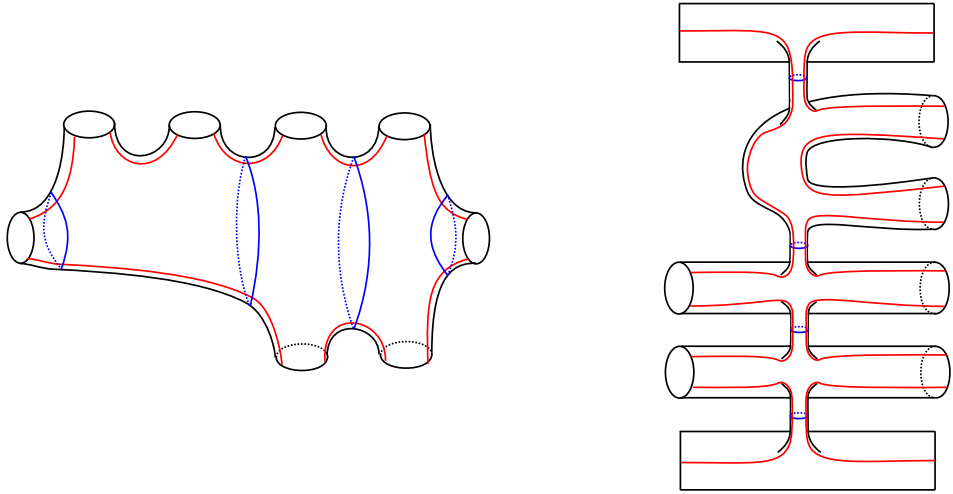}
\caption{Left: Extended middle diagram following Ozsv{\'a}th--Szab{\'o}. Right: the corresponding bordered sutured Heegaard diagram.}
\label{fig:ExtendedMiddleDiag}
\end{figure}

Below we will consider the $DA$ bimodule of a Heegaard diagram related to the one shown on the right of Figure~\ref{fig:HeegaardDiagVariant} by a disk decomposition along the horizontal boundary (this type of decomposition is invisible when defining $DA$ bimodules from the diagrams). We can almost view the diagram on the right of Figure~\ref{fig:HeegaardDiagVariant} as being associated to a half-extended middle diagram in the sense of Ozsv{\'a}th--Szab{\'o}, as on the left of Figure~\ref{fig:HeegaardDiagVariant}. However, this left diagram is not a valid half-extended middle diagram since it does not arise from a middle diagram in the correct way. Thus, while the diagrams are very simple and we can count all the disks that should contribute to their bimodules following the same heuristics that apply in \cite{OSzHolo}, we cannot directly cite \cite{OSzHolo} to turn these disk counts into a rigorous theorem that the bimodule we define below is the $DA$ bimodule associated to our Heegaard diagram by some more general definition.

\begin{figure}
\centering
\begin{minipage}{.6\textwidth}
\centering
\includegraphics[scale=0.9]{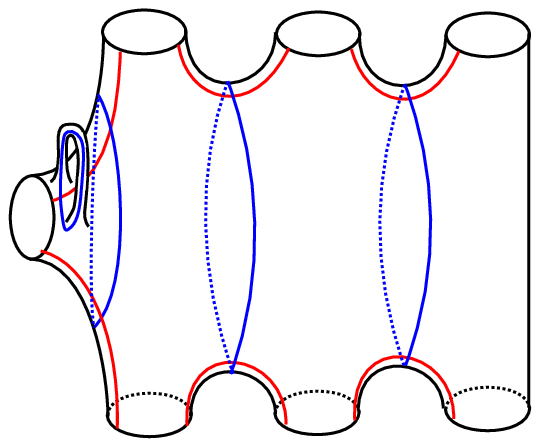}
\end{minipage}%
\begin{minipage}{.4\textwidth}
\includegraphics[scale=0.8]{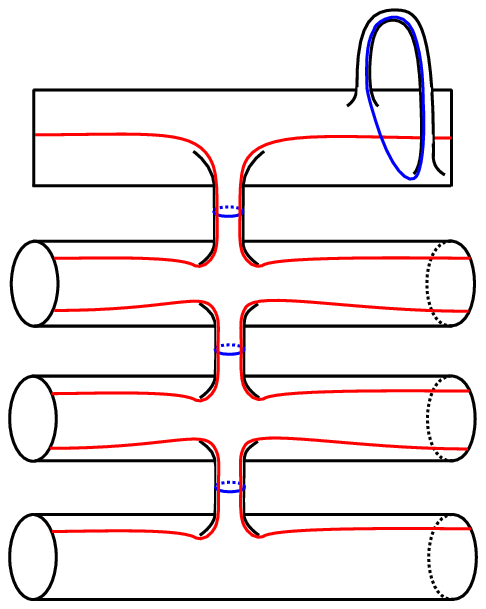}
\end{minipage}
\caption{Variant of the Heegaard diagram considered below, from Ozsv{\'a}th--Szab{\'o}'s perspective (left) and the bordered sutured perspective (right)}
\label{fig:HeegaardDiagVariant}
\end{figure}

\subsubsection{The case of interest}\label{sec:HDsCaseOfInterest}

When discussing bimodules from Heegaard diagrams in this section, we note that based on the most literal extensions of the bordered sutured theory to this case, the Heegaard diagrams would actually produce bimodules over the dg strands versions of the algebras $\A(\Zc_l)$ \cite{LP,MMW2}, which are formal with homology $\B_l(n,k)$. Ozsv{\'a}th--Szab{\'o}'s methods skip the dg step and interpret holomorphic disk counts directly in terms of algebras like $\B_l(n,k)$; we follow their approach here.

Figure~\ref{fig:HeegaardDiag} shows the relevant bordered sutured Heegaard diagram for the bimodule $E^{\vee}$ over $\A(\Zc_l)$ (here we prefer Figure~\ref{fig:HeegaardDiag} for the factorization we consider below, although the diagram of Figure~\ref{fig:HeegaardDiagVariant} has other advantages).  We expect $E^{\vee}$ should agree with the type $DA$ bimodule that would be associated to this Heegaard diagram in bordered sutured Floer homology.  In part because bordered sutured Floer homology has not been defined when arc diagrams have circles instead of just intervals, we do not prove this here (as discussed above). However, in the special case of the bimodules $E$ and $E^\vee$ from Proposition \ref{prop:gl(11)bim}, it is straightforward enough to verify that the disk counts that would typically be used to associate a type $DA$ bimodule to this type of Heegaard diagram make sense, giving a Heegaard-diagram interpretation of these bimodules.

\begin{figure}
\includegraphics[scale=0.45]{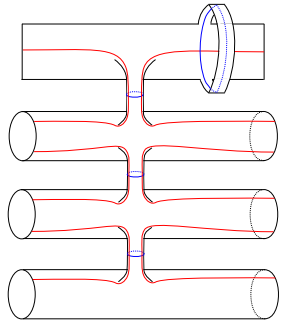}
\caption{Heegaard diagram whose $DA$ bimodule recovers the bimodule $E^{\vee}$ over $\A(\Zc_l)$.}
\label{fig:HeegaardDiag}
\end{figure}

The Heegaard diagram of Figure~\ref{fig:HeegaardDiag} can be factored (up to stabilizations) into two pieces as in Figure~\ref{fig:HeegaardFactorization}. Algebraically, this should mean that $E^{\vee}$ is homotopy equivalent to a box tensor product (as in bordered Floer homology) of two $DA$ bimodules corresponding to the two pieces; in fact, we will have an isomorphism.

\begin{figure}
\includegraphics[scale=0.45]{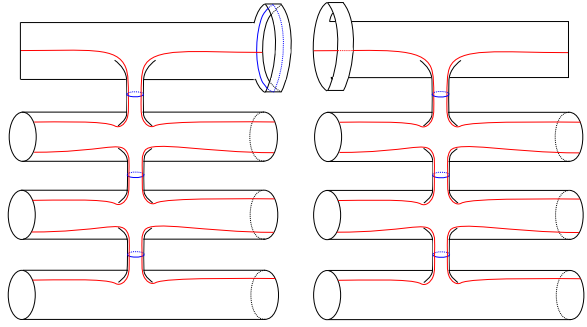}
\caption{Decomposing the Heegaard diagram of Figure~\ref{fig:HeegaardDiag}.}
\label{fig:HeegaardFactorization}
\end{figure}

In the diagram on the right of Figure~\ref{fig:HeegaardFactorization}, there is a Heegaard Floer generator (set of intersection points) for each such generator of the identity Heegaard diagram for the incoming arc diagram. See Figure~\ref{fig:IdentityDiags} for the relevant identity diagrams, which correspond to identity middle diagrams and half-extended middle diagrams in Ozsv{\'a}th--Szab{\'o}'s language and thus have well-defined $DA$ bimodules which are identity bimodules (at least as $DD$ bimodules) by \cite[Proposition 13.2]{OSzHolo}.

\begin{figure}
\includegraphics[scale=0.5]{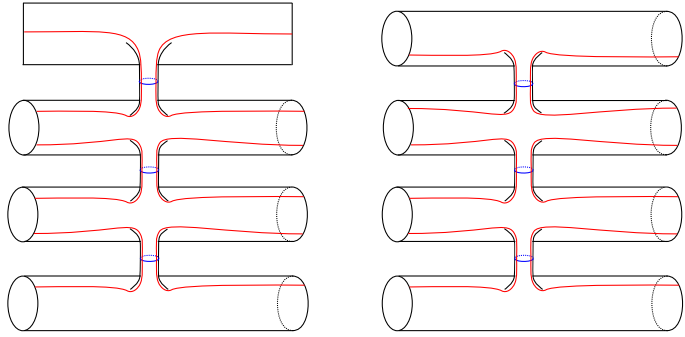}
\caption{Identity Heegaard diagrams for the arc diagrams $\Zc_l$ (left) and $\Zc'$ (right).}
\label{fig:IdentityDiags}
\end{figure}

As with the identity diagram, for the diagram in question there should be no $\delta^1_i$ actions for $i > 2$ (i.e. there should be no higher $\A_{\infty}$ terms in the right action), and there should be no $\delta^1_1$ actions (i.e. no differential). Thus, the $DA$ bimodule of this diagram should be an ordinary bimodule, projective on the left. If the algebra acting on the right is $\B_l(n,k+1)$, then the algebra acting on the left is $\B'(n+1,k+1)$, which naturally contains $\B_l(n,k+1)$ as a subalgebra. For each basic idempotent of $\B_l(n,k+1)$, the $DA$ bimodule of the diagram (as a left module) would have an indecomposable projective summand corresponding to the image of the idempotent in $\B'(n+1,k+1)$. Thus, as a left module, the $DA$ bimodule of the diagram would agree with the induction on the left of the identity bimodule over $\B_l(n,k+1)$ by the inclusion map from $\B_l(n,k+1)$ to $\B'(n+1,k+1)$. The right action of $\B_l(n,k+1)$ on the $DA$ bimodule of the diagram should involve the same curve counts as in the identity Heegaard diagram over the input arc diagram, so we should be able to identify the $DA$ bimodule of the diagram with the induced identity bimodule as bimodules over $(\B'(n+1,k+1), \B_l(n,k+1))$.

Now, in the diagram on the left of Figure~\ref{fig:HeegaardFactorization}, there is a Heegaard Floer generator for each such generator of the identity Heegaard diagram for the incoming arc diagram such that the top arc is unoccupied on the input side of the generator. As before, there should be no $\delta^1_1$ actions or $\delta^1_{\geq 3}$ actions. The input algebra is $\B'(n+1,k+1)$ and the output algebra is $\B_l(n,k)$. For a Heegaard Floer generator of the diagram corresponding to $\x \in V'(n+1,k+1)$ with $1 \in \x$, let $\x' := \{i - 1: i > 1 \in \x\}$; the summand of the $DA$ bimodule coming from $\x$ would be the indecomposable projective $\B_l(n,k)$-module corresponding to $\x'$. For the right action, elements of the form $R_i$ or $L_i$ for $2 \leq i \leq n-1$, as well as $U_i$ for $2 \leq i \leq n$, should act on the right by outputting generators of the same form on the left (assuming the left and right idempotents of these elements contain $1$; otherwise the elements act as zero). Interestingly, elements $U_1$ at $\x \in V'(n+1,k+1)$ such that $1 \in \x$ act on the right by outputting $1$ on the left; see Figure~\ref{fig:InterestingDomain} for the domain of the relevant holomorphic disk. This domain is basically identical to ones appearing in Ozsv{\'a}th--Szab{\'o}'s theory, whose holomorphic geometry and resulting algebra is known from \cite{OSzHolo}, and we treat it in the same way here. Briefly, this type of domain implies that a $U$ variable acts on the right with output determined by which parts of the output boundary are covered by the domain. In our case, the domain stays away from the output boundary, so the algebraic output is $1$.

\begin{figure}
\includegraphics[scale=0.5]{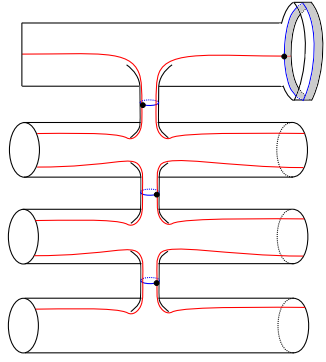}
\caption{Domain for the right action of $U_1$.}
\label{fig:InterestingDomain}
\end{figure}

\begin{remark}
The elements $R_i$, $L_i$, and $U_i$ still generate the truncated algebra $\B'(n+1,k+1)$, with relations that are identical except when $k=n-1$ (see \cite{MMW1}). Using these relations, one could check that the above $DA$ bimodule is well-defined; alternatively, well-definedness follows from the identifications of Section~\ref{sec:HFBimodsVsDelRest}.
\end{remark}

\subsection{Relationship with deletion and restriction}\label{sec:HFBimodsVsDelRest}

The Heegaard diagram factorization above also arises naturally from the point of view of polarized arrangements.  Under the identification of $\B_l(n,k)$ with $\tilde{B}(\cal{V})$ for left cyclic $\cal{V}$ from Theorem~\ref{thm:LeftCyclicIsom}, we can identify $\Fb_k$ with a tensor product of deletion and signed-restriction bimodules as defined in Section~\ref{sec:GeneralDeletionRestriction}.

In more detail, let $\V = (V,\eta,\xi)$ be a left cyclic arrangement of $n$ hyperplanes in $(k+1)$-space. Choose a left cyclic arrangement $\hat{\V} = (\hat{V}, \hat{\eta}, \hat{\xi})$ of $n+1$ hyperplanes in $(k+1)$-space whose first deletion is $\V$, and let $\V' = (V',\eta',\xi')$ be the first restriction of $\hat{\V}$. Since we are looking at the restriction to the first hyperplane, the sign change arising from replacing the restriction by the signed restriction is a global multiplication by $-1$.  Thus, by Lemma~\ref{lem:RestrictedCyclicIsCyclic}, the signed restriction $\V'' := (V', -\eta', -\xi')$ is a left cyclic arrangement of $n$ hyperplanes in $k$-space.

From Section~\ref{sec:GeneralDeletionRestriction}, we have bimodules $\Del_1 := {\rm del}_1^{\Bt}(\cal{V},+)$ and $\Rest'_1 := ({\rm rest}')^1_{\Bt}(\cal{V},+)$ over $(\tilde{B}(\hat{\V}), \tilde{B}(\V))$ and $(\tilde{B}(\V'),\tilde{B}^+(\hat{\V}))$ respectively. Since $\Del_1$ is defined using a homomorphism from $\tilde{B}(\V)$ to $\tilde{B}(\hat{\V})$ whose image is contained in the idempotent-truncated subalgebra $\tilde{B}^+(\hat{\V})$ of $\tilde{B}(\hat{\V})$, we can view $\Del_1$ as an induction on the left of a bimodule $\Del'_1$ over $(\tilde{B}^+(\hat{\V}), \tilde{B}(\V))$. Twisting $\Rest'_1$ by the algebra isomorphism arising from the sign-change relationship between $\V'$ and $\V''$, we get a bimodule $\Rest''_1$ over $(\tilde{B}(\V''), \tilde{B}^+(\hat{\V}))$.

Now, since $\V$, $\hat{\V}$, and $\V''$ are left cyclic, we have identifications
\begin{itemize}
\item $\tilde{B}(\V) \cong \B_l(n,k+1)$,
\item $\tilde{B}^+(\hat{V}) \cong \B'(n+1,k+1)$,
\item $\tilde{B}(\V'') \cong \B_l(n,k)$.
\end{itemize}
We can thus view $\Del'_1$ as a bimodule over $(\B'_l(n+1,k+1),\B_l(n,k+1))$ and $\Rest''_1$ as a bimodule over $(\B_l(n,k), \B'_l(n+1,k+1))$.

\begin{proposition}
Under the above identifications, $\Del'_1$ and $\Rest''_1$ agree with the $DA$ bimodules proposed in Section~\ref{sec:HDsCaseOfInterest} for the left and right Heegaard diagrams in Figure~\ref{fig:HeegaardFactorization} respectively.
\end{proposition}

For $\Del'_1$, this is clear from Section~\ref{sec:HDsCaseOfInterest}. For $\Rest''_1$, note that for $\x \in V'(n+1,k+1)$ with $1 \notin \x$, the sign sequence $\alpha$ corresponding to $\x$ starts with $++$. Thus, $\tilde{\alpha}$ starts with $--$, so
\[
\var_l(\tilde{\alpha}) = 1 + \var(\tilde{\alpha}) = 1 + \var(\alpha) = k+2 > k,
\]
and $e_{\tilde{\alpha}}$ is thereby zero in $\tilde{B}(\V'')$ (compare with Remark~\ref{rem:ExtendedBDef}). Thus, $\Rest''_1$ agrees with the $DA$ bimodule proposed for the right Heegaard diagram of Figure~\ref{fig:HeegaardFactorization} as left $\B_l(n,k)$-modules. By examining the right action of $R_i$, $L_i$, and $U_i$ elements, one concludes that the right actions also agree.

We can interpret $\Fb_k$ similarly, using the Heegaard diagram from Figure~\ref{fig:HeegaardDiag}; since we did not describe the $DA$ bimodule of that diagram in explicit terms above, we will just take this interpretation as the definition of the $DA$ bimodule of that Heegaard diagram.

\begin{corollary}
Under the above identifications, there is an isomorphism of $(\B_l(n,k), \B_l(n,k+1))$ bimodules
\[
\Fb_k \cong \Rest''_1 \otimes_{\B'(n+1,k+1)} \Del'_1.
\]
\end{corollary}

We state this as a corollary since the gluing theorem in bordered Floer homology (if it holds in our generalized setting) would identify the $DA$ bimodule of the diagram from Figure~\ref{fig:HeegaardDiag} with the box tensor product $\boxtimes$ (as in e.g. \cite{LOTBimodules}) of the $DA$ bimodules of the two diagrams in Figure~\ref{fig:HeegaardFactorization}, and since in this case the box tensor product agrees with the ordinary tensor product. However, the corollary makes no mention of the $DA$ bimodules associated to diagrams, and it can be verified directly without issue. Equally well, one can verify the following similar result.

\begin{proposition}
Under the above identifications, there is an isomorphism of $(\B_l(n,k), \B_l(n,k+1))$ bimodules
\[
\Fb_k \cong \Rest_1 \otimes_{\B_l(n+1,k+1)} \Del_1,
\]
where we implicitly postcompose $\Rest_1$ with a sign-change automorphism similar to those used above.
\end{proposition}

\bibliographystyle{alpha}
\bibliography{bib_clean}

\end{document}